\documentclass[fleqn]{elsarticle}

\usepackage{amsmath,amssymb,amsthm,mathrsfs,graphicx,subfigure,float,caption,epstopdf,tabularx,bbm,color,bm,amsfonts,epic,booktabs,multirow,epsfig}
\usepackage[colorlinks]{hyperref}

\usepackage{caption}
\allowdisplaybreaks[4] 
\usepackage[toc,page]{appendix} 
\biboptions{numbers,sort&compress} 

\usepackage{geometry} 
\numberwithin{equation}{section}
\newtheorem{theorem}{Theorem}[section]
\newtheorem{lemma}{Lemma}[section]

\newtheorem{remark}{Remark}[section]

\newtheorem{definition}[theorem]{Definition}

\graphicspath{{FIGS/}}

\newtheorem*{SAV-CN}{SAV-CN}
\newtheorem*{ESAV-CN}{ESAV-CN}
\newtheorem*{ESAV-Gauss}{ESAV-Gauss}
\newtheorem*{ESAV-Gauss-PC}{ESAV-Gauss-PC}

\begin{document}
\begin{frontmatter}
\title{Arbitrary high-order linearly implicit energy-preserving algorithms for Hamiltonian PDEs\vspace{-1mm}}
\author{Yonghui Bo, Yushun Wang and Wenjun Cai$^*$}
\address{Jiangsu Key Laboratory for NSLSCS,\\ Jiangsu Collaborative Innovation Center of Biomedial Functional Materials,\\ School of Mathematical Sciences, Nanjing Normal University, Nanjing 210023, China\vspace{-9mm}}
\begin{abstract}

In this paper, we present a novel strategy to systematically construct linearly implicit energy-preserving schemes with arbitrary order of accuracy for Hamiltonian PDEs. Such novel strategy is based on the newly developed exponential scalar variable (ESAV) approach that can remove the bounded-from-blew restriction of nonlinear terms in the Hamiltonian functional and provides a totally explicit discretization of the auxiliary variable without computing extra inner products, which make it more effective and applicable than the traditional scalar auxiliary variable (SAV) approach. To achieve arbitrary high-order accuracy and energy preservation, we utilize the symplectic Runge-Kutta method for both solution variables and the auxiliary variable, where the values of internal stages in nonlinear terms are explicitly derived via an extrapolation from numerical solutions already obtained in the preceding calculation. A prediction-correction strategy is proposed to further improve the accuracy. Fourier pseudo-spectral method is then employed to obtain fully discrete schemes. Compared with the SAV schemes, the solution variables and the auxiliary variable in these ESAV schemes are now decoupled. Moreover, when the linear terms are of constant coefficients, the solution variables can be explicitly solved by using the fast Fourier transform. Numerical experiments are carried out for three Hamiltonian PDEs to demonstrate the efficiency and conservation of the ESAV schemes.	
\end{abstract}

\end{frontmatter}
\begin{figure}[b]
\small \baselineskip=10pt
\rule[2mm]{1.8cm}{0.2mm} \par
$^{*}$Corresponding author.\\
E-mail address: caiwenjun@njnu.edu.cn (W. Cai).
\end{figure}
\vspace{-3mm}
\section{Introduction}
In this paper, we consider Hamiltonian PDEs of the form
\begin{equation}\label{eq-1-1}
z_t=\mathcal{D}\frac{\delta \mathcal{H}}{\delta z},
\end{equation} 
where $\mathcal{D}$ is a skew-adjoint operator and independent of the solution variable $z(x,t)$ and its derivatives, $(x,t) \in \Omega\times[0,T]$ and $\Omega \in \mathcal{R}^d$. Many classical conservative systems can be recast into the above Hamiltonian system, such as the nonlinear Schr\"odinger (NLS) equation, the sine-Gordon (SG) equation, the Korteweg-de Vries (KdV) equation and so on \cite{leimkuhler-04-SHD}. One of the most fundamental properties of  Hamiltonian PDEs \eqref{eq-1-1} is the conservation of energy, i.e., the energy functional $\mathcal{H}$ is constant along the continuous flow. Nowadays, numerical schemes that can preserve a discrete counterpart of the energy are more preferable than other non-conservative schemes, especially in the long-term simulations.

Fully implicit schemes play a dominant role in the early development of energy-preserving methods for Hamiltonian PDEs \eqref{eq-1-1}. When the energy functional has a quadratic form, any symplectic Runge-Kutta (RK) method can achieve a natural energy conservation \cite{cooper-87-RK-IMA,hairer-06-GNI-ODE}, with arbitrary high-order accuracy by increasing the internal RK stages. For energies of general forms, the discrete gradient methods \cite{harten-83-DG-SIAM-REV,itoh-88-VDQ-JCP,gonzalez-96-DH-JNS,mcLachlan-99-DG-PTRSLA} provide a novel framework to construct energy-preserving schemes, among which the most popular one is the averaged vector field (AVF) method \cite{quispel-08-AVF-JPAMT} as well as a series of subsequent extensions \cite{wu-13-oscil-JCP,cai-18-PAVF-JCP}. Such ideas are later generalized to develop the discrete variational derivative (DVD) methods \cite{furihata-99-DVD-JCP,matsuo-01-DC-JCP,furihata-11-DVD-CHCRC}. As the increasing demand of high-precision simulations, high-order energy-preserving methods emerge in recently years, such as high-order AVF methods \cite{hairer-10-coll-JNAIAM,cohen-11-lin-poisson-BIT,Li-16-AVF-JCM}, Hamiltonian boundary value methods (HBVMs) \cite{luigi-10-HBVM-JNAIAM,luigi-12-Poisson-JCAM,luigi-19-HBVM-KDV-JCAM}, time finite element methods \cite{betsch-00-time-FEM-JCP,tang-12-time-FEM-AMC} and so on. Nevertheless, all the methods mentioned above are fully implicit for general Hamiltonian PDEs and nonlinear iterations are required, which make them time consuming and less attractive in practical computations.

One alternative way to incorporate the energy-preserving property and computational efficiency is to construct linearly implicit schemes, which only involve a linear system to be solved at each time step. For polynomial energies, the multiple DVD method \cite{matsuo-01-DC-JCP} and a more general framework based on the polarization technique are proposed \cite{dahlby-11-general-IP-SIAMJSC}. However, there are no methods can conserve energy of arbitrary forms until the development of the energy quadratization approach, named the invariant energy quadratization (IEQ) \cite{yang-16-IEQ-JCP,yang-17-IEQ-JCP,gong-18-IEQ-binary-fluid-SIAMJSC} and the scalar auxiliary variable (SAV) \cite{shen-18-SAV-JCP,shen-19-SIAM-REV,qiao-19-PR-SAV-CICP} approach respectively. Although these two approaches are first proposed for gradient flow models, they have been successfully applied to various conservative systems, including Hamiltonian systems \eqref{eq-1-1} \cite{cai-19-SG-NB-JCP,cai-20-linear-MS-JCP,jiang-19-SG-IEQ-JSC,jiang-20-CH-SAV-JSC}. The basic idea of such approaches is to reformulate the original energy into a quadratic form by introducing a new variable. Further utilizing the extrapolation technique for nonlinear terms, linearly implicit energy-preserving schemes can be systematically constructed. The most popular one is the Crank-Nicolson scheme with extrapolation that has second-order accuracy. High-order linearly implicit schemes can be constructed based on the symplectic RK methods and high-order extrapolations \cite{akrivis-19-RK-SAV-SIAMJSC,gong-20-high-stable-JCP}. Although the resulting schemes under the framework of the IEQ and SAV approaches are both linearly implicit, their major difference occurs in the coefficient matrix of the linear system. The former is solution-dependent while the latter is constant so that fast solvers can be employed, which makes the SAV approach more efficient. 

Nevertheless, in the implementation of the SAV schemes, the calculation of solution variables and the auxiliary variable can not be decoupled. Moreover, extra inner products must be applied previously to obtain the solution variables, which would become more complicated for high-order SAV schemes
\cite{akrivis-19-RK-SAV-SIAMJSC}. To overcome theses shortcomings, the exponential scalar auxiliary variable (ESAV) approach \cite{liu-20-ESAV-SIAMJSC} is proposed very recently for phase field models. The computation of the solution variables and the auxiliary variable is totally decoupled and can be explicitly solved step-by-step if fast Fourier transform \cite{gong-17-FP-NLS-JCP,shen-11-spectral} is utilized for the linear system. In this paper, we extend the idea of the ESAV approach to conservative Hamiltonian systems \eqref{eq-1-1} and generalize it to construct linearly implicit energy-preserving schemes of arbitrary order of accuracy. We first reform the original system into an equivalent form by the ESAV approach, where an exponential auxiliary variable is introduced. Different from the SAV schemes, such reformation allows the approximation of entire nonlinear terms by extrapolation and still mains the energy conservation, whereas the auxiliary variable in SAV schemes has to be discretized implicitly. A second-order Crank-Nicolson scheme with extrapolation is given to illustrate the strategy of discretization, as well as the explicit calculation of both solution variables and the auxiliary variable. Subsequently, we generalize the scheme to arbitrary high order by the symplectic RK method and an extrapolation from numerical solutions already obtained from the preceding calculation. A prediction-correction strategy different from that in \cite{gong-20-high-stable-JCP} is proposed to further improve the accuracy. Rigorous proofs of the discrete energy conservation law are presented. Taking the NLS equaion, the SG equation and the KdV equation for examples, we numerically test the resulting ESAV schemes in accuracy, energy preservation. Comparisons with the classical SAV schemes are also carried out.

The rest of this paper is organized as follows. In Section \ref{Sec-2}, we first separate the Hamiltonian PDEs \eqref{eq-1-1} into linear and nonlinear terms and apply the standard SAV approach and the ESAV approach to obtain two kinds of equivalent forms. Then the Crank-Nicolson scheme equipped with extrapolation is used to construct second-order linearly implicit schemes, where the energy conservation properties and the detailed implementations are given, respectively. In Section \ref{Sec-3}, based on the symplectic RK method we first propose high-order fully implicit energy-preserving schemes under the framework of the ESAV approach. An extrapolation technique is thereafter employed for nonlinear terms to achieve linearly implicit schemes. A prediction-correction strategy is reported subsequently to further improve the accuracy. Section \ref{Sec-4} briefly reviews the pseudo-spectral method for the spatial discretization. Ample numerical examples are provided in Section \ref{Sec-5} to test their effectiveness. Finally, concluding remarks are drawn in Section \ref{Sec-6}.

\section{Scalar auxiliary variable approaches}\label{Sec-2}
To apply the scalar auxiliary variable approaches, we first separate the Hamiltonian functional $\mathcal{H}$ into linear and nonlinear terms, named $\mathcal{H}_1$ and $\mathcal{H}_2$, respectively, and subsequently the Hamiltonian system \eqref{eq-1-1} becomes
\begin{equation}\label{eq-2-1}
\begin{aligned}
z_t&=\mathcal{D}\mu,\\
\mu&=\mathcal{L}z+\mathcal{N}^\prime(z),
\end{aligned}
\end{equation} 
where $\delta\mathcal{H}_1/\delta z=\mathcal{L}z$, $\mathcal{L}$ is a symmetric non-negative linear operator, $\delta\mathcal{H}_2/\delta z=\mathcal{N}^\prime(z)$, $\mathcal{N}(z)$ is the energy density. Taking the inner product of \eqref{eq-2-1} with $\mu$ and $-z_t$ respectively and adding them together, we obtain the energy conservation law
\begin{equation}\label{eq-2-2}
\frac{d}{dt}\mathcal{H}\big(z\big)=0,\quad \mathcal{H}(z)=\frac{1}{2}\big(z,\mathcal{L}z\big)+\big(\mathcal{N}(z),1\big)=\mathcal{H}_1(z)+\mathcal{H}_2(z),
\end{equation} 
where $(\cdot,\cdot)$ denotes the normal $L^2$ inner product on the spatial region $\Omega$.

\subsection{Standard scalar auxiliary variable approach}
The SAV method is widely used by transforming $\mathcal{H}_2(z)$ into a simple quadratic form, which makes the nonlinear term $\mathcal{N}'(z)$ much easier to handle. We introduce  a scalar auxiliary variable $w(t)=\sqrt{\mathcal{H}_2(z)+C_0}$ with the assumption that $\mathcal{H}_2(z)$ is bounded from below \cite{shen-18-SAV-JCP}. This means that the constant $C_0$ enables $\mathcal{H}_2(z)+C_0$ greater than zero. Then, the system \eqref{eq-2-1} can be rewritten in the equivalent form
\begin{align}\label{eq-2-3}
\aligned
&z_t=\mathcal{D}\mu,\\
&\mu~=\mathcal{L}z+A(z)w,\quad A(z)=\frac{\mathcal{N}'(z)}{\sqrt{\mathcal{H}_2(z)+C_0}},\\
&w_t=\frac{1}{2}\big(A(z),z_t\big).
\endaligned
\end{align}
The above system satisfies the modified energy conservation law
\begin{equation}\label{eq-2-4}
\frac{d}{dt}\bar{\mathcal{H}}(t)=0,\quad\bar{\mathcal{H}}(t)=\frac{1}{2}\big(z,\mathcal{L}z\big)+w^2
\end{equation} 
by taking the inner product of \eqref{eq-2-3} with $\mu$, $z_t$ and $2w$, respectively, and summing them together.

Utilizing the Crank-Nicolson method in time coupled with an explicit extrapolation gives a second-order linearly implicit energy-preserving method
\begin{align}\label{eq-2-5}
\aligned
&\frac{z^{n+1}-z^n}{\tau}=\mathcal{D}\mu^{n+\frac{1}{2}},\\
&\mu^{n+\frac{1}{2}}=\mathcal{L}z^{n+\frac{1}{2}}+A(\bar{z}^{n+\frac{1}{2}})w^{n+\frac{1}{2}},\quad A(\bar{z}^{n+\frac{1}{2}})=\frac{\mathcal{N}'(\bar{z}^{n+\frac{1}{2}})}{\sqrt{\mathcal{H}_2(\bar{z}^{n+\frac{1}{2}})+C_0}}\triangleq A^n,\\
&w^{n+1}-w^n=\frac{1}{2}\big(A(\bar{z}^{n+\frac{1}{2}}),z^{n+1}-z^n\big)
\endaligned
\end{align}
with the time step $\tau$, $z^{n+\frac{1}{2}}=\frac{z^n+z^{n+1}}{2}$ and $\bar{z}^{n+\frac{1}{2}}=\frac{3z^n-z^{n-1}}{2}$. We name the scheme \eqref{eq-2-5} \textbf{SAV-CN}. In analogy to the continuous case, taking the inner product of \eqref{eq-2-5} with $\mu^{n+\frac{1}{2}}$, $\frac{z^{n+1}-z^n}{\tau}$ and $w^{n+1}+w^n$, the semi-discrete energy conservation law can be derived as
\begin{equation}\label{eq-2-6}
\bar{\mathcal{H}}^{n+1}=\bar{\mathcal{H}}^n,\quad \bar{\mathcal{H}}^n=\frac{1}{2}\big(z^n,\mathcal{L}z^n\big)+(w^n)^2.
\end{equation} 
In order to convey the computational aspects of \eqref{eq-2-5}, we  eliminate the variables $\mu^{n+\frac{1}{2}}$, $\omega^{n+1}$ and obtain
\begin{equation}\label{eq-2-7}
z^{n+1}=\left(I-\frac{1}{2}\tau\mathcal{D}\mathcal{L}\right)^{-1}C^n+\frac{1}{4}\tau\left(I-\frac{1}{2}\tau\mathcal{D}\mathcal{L}\right)^{-1}\mathcal{D}A^n\left(A^n,z^{n+1}\right),
\end{equation} 
where $C^n=\left(I+\frac{1}{2}\tau\mathcal{D}\mathcal{L}\right)z^n+\frac{1}{4}\tau\mathcal{D}A^n\big[4w^n-(A^n,z^n)\big]$ and $I$ is an identity matrix with appropriate dimensions. Taking the inner product of both sides of \eqref{eq-2-7} with $A^n$ leads to
\begin{equation}\label{eq-2-8}
\left(A^n,z^{n+1}\right)=\frac{\left(A^n,(I-\frac{1}{2}\tau\mathcal{D}\mathcal{L})^{-1}C^n\right)}{1-\frac{1}{4}\tau\left(A^n,(I-\frac{1}{2}\tau\mathcal{D}\mathcal{L})^{-1}\mathcal{D}A^n\right)}.
\end{equation} 
Substituting \eqref{eq-2-8} into \eqref{eq-2-7} gives $z^{n+1}$, and subsequently $w^{n+1}$ is obtained from the third equation of \eqref{eq-2-5}. The advantage of scheme {\bf SAV-CN} is that only two linear systems with constant coefficient need to be solved at each time step, which usually makes it very efficient when fast solvers are available.

\begin{remark}\label{Re-2-1}
One shortcoming of the SAV approach is that the nonlinear term $\mathcal{H}_2$ is assumed to be bounded-from-below such that there exists a constant $C_0$ to make the definition of the auxiliary variable meaningful. However, this assumption is not always satisfied. For example, consider the classic KdV equation \cite{luigi-19-HBVM-KDV-JCAM} which can be written as the form of \eqref{eq-2-1}
\begin{equation}\label{eq-2-9}
\begin{aligned}
z_t&={\partial_x}\mu,\\
\mu&=-\partial_{xx}z-\frac{1}{2}z^2.
\end{aligned}
\end{equation} 
The nonlinear term $\mathcal{H}_2$ is $-\frac{1}{6}(z^3,1)$ and obviously is not bounded-from-below. Consequently, the feasibility of the SAV approach can not be guaranteed, although no problems occur in most of conventional examples of the KdV equation, which will be shown in the following tests.
\end{remark}
To overcome this issue, the ESAV approach is proposed very recently \cite{liu-20-ESAV-SIAMJSC}. The major difference between the SAV and ESAV approaches in the system reformation appears in the definition of the auxiliary variable. The ESAV approach introduce an exponential auxiliary variable without square root such that the restriction of the bounded-from-below is removed. In the following section, we briefly present the idea of the ESAV approach and a related second-order Crank-Nicolson discretization.

\subsection{Exponential scalar auxiliary variable approach}
In the framework of the ESAV approch, let $e(t)=\exp(\mathcal{H}_2)$, the system \eqref{eq-2-1} can be reformulated as the following equivalent system:
\begin{align}\label{eq-2-10}
\aligned
&z_t=\mathcal{D}\mu,\\
&\mu~=\mathcal{L}z+B(z,e),\quad B(z,e)=\frac{\mathcal{N}'(z)e}{\exp\big(\mathcal{H}_2(z)\big)},\\
&\frac{d}{dt}\mbox{ln}(e)=\big(B(z,e),z_t\big).
\endaligned
\end{align}
Taking the inner product of the first two equations in \eqref{eq-2-10} with $\mu$, $z_t$, and combining them with the third equation, we can obtain a  modified energy conservation law of the ESAV reformation
\begin{equation}\label{eq-2-11}
\frac{d}{dt}\tilde{\mathcal{H}}(t)=0,\quad \tilde{\mathcal{H}}(t)=\frac{1}{2}\big(z,\mathcal{L}z\big)+\mbox{ln}(e).
\end{equation} 

\begin{remark}\label{Re-2-2}
Due to the exponential form of the auxiliary variable, the above KdV equation \eqref{eq-2-9} can now be reformulated as \eqref{eq-2-10} without additional assumptions.
\end{remark}

\begin{remark}\label{Re-2-3}
An alternative auxiliary variable $e(t)=\exp\big(\mathcal{H}_2(z)/C_0\big)$ is introduced in some numerical calculations, where $C_0$ is a constant to remove the risk of calculation failure caused by the rapidly increasing exponential function. And $C_0$ is generally chosen to be equal to the absolute value of $\mathcal{H}_2(z^0)$ that can be directly obtained from initial conditions. At this time, the system \eqref{eq-2-10} can be rewritten as
\begin{align}\label{eq-2-12}
\aligned
&z_t=\mathcal{D}\mu,\\
&\mu~=\mathcal{L}z+B(z,e),\quad B(z,e)=\frac{\mathcal{N}'(z)e}{\exp \big(\mathcal{H}_2(z)/C_0\big)},\\
&\frac{d}{dt}\ln(e)=\big(B(z,e),z_t\big)\big/C_0.
\endaligned
\end{align}
with the modified energy $\tilde{\mathcal{H}}(t)=\frac{1}{2}\big(z,\mathcal{L}z\big)+C_0\ln (e)$.
\end{remark} 

Similarly, we can apply the Crank-Nicolson and extrapolation methods to generate a simple second-order method
\begin{align}\label{eq-2-13}
\aligned
&\frac{z^{n+1}-z^n}{\tau}=\mathcal{D}\mu^{n+\frac{1}{2}},\\
&\mu^{n+\frac{1}{2}}=\mathcal{L}z^{n+\frac{1}{2}}+B(\bar{z}^{n+\frac{1}{2}},\bar{e}^{n+\frac{1}{2}}),\quad B(\bar{z}^{n+\frac{1}{2}},\bar{e}^{n+\frac{1}{2}})=\frac{\mathcal{N}'(\bar{z}^{n+\frac{1}{2}})\bar{e}^{n+\frac{1}{2}}}{\mbox{exp}\big(\mathcal{H}_2(\bar{z}^{n+\frac{1}{2}})\big)}\triangleq B^n,\\
&\mbox{ln}(e^{n+1})-\mbox{ln}(e^n)=\big(B(\bar{z}^{n+\frac{1}{2}},\bar{e}^{n+\frac{1}{2}}),z^{n+1}-z^n\big).
\endaligned
\end{align}
where $\bar{z}^{n+\frac{1}{2}}, \bar{e}^{n+\frac{1}{2}}=\frac{3e^n-e^{n-1}}{2}$ are defined as above. Hereafter, the method \eqref{eq-2-13} is abbreviated as \textbf{ESAV-CN}.

\begin{theorem}\label{Th-2-2}
The method \eqref{eq-2-13} is a linear second-order energy-preserving method with the discrete version of the energy as
\begin{equation}\label{eq-2-14}
\tilde{\mathcal{H}}^{n+1}=\tilde{\mathcal{H}}^n,\quad \tilde{\mathcal{H}}^n=\frac{1}{2}\big(z^n,\mathcal{L}z^n\big)+\ln (e^n).
\end{equation} 
\end{theorem} 

\begin{proof}
The linear property and the accuracy of \eqref{eq-2-13} are obvious. The equation \eqref{eq-2-14} can be derived by taking the inner products of the first two equations in \eqref{eq-2-13} with $\mu^{n+\frac{1}{2}}$, $\frac{z^{n+1}-z^n}{\tau}$, respectively, and using the third equation in \eqref{eq-2-13}.
\end{proof}

Comparing the schemes \textbf{SAV-CN} \eqref{eq-2-5} and \textbf{ESAV-CN} \eqref{eq-2-13}, the major difference corresponds to the discretization of nonlinear term. \textbf{SAV-CN} treats the auxiliary variable implicitly and $z$ explicitly to achieve the energy conservation, whereas in  \textbf{ESAV-CN} both variables are treated explicitly, without destroying the energy conservation. Meanwhile, the resulting scheme \eqref{eq-2-13} is fully decoupled, i.e., the solution variable and the auxiliary variable can be solved step by step
\begin{align}\label{eq-2-15}
\aligned
&z^{n+1}=\left(I-\frac{1}{2}\tau\mathcal{D}\mathcal{L}\right)^{-1}\left[\big(I+\frac{1}{2}\tau\mathcal{D}\mathcal{L}\big)z^n+\tau\mathcal{D}B^n\right],\\
&e^{n+1}=\mbox{exp}\left(\mbox{ln}(e^n)+(B^n,z^{n+1}-z^n)\right).
\endaligned
\end{align}
Hence, \textbf{ESAV-CN} is much more efficient than \textbf{SAV-CN} in practical computations. Such advantage is also inherited in the following construction of high-order linearly implicit energy-preserving algorithms.

\section{The construction of high-order linearly implicit energy-preserving algorithms}\label{Sec-3}
For Hamiltonian PDEs \eqref{eq-2-1}, the construction of linear high-order energy-preserving methods is a popular and confused project. In this section, we apply the Gauss collocation method, the high-order extrapolation and prediction-correction techniques to discretize \eqref{eq-2-10} in time and advocate two novel methodologies to construct arbitrary high-order linear energy-preserving methods which can be implemented simply. As we all know, symplectic Runge-Kunta methods can maintain quadratic invariants \cite{cooper-87-RK-IMA,hairer-06-GNI-ODE}. Owing to the use of extrapolation, in this paper we investigate the Gauss collocation method to devise energy-preserving methods, in view of their excellent stability and high accuracy \cite{leimkuhler-04-SHD,hairer-06-GNI-ODE}.

\subsection{Symplectic RK methods}
Consider the non-autonomous system of first-order ordinary differential equations
\begin{equation}\label{eq-3-1}
\frac{d}{dt}z = f(t,z),\quad z(t_0)=z_0.
\end{equation}

\begin{definition}[RK methods \cite{leimkuhler-04-SHD,hairer-06-GNI-ODE}]\label{De-3-1} 
For one-step interval $[t_n,t_{n+1}]$, let $b_i,~a_{ij}~(i,j=1,2,\cdots,s)$ be real numbers and $c_i=\sum_{i=1}^sa_{ij}$. An $s$-stage Runge-Kutta method is given by
\begin{equation}\label{eq-3-2}
\begin{aligned}
&z_i^n=z^n+\tau\sum_{j=1}^sa_{ij}k_j^n,\quad k_i^n=f\left(t_0+c_i\tau,z_i^n\right),\quad i=1,2,\cdots,s,\\
&z^{n+1}=z^n+\tau\sum_{i=1}^sb_ik_i^n,
\end{aligned}
\end{equation}
where $z_i^n$ are the internal values at the current step.
\end{definition}

By Butcher's tabular \cite{hairer-06-GNI-ODE}, the coefficients in the RK method \eqref{eq-3-2} are usually displayed as follows:
\begin{equation}\label{eq-3-3}
\begin{array}{c|ccc}
c_1 & a_{11} & \cdots & a_{1s}\\
\vdots & \vdots &  & \vdots\\
c_s & a_{s1} & \cdots & a_{ss}\\
\hline
& b_1 & \cdots & b_s\\
\end{array}.
\end{equation}

\begin{lemma}[RK symplecticity conditions \cite{leimkuhler-04-SHD,hairer-06-GNI-ODE}]\label{Le-3-1}
If the coefficients of a RK method \eqref{eq-3-2} satisfy
\begin{equation}\label{eq-3-4}
b_ia_{ij}+b_ja_{ji}=b_ib_j\quad \mbox{for all}\;\; i,j=1,2,\cdots,s,
\end{equation}
then it is symplectic and can conserve all quadratic invariants of \eqref{eq-3-1}.
\end{lemma}

Let $c_1,\cdots,c_s$ be the zeros of the $s$th shifted Legendre polynomial \cite{hairer-06-GNI-ODE} and $b_1,\cdots,b_s$ be the weights of the Gauss quadrature formula, we recognize the $s$-stage Gauss method. In particular, the following $2$-stage~(\textbf{Gauss2})~and $3$-stage~(\textbf{Gauss3})~Gauss methods will be performed in the numerical computations. The monographs \cite{leimkuhler-04-SHD,hairer-06-GNI-ODE} are suggested to readers for discovering higher order Guass methods. 
\begin{equation}\label{eq-3-5}
\begin{array}{c|cc}
\frac{1}{2}-\frac{\sqrt{3}}{6} & \frac{1}{4} & \frac{1}{4}-\frac{\sqrt{3}}{6}\\
\frac{1}{2}+\frac{\sqrt{3}}{6} & \frac{1}{4}+\frac{\sqrt{3}}{6} & \frac{1}{4}\\
\hline
& \frac{1}{2} & \frac{1}{2}
\end{array}\quad\mbox{and}\quad
\begin{array}{c|cccc}
\frac{1}{2}-\frac{\sqrt{15}}{10} & \frac{5}{36} & \frac{2}{9}-\frac{\sqrt{15}}{15} & \frac{5}{36}-\frac{\sqrt{15}}{30}\\
\frac{1}{2} & \frac{5}{36}+\frac{\sqrt{15}}{24} & \frac{2}{9}  &\frac{5}{36}-\frac{\sqrt{15}}{24} \\
\frac{1}{2}+\frac{\sqrt{15}}{10} & \frac{5}{36}+\frac{\sqrt{15}}{30} & \frac{2}{9}+\frac{\sqrt{15}}{15} & \frac{5}{36}\\
\hline
& \frac{5}{18} & \frac{4}{9} & \frac{5}{18}
\end{array}.
\end{equation}

Using the $s$-stage symplectic RK method to discretize \eqref{eq-2-10} in time, the following semi-discrete $s$-stage method is presented as
\begin{subequations}\label{eq-3-6}
\begin{align}
& z_i^n=z^n+\tau\sum_{j=1}^sa_{ij}k_j^n,\quad k_i^n=\mathcal{D}\mu_i^n,\quad \mu_i^n=\mathcal{L}z_i^n+B(z_{s,i}^n,e_{s,i}^n),\label{eq-3-6-a}\\
&e_i^n=\mbox{exp}\left(\mbox{ln}(e^n)+\tau\sum_{j=1}^sa_{ij}l_j^n\right),\quad l_i^n=\left(B(z_{s,i}^n,e_{s,i}^n),k_i^n\right),\quad i=1,2,\cdots,s,\label{eq-3-6-b}\\
&z^{n+1}=z^n+\tau\sum_{i=1}^sb_ik_i^n,\quad e^{n+1}=\mbox{exp}\left(\mbox{ln}(e^n)+\tau\sum_{i=1}^sb_il_i^n\right).\label{eq-3-6-c}
\end{align}
\end{subequations}

Moreover, when all the coefficients are given by the ones of the Gauss methods \eqref{eq-3-5}, we name this method \eqref{eq-3-6} \textbf{ESAV-Gauss}. Suppose we have time grid points $t_n=t_0+n\tau$. For the first time interval $[t_0,t_1]$, we use a fully implicit symplectic RK method to calculate $z_i^0,~e_i^0$ and $z^1,~e^1$ as 
\begin{align*}
& z_i^0=z^0+\tau\sum_{j=1}^sa_{ij}k_j^0,\quad k_i^0=\mathcal{D}\mu_i^0,\quad \mu_i^0=\mathcal{L}z_i^0+B(z_i^0,e_i^0),\\
&e_i^0=\mbox{exp}\left(\mbox{ln}(e^0)+\tau\sum_{j=1}^sa_{ij}l_j^0\right),\quad l_i^0=\left(B(z_{i}^0,e_{i}^0),k_i^0\right),\quad i=1,2,\cdots,s,\\
&z^{1}=z^0+\tau\sum_{i=1}^sb_ik_i^0,\quad e^{1}=\mbox{exp}\left(\mbox{ln}(e^1)+\tau\sum_{i=1}^sb_il_i^0\right).
\end{align*}
For $n>1$, we apply scheme \textbf{ESAV-Gauss} \eqref{eq-3-6} to update the solutions. More specifically, suppose we have obtained the values of $z^{n-1}$ and internal stages $z^{n-1}_i$ in time interval $[t_{n-1},t_n]$, then the internal values $z_{s,i}^n$ are approximated by the interpolation polynomial of points $(t_{n-1},z^{n-1})$ and $\left(t_{n-1}+c_i\tau,z_i^{n-1}\right)$. The treatment of the auxiliary variable $e$ is exactly the same, so one obtains the nonlinear term $B(z_{s,i}^n,e_{s,i}^n)$ as an explicit high-order extrapolation of $B(z_i^n,e_i^n)$.

\begin{remark}\label{Re-3-1}
Considering the implementation of \eqref{eq-3-6}, we first get a linear equation of $z_i^n$ from \eqref{eq-3-6-a}. Then substituting the obtained $z_i^n$ into \eqref{eq-3-6-b} can calculate $e_i^n$. Finally, the numerical solutions $z^{n+1}$ and $e^{n+1}$ are updated by \eqref{eq-3-6-c}. Noting that the internal values $z_i^n$ and $e_i^n$ are decoupled to be solved step by step. Moreover, the scheme \eqref{eq-3-6} removes the computational complexity caused by high-order SAV schemes \cite{akrivis-19-RK-SAV-SIAMJSC}.
\end{remark}

Taking $s=2$ as an interpretation for the nonlinear term $B(z_{s,i}^n,e_{s,i}^n)$, one first derive the following interpolation polynomial of degree $2$ as:
\begin{equation}\label{eq-3-7}
z_s^n(c)=\frac{(1+c-c_1)(1+c-c_1)}{c_1c_2}z^{n-1}+\frac{(1+c)(1+c-c_2)}{c_1(c_1-c_2)}z_1^{n-1}+\frac{(1+c)(1+c-c_1)}{c_2(c_2-c_1)}z_2^{n-1}
\end{equation}
with $c_1=\frac{1}{2}-\frac{\sqrt{3}}{6}$ and $c_2=\frac{1}{2}+\frac{\sqrt{3}}{6}$. Then we have the approximation of internal values as
\begin{equation}\label{eq-3-8}
\begin{aligned}
&z_{2,1}^n=z_s^n(c_1)=(-2\sqrt{3}+6)z^{n-1}+(-3\sqrt{3}+1)z_1^{n-1}+(5\sqrt{3}-6)z_2^{n-1},\\
&z_{2,2}^n=z_s^n(c_2)=(2\sqrt{3}+6)z^{n-1}+(-5\sqrt{3}-6)z_1^{n-1}+(3\sqrt{3}+1)z_2^{n-1}.
\end{aligned}
\end{equation}
Similarly, the expressions of $e_{2,i}^n$ can be listed with the same coefficients as \eqref{eq-3-8}. Furthermore, the coefficients of the interpolation polynomial determined by \textbf{Gauss3} are reported in the following table.
\begin{table}[H]
\tabcolsep=19pt \small \renewcommand \arraystretch{1.3} \centering
\caption{Coefficients of the interpolation polynomial with $s=3$.}\label{Tab-1}
\begin{tabularx}{\textwidth}{*{5}{c}} \toprule
$t_n+c_i\tau$ & $z^{n-1}~(e^{n-1})$ & $z_1^{n-1}~(e_1^{n-1})$ & $z_2^{n-1}~(e_2^{n-1})$  & $z_3^{n-1}~(e_3^{n-1})$ \\ \midrule
\multirow{1}*{$z_{3,1}^{n}~(e_{3,1}^{n})$}
& $6\sqrt{15}-26$ & $-5\sqrt{15}/3+11$ & $16\sqrt{15}/3-24$ & $-29\sqrt{15}/3+40$\\
\multirow{1}*{$z_{3,2}^{n}~(e_{3,2}^{n})$}
& $-17$ & $5\sqrt{15}/2+35/2$ & $-17$ & $-5\sqrt{15}/2+35/2$\\ 
\multirow{1}*{$z_{3,3}^{n}~(e_{3,3}^{n})$}
& $-6\sqrt{15}-26$ & $29\sqrt{15}/3+40$ & $-16\sqrt{15}/3-24$ & $5\sqrt{15}/3+11$\\ \bottomrule
\end{tabularx}
\end{table}	 

In particular, the remainder of the interpolation polynomial shows that the equations
\begin{equation}\label{eq-3-9}
z_{2,i}^n = z(t_n+c_i\tau)+\mathcal{O}(\tau^3),\quad e_{2,i}^n = e(t_n+c_i\tau)+\mathcal{O}(\tau^3),\quad i=1,2.
\end{equation}
This indicates that the scheme \eqref{eq-3-6} may achieve third order accuracy when $s=2$ and could be accompanied with a general accuracy $\mathcal{O}(\tau^{s+2})$. Note that too many interpolation points are not selected here, otherwise, the resulting interpolation polynomial could be highly oscillating, which makes the extrapolation not sufficiently accurate \cite{gong-20-high-stable-JCP}. In order to improve the accuracy, a novel prediction-correction strategy is proposed later to regain the accuracy of the Gauss method. Nevertheless, the above scheme \eqref{eq-3-6} satisfies the following energy conservation property.

\begin{theorem}\label{Th-3-2}
The semi-discrete $s$-stage method \eqref{eq-3-6} is a linearly implicit energy-preserving method subject to the discrete energy conservation law
\begin{equation}\label{eq-3-10}
\tilde{\mathcal{H}}_s^{n+1}=\tilde{\mathcal{H}}_s^n,\quad \tilde{\mathcal{H}}_s^n=\frac{1}{2}\big(z^n,\mathcal{L}z^n\big)+\ln (e^n).
\end{equation} 
\end{theorem} 

\begin{proof}
The linear property of \eqref{eq-3-6} is discussed in Remark \ref{Re-3-1}. Using the left-hand equation of \eqref{eq-3-6-c}, one can get 
\begin{equation}\label{eq-3-11}
\left(z^{n+1},\mathcal{L}z^{n+1}\right)=\left(z^n,\mathcal{L}z^n\right)+\tau\sum_{i=1}^{s}b_i\left(k_i^{n},\mathcal{L}z^n\right)+\tau\sum_{j=1}^{s}b_j\left(z^{n},\mathcal{L}k_j^{n}\right) +\tau^{2}\sum_{i,j=1}^{s}b_ib_{j}\left(k_i^{n},\mathcal{L}k_j^{n}\right).
\end{equation} 	
Substituting $\mathcal{L}z^n=\mathcal{L}z_i^n-\tau\sum_{j=1}^sa_{ij}\mathcal{L}k_j^n$ into \eqref{eq-3-11}, and noting the symmetry of $\mathcal{L}$, we derive
\begin{equation}\label{eq-3-12}
\left(z^{n+1},\mathcal{L}z^{n+1}\right)=\left(z^n,\mathcal{L}z^n\right)+2\tau\sum_{i=1}^{s}b_i\left(k_i^{n},\mathcal{L}z_i^n\right)+\tau^{2}\sum_{i,j=1}^{s}\left(b_ib_j-b_ia_{ij}-b_ja_{ji}\right)\left(k_i^{n},\mathcal{L}k_j^{n}\right).
\end{equation} 	
The symplecticity conditions \eqref{eq-3-4} mean that $b_ib_j-b_ia_{ij}-b_ja_{ji}=0$. The right-hand equation of \eqref{eq-3-6-c} implies that
\begin{equation}\label{eq-3-13}
\mbox{ln}\left(e^{n+1}\right)=\mbox{ln}\left(e^n\right)+\tau\sum_{i=1}^sb_il_i^n.
\end{equation} 	
Summing \eqref{eq-3-12} and \eqref{eq-3-13}, it is a simple check for the equation
\begin{equation}\label{eq-3-14}
\tilde{\mathcal{H}}_s^{n+1}=\tilde{\mathcal{H}}_s^n+\tau\sum_{i=1}^{s}b_i\big[\left(k_i^n,\mathcal{L}z_i^n\right)+l_i^n\big].
\end{equation} 
The final result \eqref{eq-3-10} can be obtained from
\begin{equation}\label{eq-3-15}
\left(k_i^n,\mathcal{L}z_i^n\right)=\left(k_i^n,\mu_i^n-B(z_{s,i}^n,e_{s,i}^n)\right)=(\mathcal{D}\mu_i^n,\mu_i^n)-\left(k_i^n,B(z_{s,i}^n,e_{s,i}^n)\right)=-l_i^n.
\end{equation} 
\end{proof}

\subsection{Prediction-correction methods}
A new prediction-correction strategy different from that in \cite{gong-20-high-stable-JCP} is developed in this subsection to improve the accuracy of \eqref{eq-3-6}. Thus, another family of linear high-order energy-preserving methods is proposed for \eqref{eq-2-10}. When all the coefficients are selected as the ones of the Gauss method, we simply denote these methods as \textbf{ESAV-Gauss-PC}.

\begin{ESAV-Gauss-PC}
Let $\lambda$ and $\Lambda$ be iteration variable and maximal iteration step, respectively. The mark $TOL$ is a given tolerable error. Assume $z_{s,i}^n$ and $e_{s,i}^n$ have been obtained from the interpolations with low-order accuracy. Setting $z_i^{n(1)}=z_{s,i}^n$ and $e_i^{n(1)}=e_{s,i}^n$, then we iteratively solve $z_i^{n(\lambda+1)}$ and $e_i^{n(\lambda+1)}$ from $\lambda=1,2,\cdots,\Lambda$ by
\begin{equation}\label{eq-3-16}
\begin{aligned}
& z_i^{n(\lambda+1)}=z^n+\tau\sum_{j=1}^sa_{ij}k_j^{n(\lambda+1)},\quad k_i^{n(\lambda+1)}=\mathcal{D}\mu_i^{n(\lambda+1)},\quad \mu_i^{n(\lambda+1)}=\mathcal{L}z_i^{n(\lambda+1)}+B\left(z_i^{n(\lambda)},e_i^{n(\lambda)}\right),\\
&e_i^{n(\lambda+1)}=\exp\left(\ln (e_n)+\tau\sum_{j=1}^sa_{ij}l_j^{n(\lambda+1)}\right),\quad l_i^{n(\lambda+1)}=\left(B\left(z_i^{n(\lambda)},e_i^{n(\lambda)}\right),k_i^{n(\lambda+1)}\right),\quad i=1,2,\cdots,s.\\
\end{aligned}
\end{equation}
If $\max\limits_{1\leq i\leq s}\left\{\|z_i^{n(\lambda+1)}-z_i^{n(\lambda)}\|_{\infty},~\|e_i^{n(\lambda+1)}-e_i^{n(\lambda)}\|_{\infty}\right\}<\mbox{TOL}$, the iteration procedure is terminated with furtherly evaluating $k_i^{n(\lambda+1)}$ and $l_i^{n(\lambda+1)}$. And if not, we take $k_i^{n(\lambda+1)}=k_i^{n(\Lambda+1)}$, $l_i^{n(\lambda+1)}=l_i^{n(\Lambda+1)}$. Finally, the numerical solutions are updated by
\begin{equation}\label{eq-3-17}
z^{n+1}=z^n+\tau\sum_{i=1}^sb_ik_i^{n(\lambda+1)},\quad e^{n+1}=\exp\left(\ln (e^n)+\tau\sum_{i=1}^sb_il_i^{n(\lambda+1)}\right).
\end{equation}
\end{ESAV-Gauss-PC}

Similar as schemes \textbf{ESAV-CN} and \textbf{ESAV-Gauss}, the implementation of \textbf{ESAV-Gauss-PC} is totally explicit, and $z_i^{n(\lambda+1)}$, $e_i^{n(\lambda+1)}$ can be solved step by step. Moreover, such high-order scheme also admits an energy conservation law.

\begin{theorem}\label{Th-3-3}
For each iteration step $\lambda=1,2,\cdots,\Lambda$, \textbf{ESAV-Gauss-PC} is a linear energy-preserving method with the discrete energy conservation law as follows: 
\begin{equation}\label{eq-3-18}
\tilde{\mathcal{H}}_s^{n+1}=\tilde{\mathcal{H}}_s^n,\quad \tilde{\mathcal{H}}_s^n=\frac{1}{2}\big(z^n,\mathcal{L}z^n\big)+\ln (e^n).
\end{equation} 	
\end{theorem}

\begin{proof}
According to \eqref{eq-3-17}, this gives
\begin{equation}\label{eq-3-19}
\mathcal{L}z^{n+1}=\mathcal{L}z^n+\tau\sum_{i=1}^sb_i\mathcal{L}k_i^{n(\lambda+1)},\quad \mbox{ln}\left(e^{n+1}\right)=\mbox{ln}(e^n)+\tau\sum_{i=1}^sb_il_i^{n(\lambda+1)}.
\end{equation}
Taking the inner product of the left-hand equation in \eqref{eq-3-19} with $z^{n+1}$, we obtain
\begin{equation}\label{eq-3-20}
\begin{aligned}
\left(z^{n+1},\mathcal{L}z^{n+1}\right)
&=\left(z^n,\mathcal{L}z^n\right)+\tau\sum_{i=1}^{s}b_i\left(k_i^{n(\lambda+1)},\mathcal{L}z^n\right)+\tau\sum_{j=1}^{s}b_j\left(z^{n},\mathcal{L}k_j^{n(\lambda+1)}\right)\\
&~~~+\tau^{2}\sum_{i,j=1}^{s}b_ib_{j}\left(k_i^{n(\lambda+1)},\mathcal{L}k_j^{n(\lambda+1)}\right).
\end{aligned}
\end{equation}
Inserting $\mathcal{L}z^n=\mathcal{L}z_i^{n(\lambda+1)}-\tau\sum_{j=1}^sa_{ij}\mathcal{L}k_j^{n(\lambda+1)}$ into \eqref{eq-3-20} and using the symmetry of $\mathcal{L}$ yield
\begin{equation}\label{eq-3-21}
\begin{aligned}
\left(z^{n+1},\mathcal{L}z^{n+1}\right)
&=\left(z^n,\mathcal{L}z^n\right)+2\tau\sum_{i=1}^{s}b_i\left(k_i^{n(\lambda+1)},\mathcal{L}z_i^{n(\lambda+1)}\right)\\
&~~~+\tau^{2}\sum_{i,j=1}^{s}\left(b_ib_j-b_ia_{ij}-b_ja_{ji}\right)\left(k_i^{n(\lambda+1)},\mathcal{L}k_j^{n(\lambda+1)}\right).
\end{aligned}
\end{equation}
Applying the symplecticity condition\eqref{eq-3-4} and summing with the right-hand equation of \eqref{eq-3-19} report that
\begin{equation}\label{eq-3-22}
\tilde{\mathcal{H}}_s^{n+1}=\tilde{\mathcal{H}}_s^n+\tau\sum_{i=1}^{s}b_i\bigg[\left(k_i^{n(\lambda+1)},\mathcal{L}z_i^{n(\lambda+1)}\right)+l_i^{n(\lambda+1)}\bigg].
\end{equation} 
With the first row of \eqref{eq-3-16} and the skew-adjoint operator $\mathcal{D}$, it is sufficient to verify \eqref{eq-3-18} by
\begin{equation}\label{eq-3-23}
\begin{aligned}
\left(k_i^{n(\lambda+1)},\mathcal{L}z_i^{n(\lambda+1)}\right)
&=\left(k_i^{n(\lambda+1)},\mu_i^{n(\lambda+1)}-B\left(z_i^{n(\lambda)},e_i^{n(\lambda)}\right)\right)\\
&=\left(\mathcal{D}\mu_i^{n(\lambda+1)},\mu_i^{n(\lambda+1)}\right)-\left(k_i^{n(\lambda+1)},B\left(z_i^{n(\lambda)},e_i^{n(\lambda)}\right)\right)=-l_i^{n(\lambda+1)}.
\end{aligned}
\end{equation}
\end{proof}

\begin{remark}\label{Re-3-2}
The above prediction-correction strategy is actually different from that in \cite{gong-20-high-stable-JCP} where the prediction step is to obtain a high-order approximation of internal stages and the correction step is used to maintain the energy conservation. Here, the prediction step corresponds to the interpolation of the internal stages with low-order accuracy, whereas the conservative correction step is only used to improve the accuracy.	
\end{remark}

\begin{remark}\label{Re-3-3}
Scheme \textbf{ESAV-Gauss-PC} can be viewed as a fully implicit symplectic RK method applied to the reformulated system \eqref{eq-2-10}, with initial values in the fixed point iteration choosing as the low-order interpolation. From the following numerical examples,  we observe that the correction step in \eqref{eq-3-16} is nearly 5. While in the standard symplectic RK method, the initial values are usually chosen as $z^n$ and $e^n$ uniformly so that the convergence is much slower than scheme \textbf{ESAV-Gauss-PC}.
\end{remark}

\section{Spatial Fourier pseudo-spectral methods}\label{Sec-4}
This paper focuses on the setting of periodic boundary conditions of \eqref{eq-2-1}. This suggests that we apply Fourier pseudo-spectral methods \cite{gong-17-FP-NLS-JCP,shen-11-spectral} to spatial discretization. Taking $\Omega=[x_R,x_L]\times[y_R,y_L]$ as an example, the spatial mesh sizes equal to $h_x=l_1/N_x$, $h_y=l_2/N_y$ with even numbers $N_x$, $N_y$ and periods $l_1=x_L-x_R$, $l_2=y_L-y_R$. Then grid points in space can be presented by
\begin{equation}\label{eq-4-1}
\Omega_h=\big\{(x_j,y_k)|x_j=x_R+(j-1)h_x,y_k=y_R+(k-1)h_y,j=1,2,\cdots,N_x,k=1,2,\cdots,N_y\big\}.
\end{equation} 
The following interpolation space is as follows
\begin{equation}\label{eq-4-2}
I_N=\mbox{span}\big\{X_m(x)Y_n(y),m=1,2,\cdots,N_x,n=1,2,\cdots,N_y\big\},
\end{equation} 
where $X_m(x)$ and $Y_n(y)$ are the exponential polynomials of degree $N_x/2$ and $N_y/2$, respectively. Using the Kronecker notation, these polynomials satisfy $X_m(x_j)=\delta_m^j$, $Y_n(y_k)=\delta_n^k$ and have explicit expressions  
\begin{equation}\label{eq-4-3}
X_m(x)=\frac{1}{N_x}\sum_{w=-N_x/2}^{N_x/2}\frac{1}{\alpha_w}e^{iw\mu_x\left(x-x_m\right)},\quad Y_n(y)=\frac{1}{N_y}\sum_{w=-N_y/2}^{N_y/2}\frac{1}{\alpha_w}e^{iw\mu_y\left(y-y_n\right)},
\end{equation} 
where $\alpha_w=1$ ($-N_\gamma/2<w<N_\gamma/2$) and $\alpha_{-N_\gamma/2}=\alpha_{N_\gamma/2}=2$ with $\gamma=x\ \mbox{or}\ y$. The notations $\mu_x=2\pi/l_1$ and $\mu_y=2\pi/l_2$ are correction factors for the spatial region. Assume that $z_{jk}$ is the exact solution of $z$ at grid point $(x_j,y_k)$. Then the interpolation polynomial of $z$ can be given by
\begin{equation}\label{eq-4-4}
\big(\mathcal{I}_Nz\big)\left(x,y\right)=\sum_{m=1}^{N_x}\sum_{n=1}^{N_y}z_{mn}X_m(x)Y_n(y)
\end{equation} 
with $\big(\mathcal{I}_Nz\big)\left(x_j,y_k\right)=z_{jk}$. In order to evaluate the values of derivatives $\frac{\partial^2(\mathcal{I}_Nz)}{\partial x^2}$ and $\frac{\partial^2(\mathcal{I}_Nz)}{\partial y^2}$ at $(x_j,y_k)$, differentiating \eqref{eq-4-4} induces
\begin{align}
&\frac{\partial^2(\mathcal{I}_Nz)}{\partial x^2}(x_j,y_k)=\sum_{m=1}^{N_x}\sum_{n=1}^{N_y}z_{mn}\frac{d^2 X_m}{dx^2}\left(x_j\right)Y_n(y_k)=\left(D_x^2\textbf{z}\right)_{jk},\label{eq-4-5}\\
&\frac{\partial^2(\mathcal{I}_Nz)}{\partial y^2}(x_j,y_k)=\sum_{m=1}^{N_x}\sum_{n=1}^{N_y}z_{mn}X_m(x_j)\frac{d^2 Y_n}{dy^2}\left(y_k\right)=\left(\textbf{z}D_y^2\right)_{jk},\label{eq-4-6}
\end{align} 
where $\textbf{z}=\left(z_{jk}\right)$ is the $N_x\times N_y$ matrix and $D_\gamma^2\ (\gamma=x\ \mbox{or}\ y)$ is the second-order differentiation matrix associated with Fourier pseudo-spectral methods. In addition, $D_\gamma^2$ is the $N_\gamma\times N_\gamma$ real symmetric matrix that can be explicitly displayed as follows:
\begin{align}\label{eq-4-7}
(D_\gamma^2)_{mn}=
\left\{\aligned
&-\mu_\gamma^2\frac{N_\gamma^2+2}{12},\quad\quad\quad\quad\quad\quad\quad\quad\quad\quad\quad\mbox{if}\ m=n,\\
&(-1)^{m+n+1}\frac{\mu_\gamma^2}{2}\mbox{csc}^2\left(\frac{(m-n)\mu_\gamma h_\gamma}{2}\right),\quad \mbox{if}\ m\neq n.\\
\endaligned
\right.
\end{align}
The expression of first-order differentiation matrix can be derived in exactly the same way as
\begin{align}\label{eq-4-8}
(D_\gamma^1)_{mn}=
\left\{\aligned
&0,\quad\quad\quad\quad\quad\quad\quad\quad\quad\quad\quad\quad\quad\quad~~\mbox{if}\ m=n,\\
&(-1)^{m+n}\frac{\mu_\gamma}{2}\mbox{cot}\left(\frac{(m-n)\mu_\gamma h_\gamma}{2}\right),\quad \mbox{if}\ m\neq n.\\
\endaligned
\right.
\end{align}
For higher order spectral differential matrices, please refer to \cite{gong-17-FP-NLS-JCP,shen-11-spectral}. Let $Z=(Z_{jk})$ be matrix of the numerical solution, for arbitrary grid functions $Z$ and $\hat{Z}$, we define the discrete inner product
\begin{equation}\label{eq-4-9}
\left(Z,\hat{Z}\right)_h=h_xh_y\sum_{j=1}^{N_x}\sum_{k=1}^{N_y}{Z}_{jk}{\hat{Z}}_{jk}
\end{equation} 
with the induced discrete $L^2$ and $L^{\infty}$ norms
\begin{equation}\label{eq-4-10}
||Z||_h=\sqrt{(Z,Z)_h},\quad ||Z||_{h,\infty}=\mathop{\max}\limits_{(x_j,y_k)\in\Omega_h}\left|Z_{jk}\right|
\end{equation} 
and the following semi-norm \cite{gong-17-FP-NLS-JCP} 
\begin{equation}\label{eq-4-11}
|Z|_h=\sqrt{(-D_x^2Z,Z)_h+(-ZD_y^2,Z)_h}.
\end{equation} 

These norms are used in the discrete energy expression of the corresponding fully discrete scheme. For space limitation, the relevant details can be omitted. In addition, there exists the relation $D_\gamma^\alpha=F_{N_\gamma}^{H}\Lambda_\gamma^\alpha F_{N_\gamma}$~$(\alpha=1\ \mbox{or}\ 2)$, where $F_{N_\gamma}$ is the discrete Fourier transform matrix with its conjugate transpose $F_{N_\gamma}^H$. The matrix $\Lambda_\gamma^\alpha$ is diagonal with the eigenvalues of $D_\gamma^\alpha$ being its entries. In actual calculations, the linear and decoupled properties make the implementation of all fully discrete schemes efficient with the fast Fourier transformation.

\section{Numerical experiments}\label{Sec-5}
In this section, extensive numerical experiments demonstrate the accuracy and energy conservation of the proposed algorithms. For the prediction-correction method \textbf{ESAV-Gauss-PC}, we always take $TOL=10^{-12}$ in all numerical examples. The computing environment is Matlab R2016a with Intel Core i5-3470 CPU, 3.20 GHz and 4GB memory.\\

\noindent \textbf{Example 5.1.}~(Accuracy and efficiency tests) We investigate the NLS equation
\begin{equation}\label{eq-5-1}
\mbox{i}u_t+\triangle u +\beta |u|^2u=0,\quad (x,y)\in\Omega\in \mathbb{R}^2,\ t\in (0,T]
\end{equation} 
that generates a progressive plane wave solution \cite{gong-17-FP-NLS-JCP}
\begin{equation}\label{eq-5-2}
u(x,y,t)=A\ \mbox{exp}\big(i(c_1x+c_2y-\omega t)\big)
\end{equation} 
with $\omega=c_1^2+c_2^2-\beta A^2$. This example is numerically solved on the spatial region $\Omega=[0,2\pi)\times[0,2\pi)$, where we choose $\beta=c_1=c_2=A=1$ and $N_x=N_y$. The exact solution \eqref{eq-5-2} gives the initial condition by setting $t=0$, and $(2\pi,2\pi)$-periodic boundary condition is used. The different choices of the space division make the spatial error small to be negligible.

Table. \ref{Tab-2} reports that when the temporal error is small enough, the total error depends entirely on the error of the time discretization. In particular, this also confirms that the exact solution \eqref{eq-5-2} is sufficiently smooth to ensure that the spatial Fourier pseudo-spectral method has high precision. In the following, the temporal convergence rates of \textbf{SAV-CN} and \textbf{ESAV-CN} are calculated by
\begin{equation}\label{eq-5-3}
\mbox{Order}=\mbox{log}_2\frac{\mbox{Error}(h,\tau)}{\mbox{Error}(h,\tau/2)}
\end{equation} 
with $h=2\pi/N_x$, where $\mbox{Error}(h,\tau)$ denotes the $L^2$ or $L^\infty$ error between the numerical solution and the exact solution \eqref{eq-5-2} under grid step sizes $h$ and $\tau$. As displayed in Table. \ref{Tab-3}, the second-order convergence of \textbf{SAV-CN} and \textbf{ESAV-CN} in the time direction is accurately evaluated. The two tables also list the total CPU times obtained by these two schemes in different situations, which indicates that \textbf{ESAV-CN} is more economical to fast simulation. In addition, Fig. \ref{Fig-1} depicts the comparisons of the numerical errors of these two schemes at different time steps, which also includes the derivation of the modified energy error. The computational results illustrate that \textbf{ESAV-CN} has a smaller numerical error than \textbf{SAV-CN}, although the energy errors appear to be almost the same up to the machine accuracy. 

For the proposed high-order schemes, in Fig. \ref{Fig-2}, it is an immediate fact that the accuracy of \textbf{ESAV-Gauss} is equal to the number of interpolation nodes. In terms of the prediction-correction approach, \textbf{ESAV-Gauss-PC} recovers the local error $\mathcal{O}(\tau^{2s+1})$ of the Gauss method. As expected in Table. \ref{Tab-4}, the smaller number of iterations required by \textbf{ESAV-Gauss-PC} makes it faster than the fully implicit Gauss method. And as the step size decreases, the number of iterations required also decreases. Fig. \ref{Fig-3} plots the CPU times corresponding to Table. \ref{Tab-4} for the sake of clarity.

\begin{table}[H]
\tabcolsep=12pt \small \renewcommand \arraystretch{1.3} \centering
\caption{Convergence tests of \textbf{SAV-CN} and \textbf{ESAV-CN} in space with $\tau=1e-05$ at $T=1$.}\label{Tab-2}
\begin{tabularx}{\textwidth}{*{7}{l}} \toprule
Schemes & $N_x\times N_y$ & $8\times8$ & $16\times16$ & $32\times32$ & $64\times64$ & $128\times128$\\ \midrule 
\multirow{3}*{\centering \textbf{SAV-CN}}
& L$^2$ error & 4.5457e-10 & 3.9004e-10 & 3.7237e-10 & 4.0874e-10 & 4.1974e-10\\
& L$^\infty$ error & 7.2610e-11 & 6.2875e-11 & 6.0072e-11 & 6.5728e-11 & 6.7738e-11\\
& CPU time & 33.63s & 44.93s & 79.35s & 213.59s & 668.90s\\ \midrule
\multirow{3}*{\textbf{ESAV-CN}}
& L$^2$ error & 8.4043e-11 & 8.6179e-11 & 1.0025e-10 & 7.6565e-11 & 7.4399e-11\\
& L$^\infty$ error & 1.4711e-11 & 1.4618e-11 & 1.6797e-11 & 1.3104e-11 & 1.2738e-11\\
& CPU time & 17.91s & 24.09s & 45.92s & 132.12s & 398.58s\\ \bottomrule
\end{tabularx}
\end{table}	

\begin{table}[H]
\tabcolsep=12pt \small \renewcommand \arraystretch{1.3} \centering
\caption{Convergence rates of \textbf{SAV-CN} and \textbf{ESAV-CN} in time with $N_x=N_y=64$ at $T=1$.}\label{Tab-3}
\begin{tabularx}{\textwidth}{*{7}{l}} \toprule
Schemes & $\tau$ & $0.001$ & $0.001/2$ & $0.001/4$ & $0.001/8$ & $0.001/16$\\ \midrule 
\multirow{5}*{\centering \textbf{SAV-CN}}
& L$^2$ error & 3.6589e-06 & 9.1551e-07 & 2.2898e-07 & 5.7261e-08 & 1.4322e-08\\
& Order & -- & 1.9988 & 1.9994 & 1.9996 & 1.9993\\
& L$^\infty$ error & 5.8233e-07 & 1.4571e-07 & 3.6443e-08 & 9.1134e-09 & 2.2796e-09\\
& Order & -- & 1.9988 & 1.9994 & 1.9996 & 1.9992\\
& CPU time & 2.38s & 4.35s & 8.67s & 17.32s & 34.54s\\ \midrule
\multirow{5}*{\textbf{ESAV-CN}}
& L$^2$ error & 1.0488e-06 & 2.6200e-07 & 6.5475e-08 & 1.6361e-08 & 4.0846e-09\\
& Order & -- & 2.0011 & 2.0005 & 2.0007 & 2.0020\\
& L$^\infty$ error & 1.6692e-07 & 4.1698e-08 & 1.0421e-08 & 2.6040e-09 & 6.5025e-10\\
& Order & -- & 2.0011 & 2.0005 & 2.0007 & 2.0017\\
& CPU time & 1.53s & 2.84s & 5.16s & 10.39s & 21.50s\\ \bottomrule
\end{tabularx}
\end{table}	

\begin{figure}[H]
\centering
\includegraphics[width=0.32\linewidth]{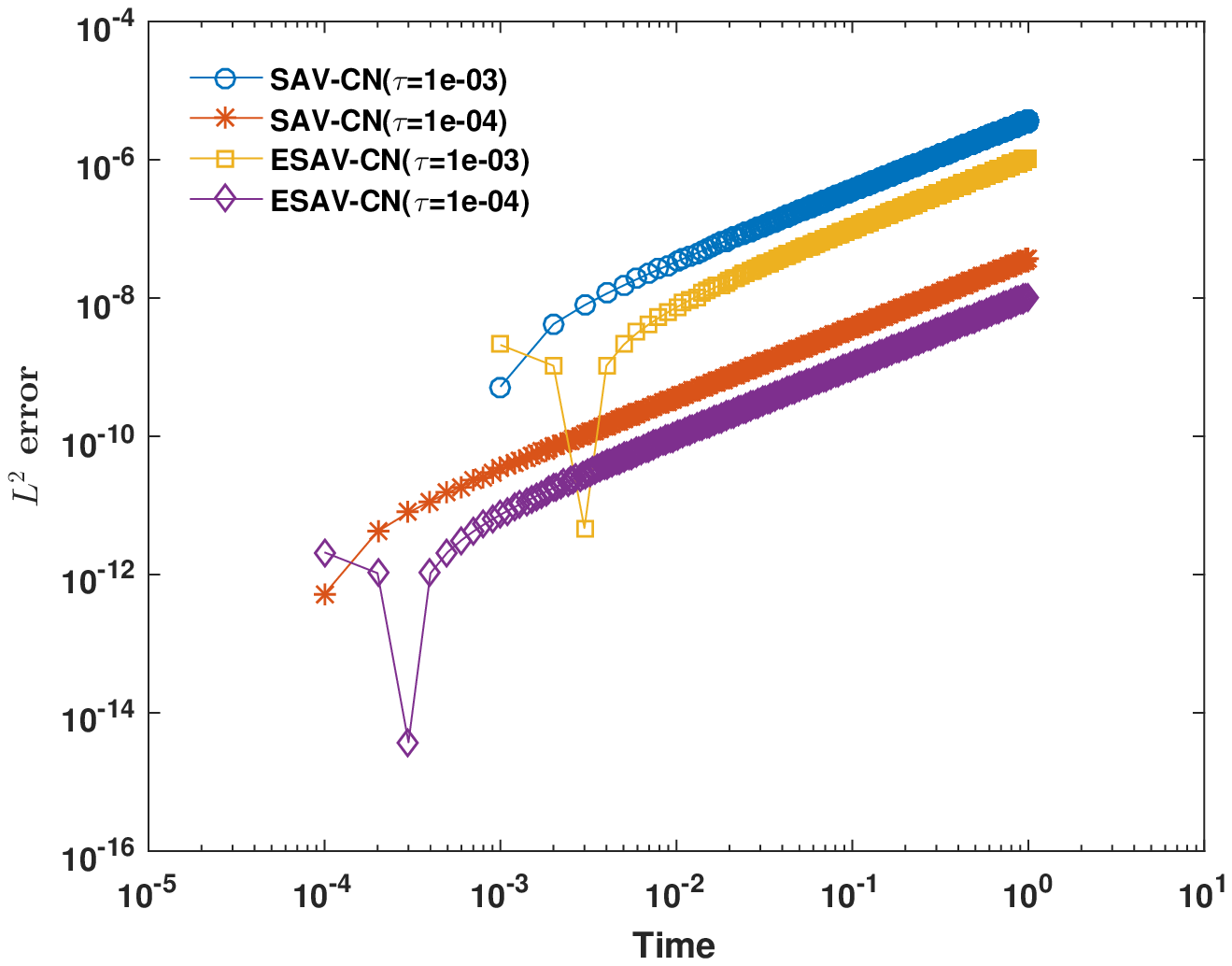}\hspace{-3.5mm}
\includegraphics[width=0.32\linewidth]{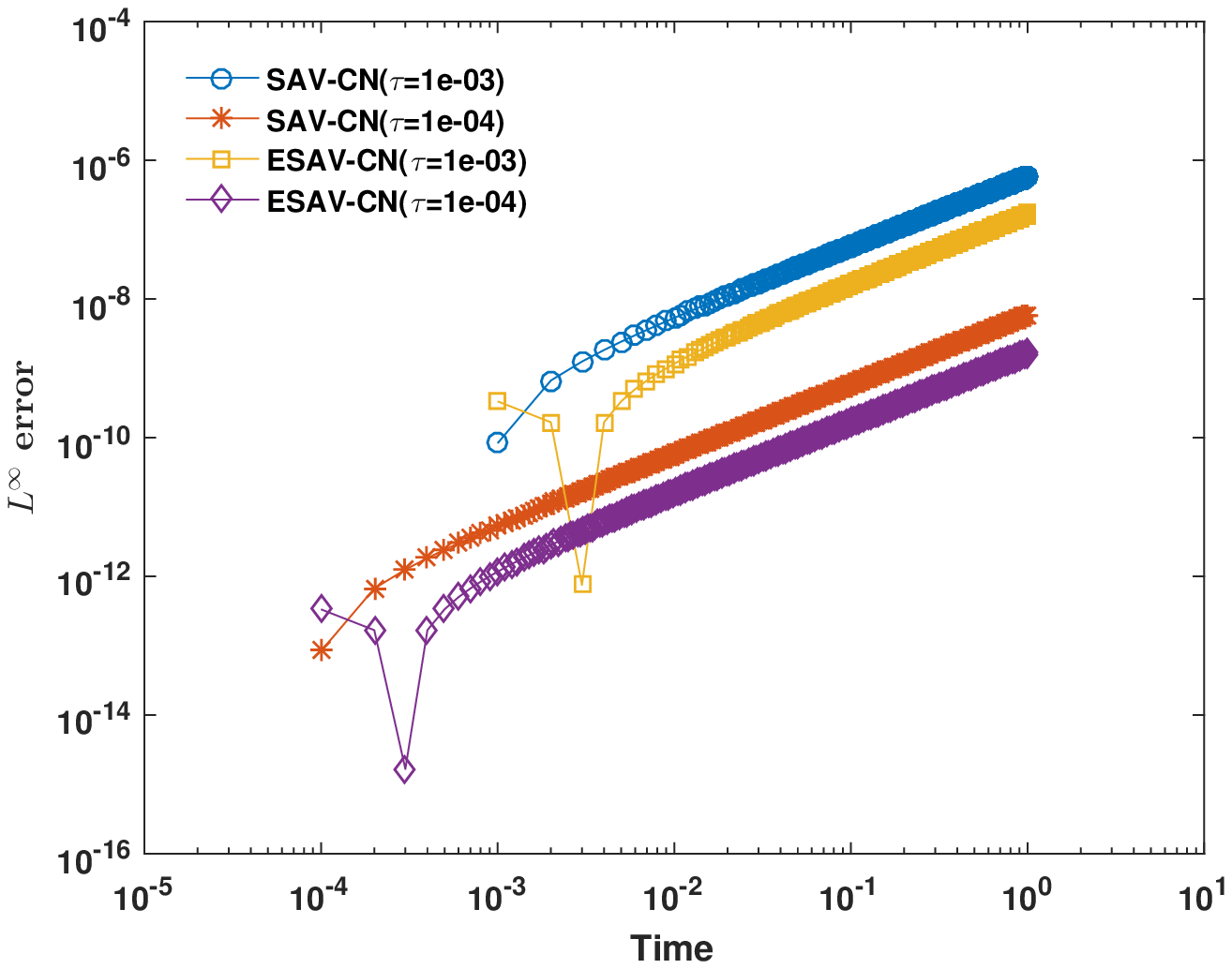}\hspace{-3.5mm}
\includegraphics[width=0.32\linewidth]{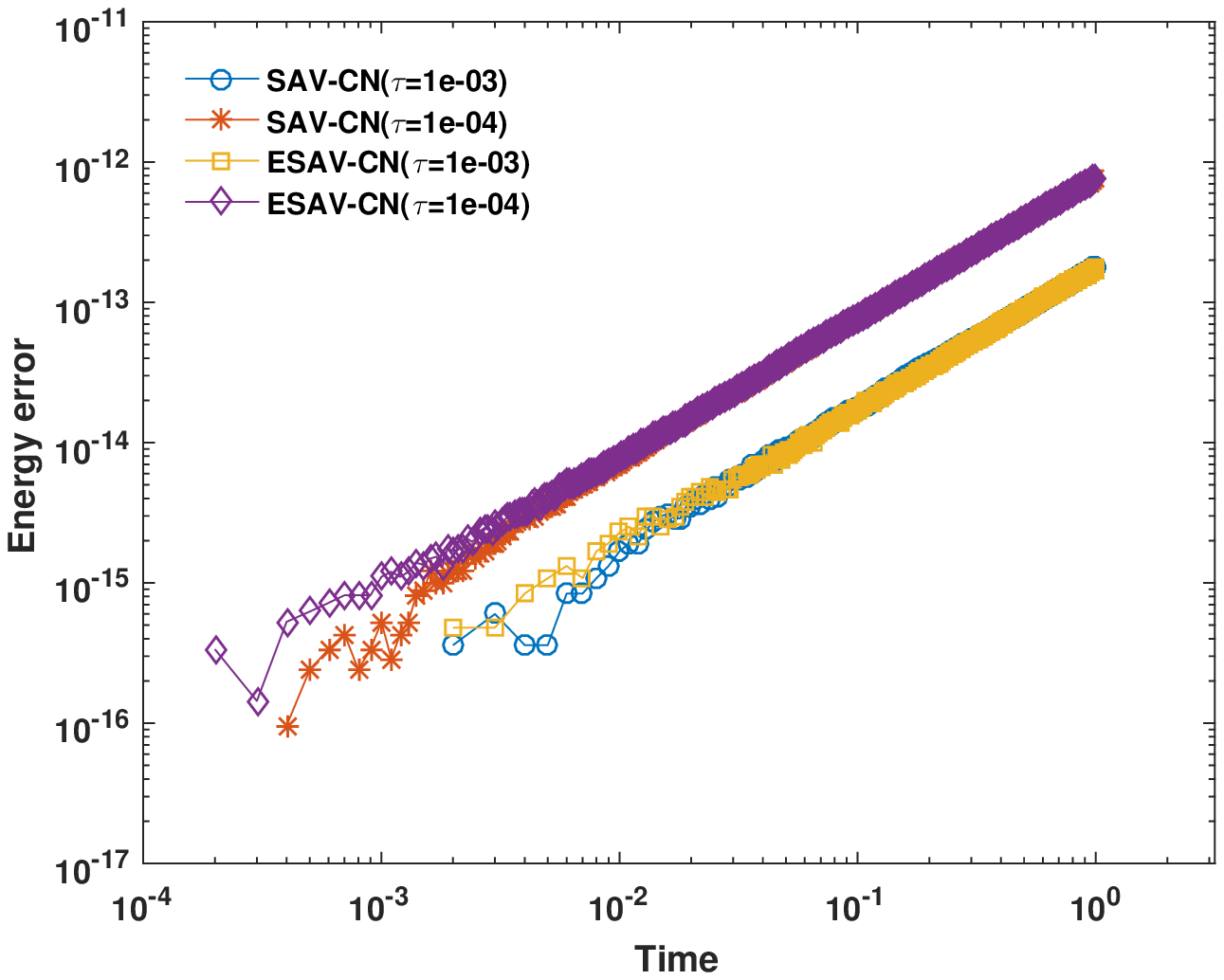}
\caption{Numerical error comparisons between \textbf{SAV-CN} and \textbf{ESAV-CN} with $N_x=N_y=64$ until $T=1$.}\label{Fig-1}
\end{figure}

\begin{figure}[H]
\centering
\includegraphics[width=0.42\linewidth]{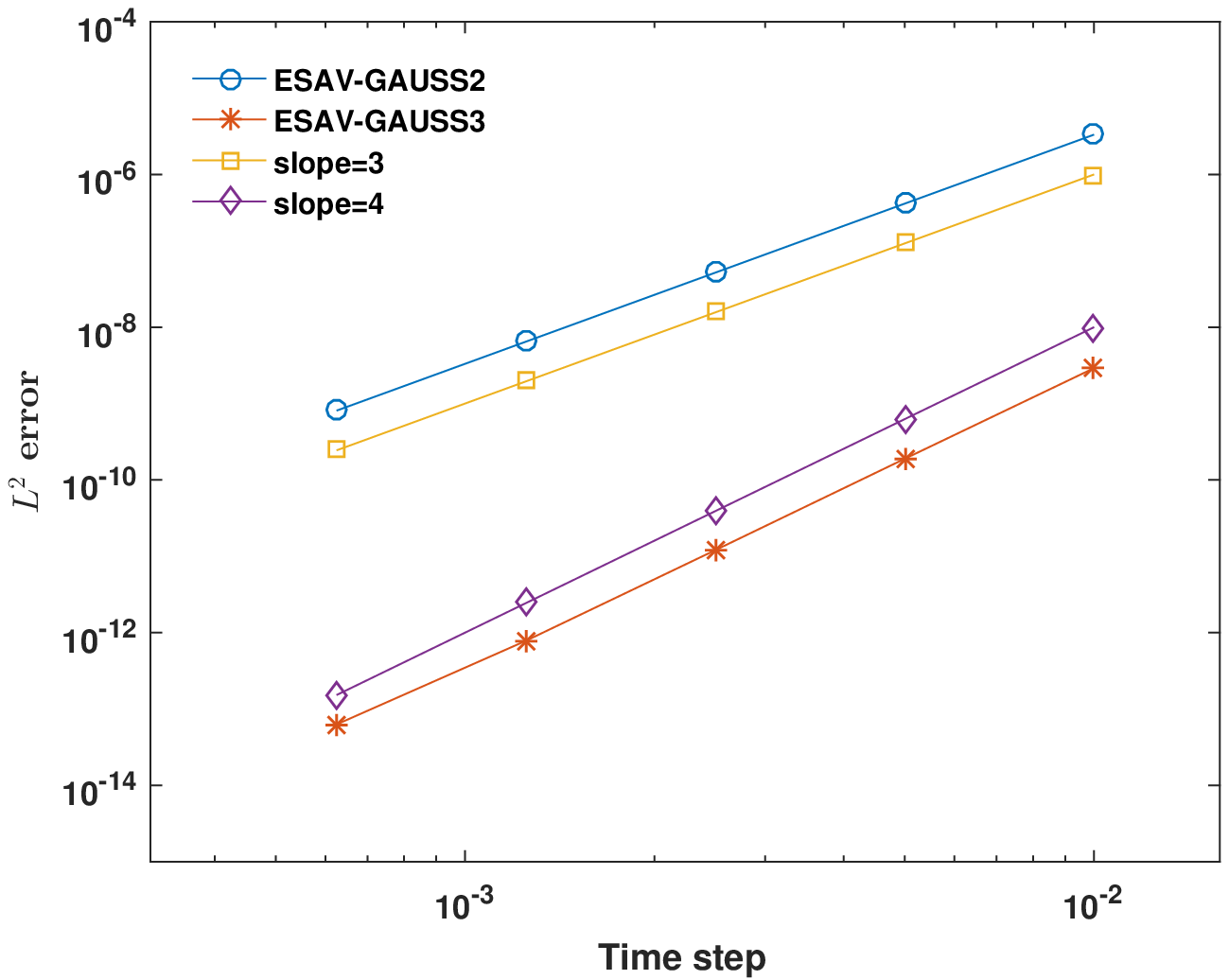}\hspace{-3.5mm}
\includegraphics[width=0.42\linewidth]{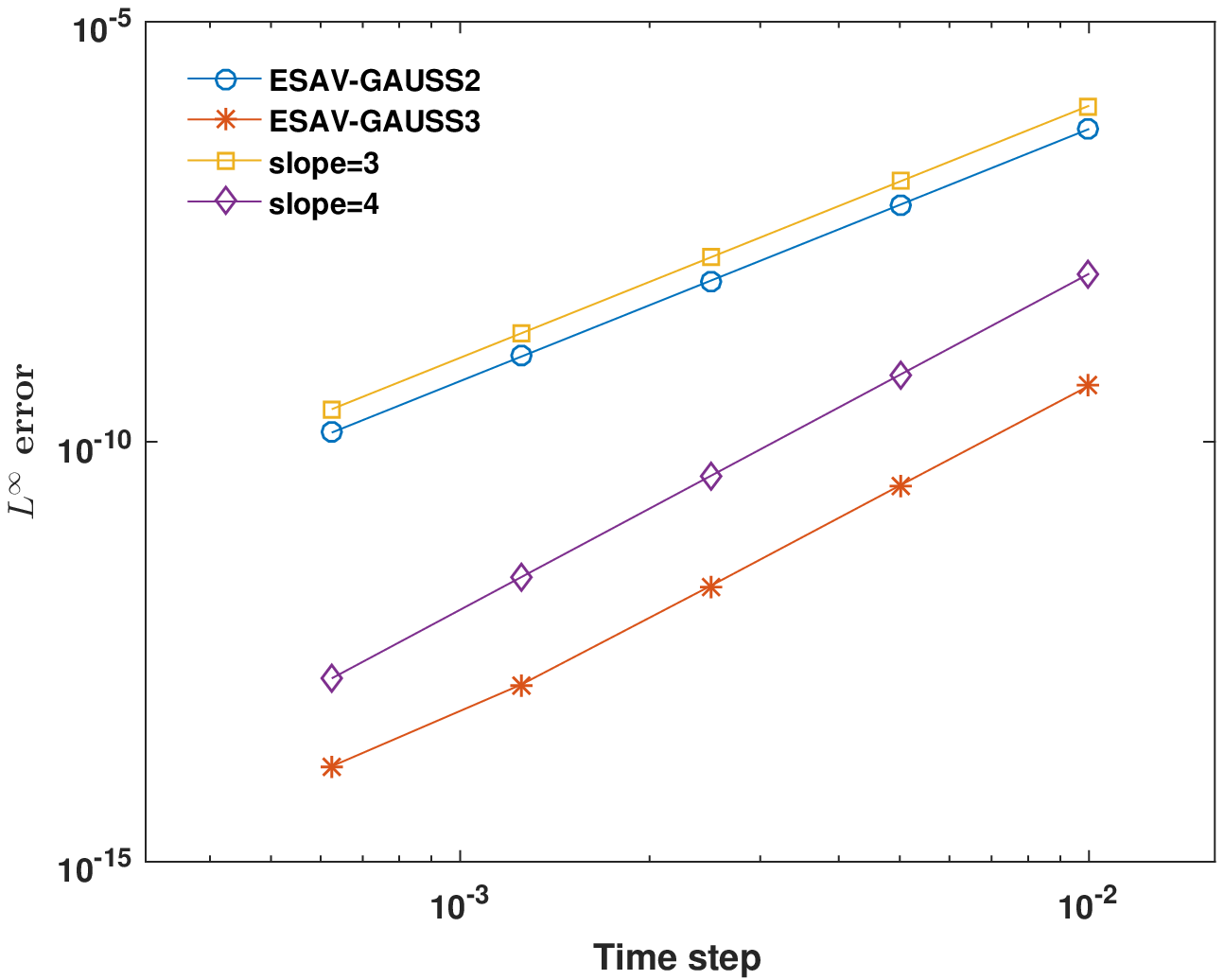}\\
\includegraphics[width=0.42\linewidth]{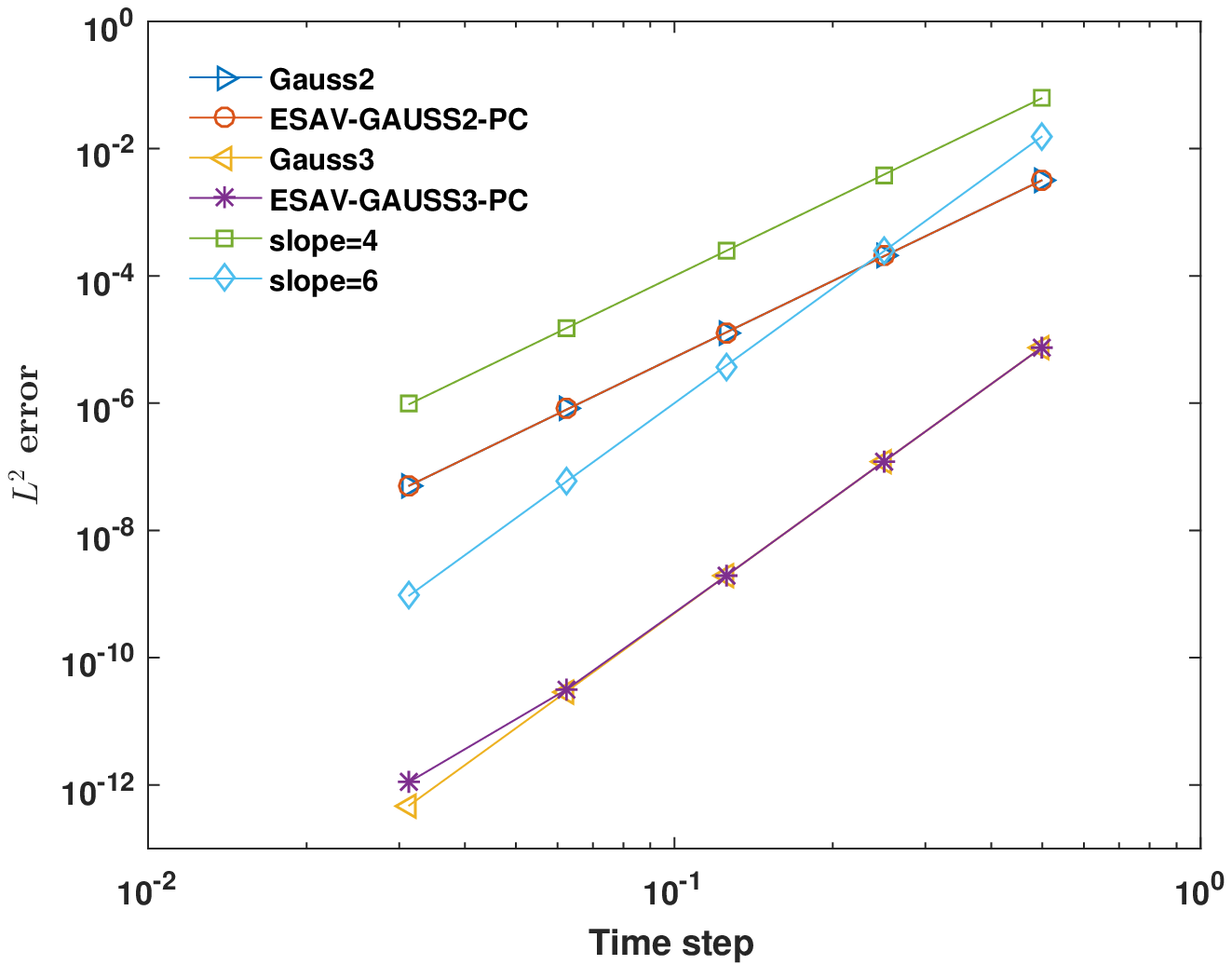}\hspace{-3.5mm}
\includegraphics[width=0.42\linewidth]{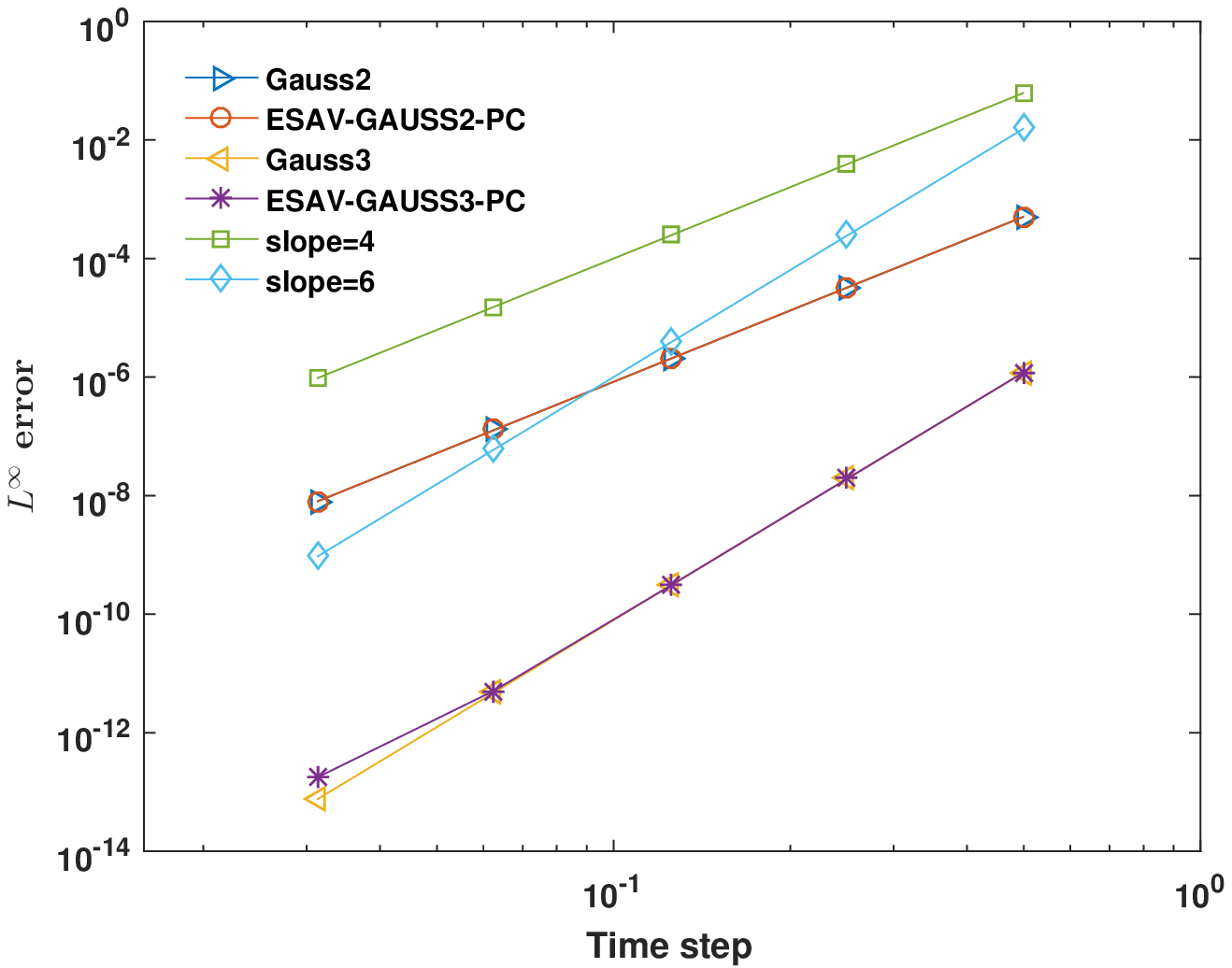}\vspace{-1mm}
\caption{Temporal convergence rates of \textbf{Gauss}, \textbf{ESAV-Guass} and \textbf{ESAV-Gauss-PC} with $N_x=N_y=8$ at $T=1$.}\label{Fig-2}
\end{figure}
\vspace{-4mm}
\begin{table}[H]
\tabcolsep=23pt \small \renewcommand \arraystretch{1.3} \centering
\caption{Comparisons of maximum iteration numbers between \textbf{Gauss} and \textbf{ESAV-Gauss-PC} at different temporal step sizes with $N_x=N_y=8$ until $T=100$.}\label{Tab-4}\vspace{-2mm}
\begin{tabularx}{\textwidth}{*{6}{l}} \toprule
Schemes & 0.2 & 0.2/2 & 0.2/4  & 0.2/8 & 0.2/16 \\ \midrule 
\multirow{1}*{\textbf{Gauss2}}
& 15 & 12 & 11 & 9 & 8\\ 
\multirow{1}*{\textbf{ESAV-Gauss2-PC}}
& 10 & 8 & 6 & 5 & 4\\ \midrule
\multirow{1}*{\textbf{Gauss3}}
& 14 & 12 & 10 & 9 & 8\\ 
\multirow{1}*{\textbf{ESAV-Gauss3-PC}}
& 9 & 6 & 5 & 4 & 3\\ \bottomrule
\end{tabularx}
\end{table}	 
\vspace{-4mm}
\begin{figure}[H]
\centering
\includegraphics[width=0.42\linewidth]{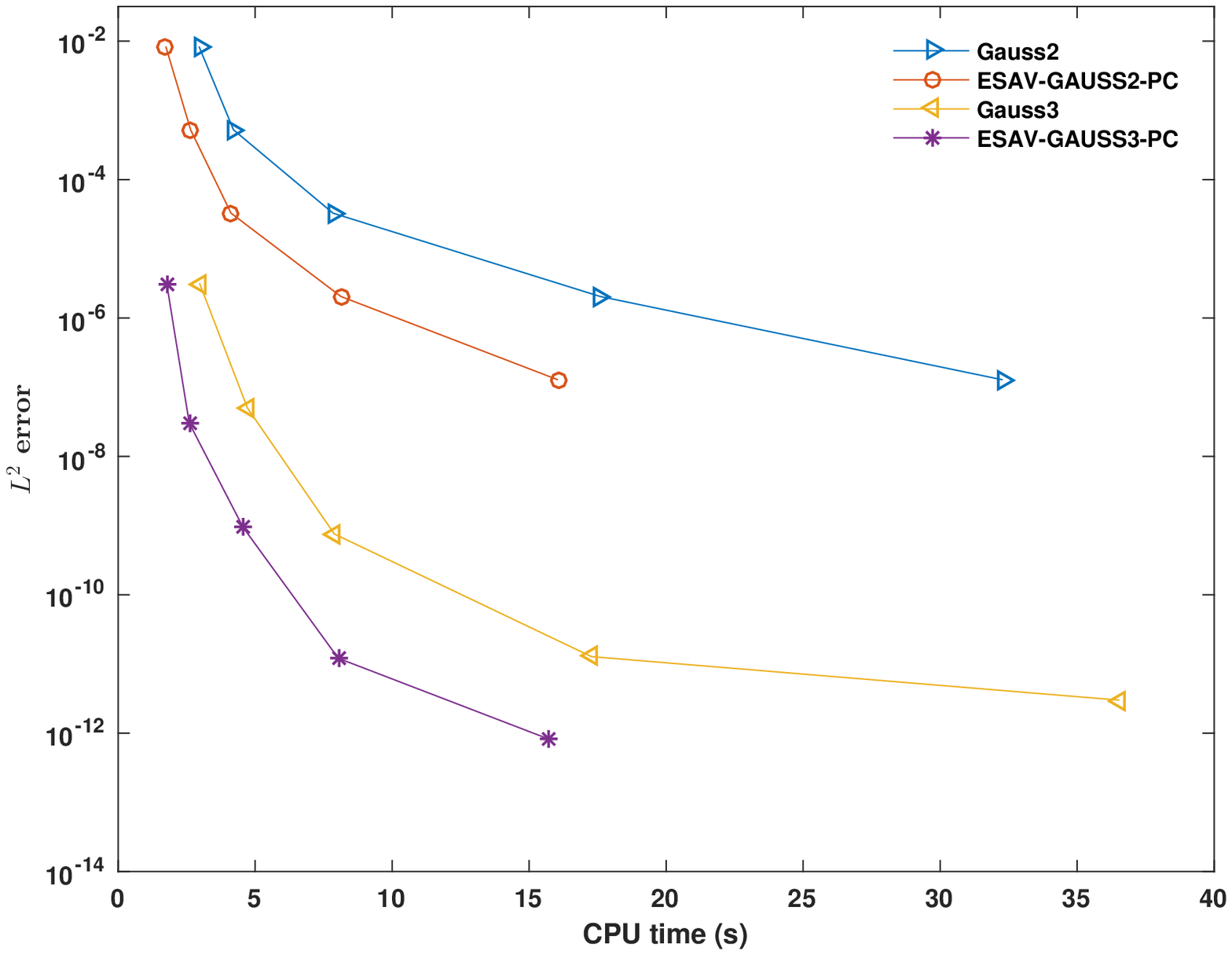}\hspace{-3.5mm}
\includegraphics[width=0.42\linewidth]{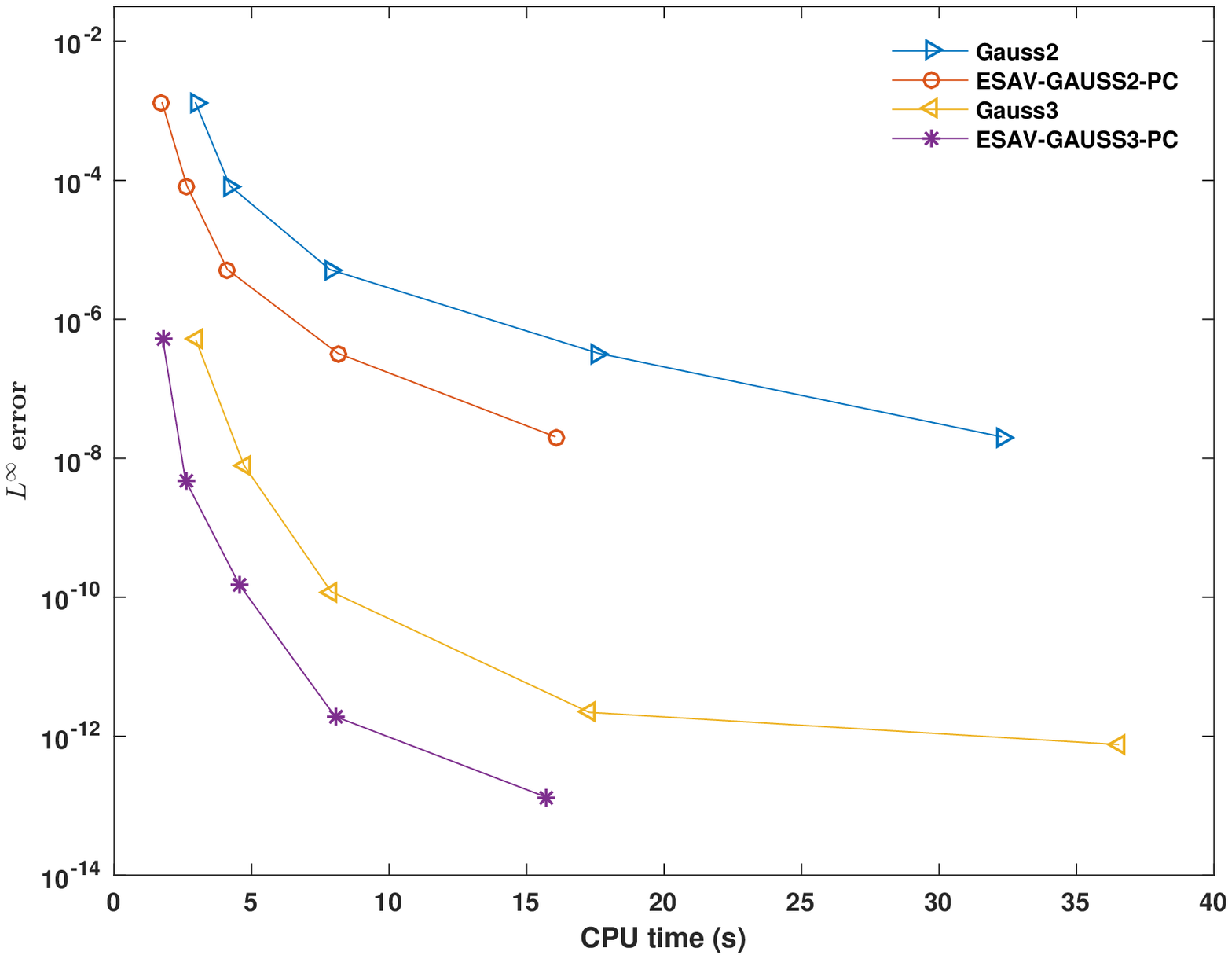}\vspace{-1mm}	
\caption{Comparisons of CPU times between \textbf{Gauss} and \textbf{ESAV-Gauss-PC} with $N_x=N_y=8$ until $T=100$.}\label{Fig-3}
\end{figure}

\noindent \textbf{Example 5.2.}~(The NLS equation) In this example, we consider the NLS equation \eqref{eq-5-1} that has a singular solution \cite{gong-17-FP-NLS-JCP} developed at $t=0.108$ with $\beta=1$. The initial condition is selected as
\begin{equation}\label{eq-5-4}
u(x,y,0)=\big(1+\mbox{sin}(x)\big)\big(2+\mbox{sin}(y)\big).
\end{equation} 
Our numerical experiments are carried out on $\Omega=[0,2\pi)\times[0,2\pi)$ with $(2\pi,2\pi)$-periodic boundary condition. Fig. \ref{Fig-4} plots the singular solutions of \textbf{SAV-CN} and \textbf{ESAV-CN} and the corresponding contours, while these numerical plots completely matches the one in \cite{gong-17-FP-NLS-JCP}. The numerical trajectories of the higher-order schemes are almost the same as Fig. \ref{Fig-4}, and hereafter we only take these two second-order schemes as examples. For the simulation until time $T=1$ in Fig. \ref{Fig-5}, \textbf{SAV-CN} present the wrong behavior with $N_x=N_y=128$ and the time step $\tau=0.0001$ after $t=0.2$, although other schemes still maintain the modified energy better. This also reflects that other schemes have better stability than \textbf{SAV-CN} for this example.
\begin{figure}[H]
\centering
\includegraphics[width=0.32\linewidth]{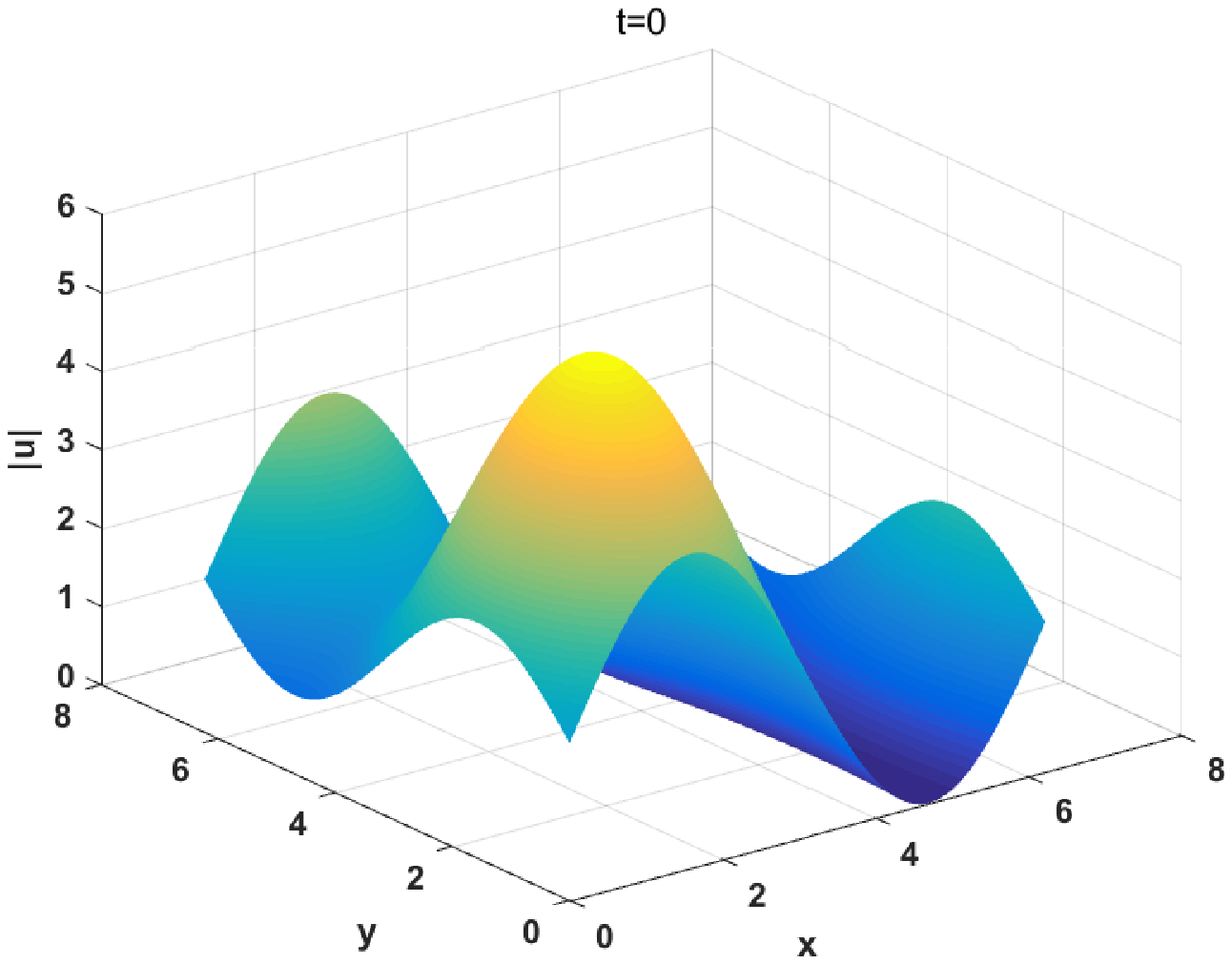}\hspace{-3.5mm}
\includegraphics[width=0.32\linewidth]{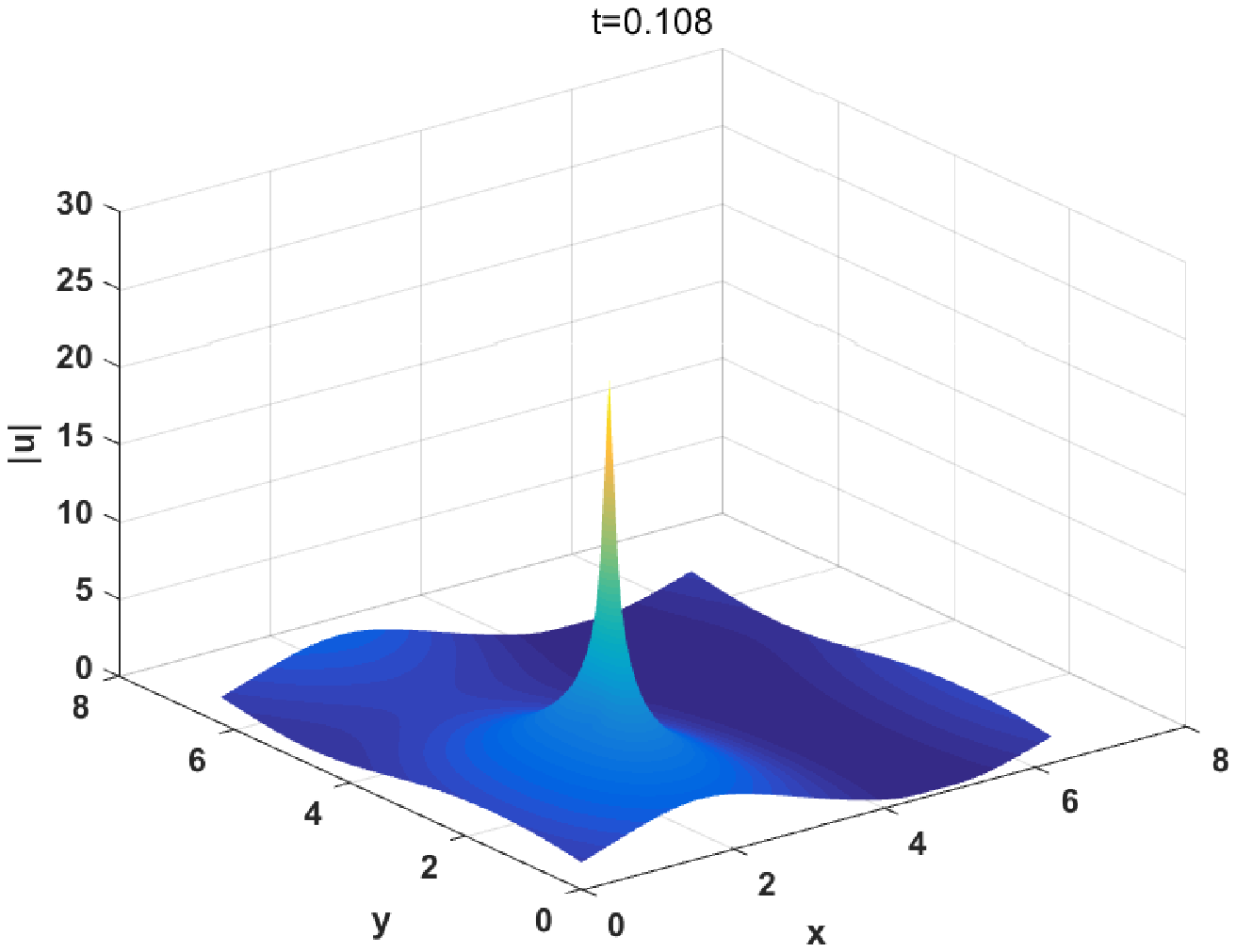}\hspace{-3.5mm}
\includegraphics[width=0.32\linewidth]{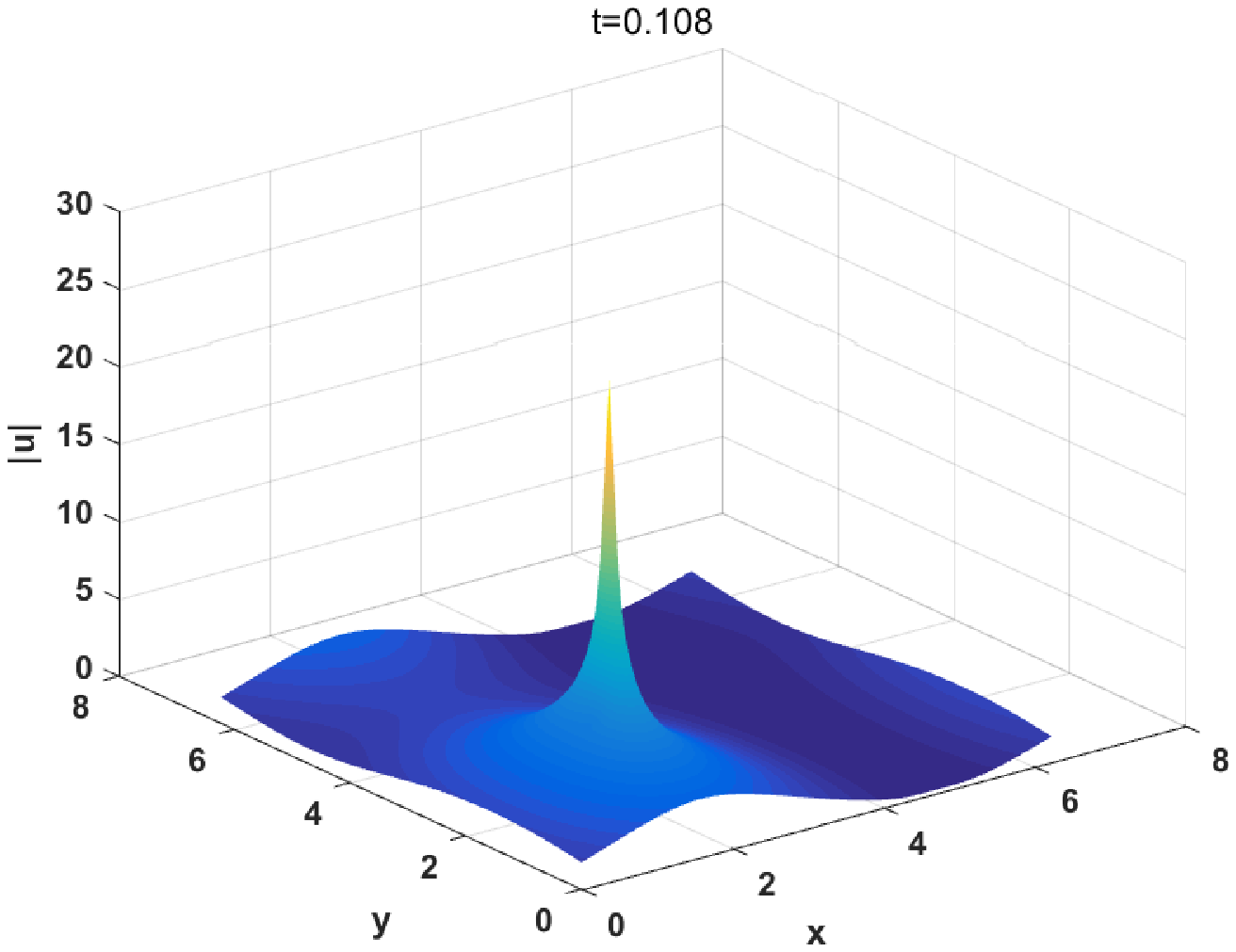}
\includegraphics[width=0.32\linewidth]{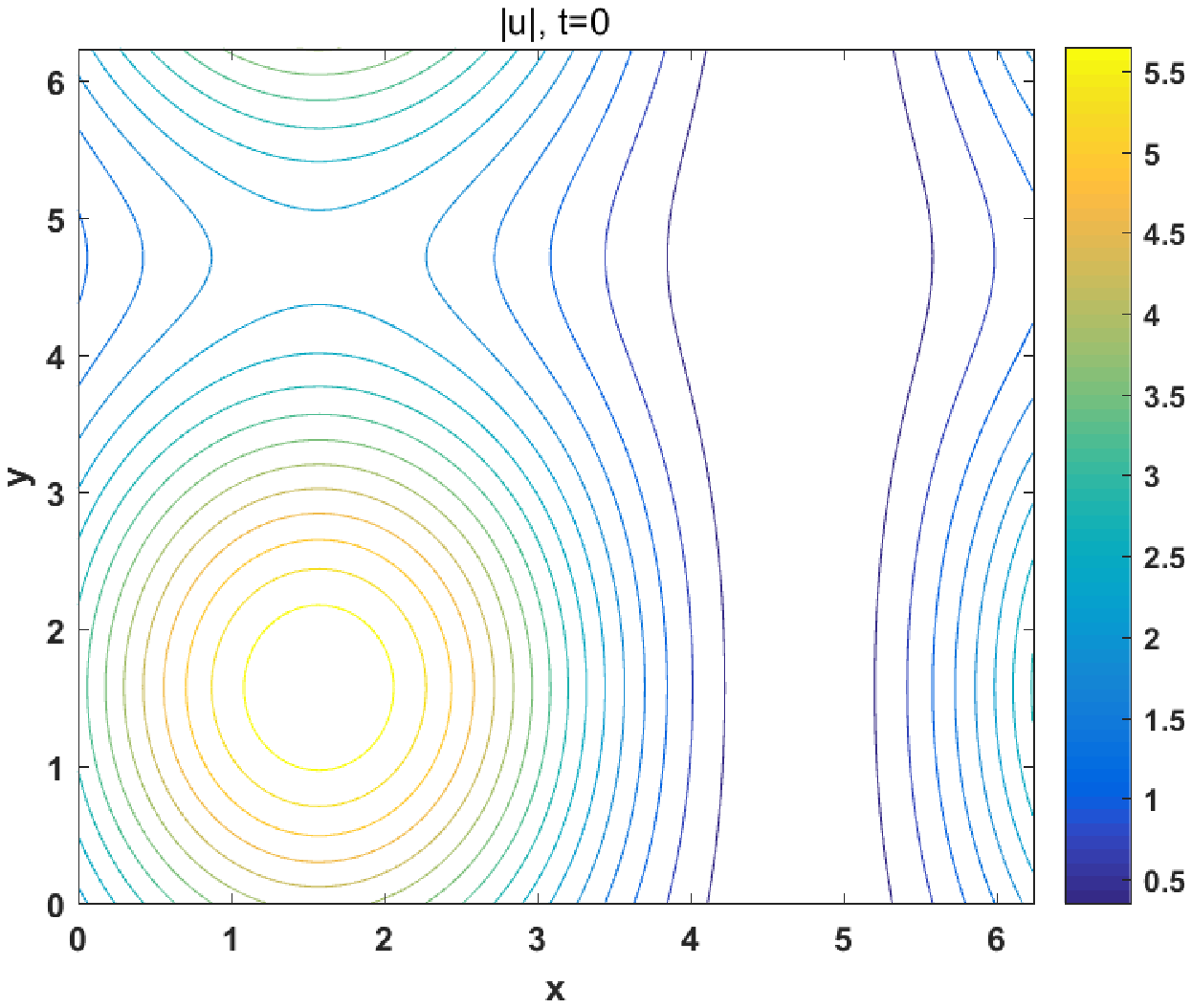}\hspace{-3.5mm}
\includegraphics[width=0.32\linewidth]{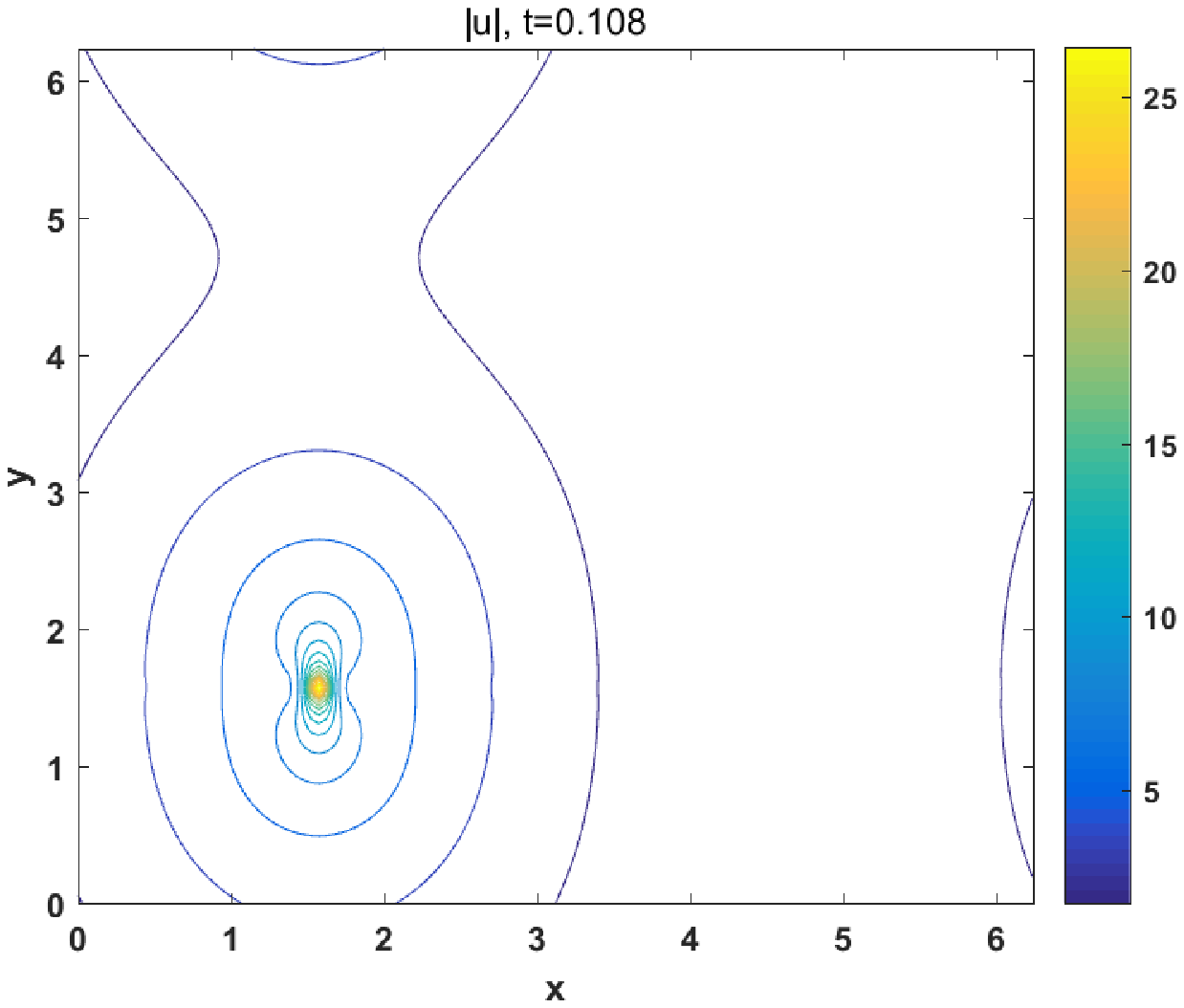}\hspace{-3.5mm}
\includegraphics[width=0.32\linewidth]{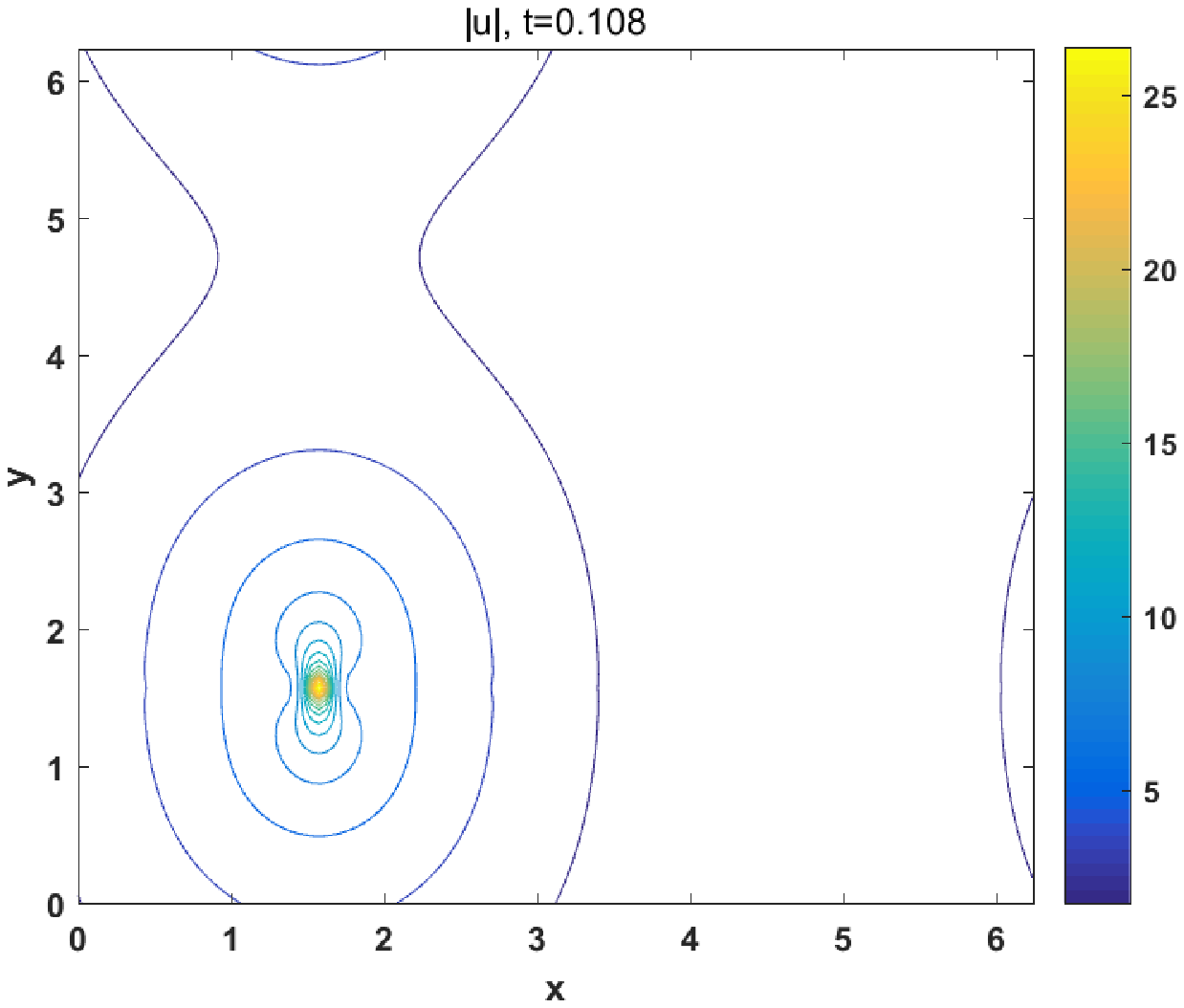}	
\caption{Surfaces and contours of modulus of the initial condition and singular solutions for \textbf{SAV-CN} (second column) and \textbf{ESAV-CN} (third column) with $N_x=N_y=128$ and $\tau=0.0001$.}\label{Fig-4}
\end{figure}

\begin{figure}[H]
\centering
\includegraphics[width=0.42\linewidth]{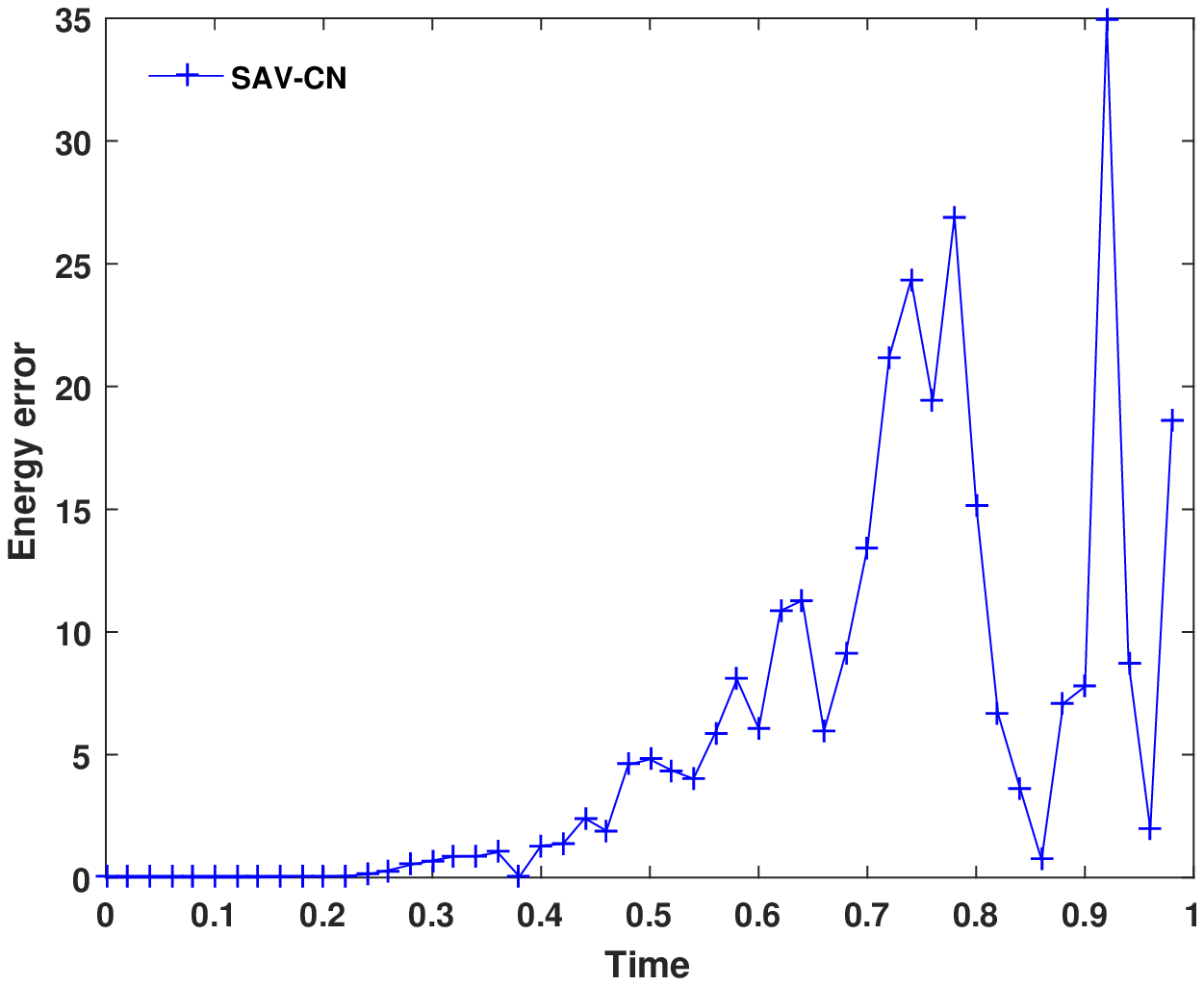}\hspace{-3.5mm}
\includegraphics[width=0.42\linewidth]{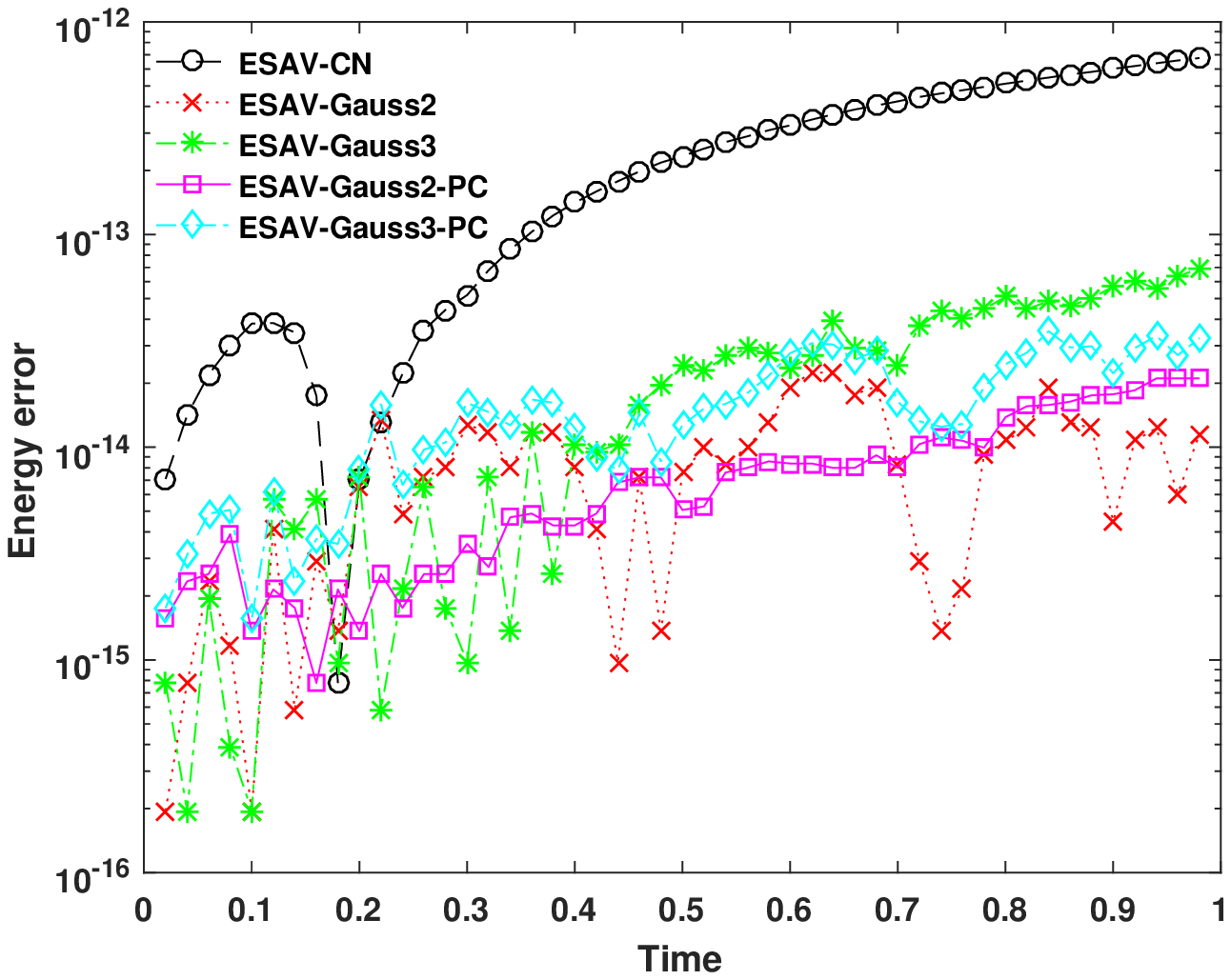}
\caption{The modified energy errors of \textbf{SAV-CN}~(left)~and other schemes~(right)~with $N_x=N_y=128$ and $\tau=0.0001$ until $T=1$.}\label{Fig-5}
\end{figure}

\noindent \textbf{Example 5.3.}~(The SG equation) We test the the SG equation in the form
\begin{equation}\label{eq-5-5}
u_{tt}-\triangle u+\phi(x,y)\ \mbox{sin}\left(u\right)=0,\quad (x,y)\in\Omega\in\mathcal{R}^2,\ t\in (0,T]
\end{equation} 
with $\phi(x,y)=1$ and periodic boundary conditions. The following two different initial conditions are carried out by selecting
\begin{itemize}
\item circular ring soliton \cite{jiang-19-SG-IEQ-JSC,cai-19-SG-NB-JCP}~(see Fig. \ref{Fig-6}) :
\begin{align}\label{eq-5-6}
\aligned
&u(x,y,0)=4\ \mbox{tan}^{-1}\bigg(\mbox{exp}\left(3-\sqrt{x^2+y^2}\right)\bigg),\quad \Omega=[-7,7)\times[-7,7),\\
&u_t(x,y,0)=0,
\endaligned
\end{align}
\item collision of four circular solitons \cite{jiang-19-SG-IEQ-JSC,cai-19-SG-NB-JCP}~(see Fig. \ref{Fig-7}):
\begin{align}\label{eq-5-7}
\aligned
&u(x,y,0)=4\ \mbox{tan}^{-1}\bigg[\mbox{exp}\left(3-\sqrt{(x+3)^2+(y+7)^2}\right)/0.436\bigg],\quad \Omega=[-30,10)\times[-30,10),\\
&u_t(x,y,0)=4.13\ \mbox{sech}\bigg[\mbox{exp}\left(3-\sqrt{(x+3)^2+(y+7)^2}\right)/0.436\bigg].
\endaligned
\end{align}
\end{itemize} 

Fig. \ref{Fig-6} and Fig. \ref{Fig-7} depict the initial conditions and the evolution of the two different soliton waves in terms of $\mbox{sin}(u/2)$ at different times. Here we only present the numerical solutions of \textbf{ESAV-CN}, while other schemes are similar. The obtained numerical solutions are in good agreement with the results \cite{jiang-19-SG-IEQ-JSC,cai-19-SG-NB-JCP}. In Fig. \ref{Fig-7}, the simulation is performed to $12.5$ seconds with the reproduction of symmetry conditions along the lines $x=-10$ and $y=10$. \textbf{ESAV-CN} correctly characterizes the collision of four expanding circular ring solitons and reflects an extremely complex interaction with rapidly varying values of $u$ in the center. Subsequently, in Fig. \ref{Fig-8}, by virtue of the fast implementation of linearly implicit schemes, we sketch the energy-preserving evolution of all the schemes under these two different initial conditions over a long time interval.
\begin{figure}[H]
\centering
\includegraphics[width=0.32\linewidth]{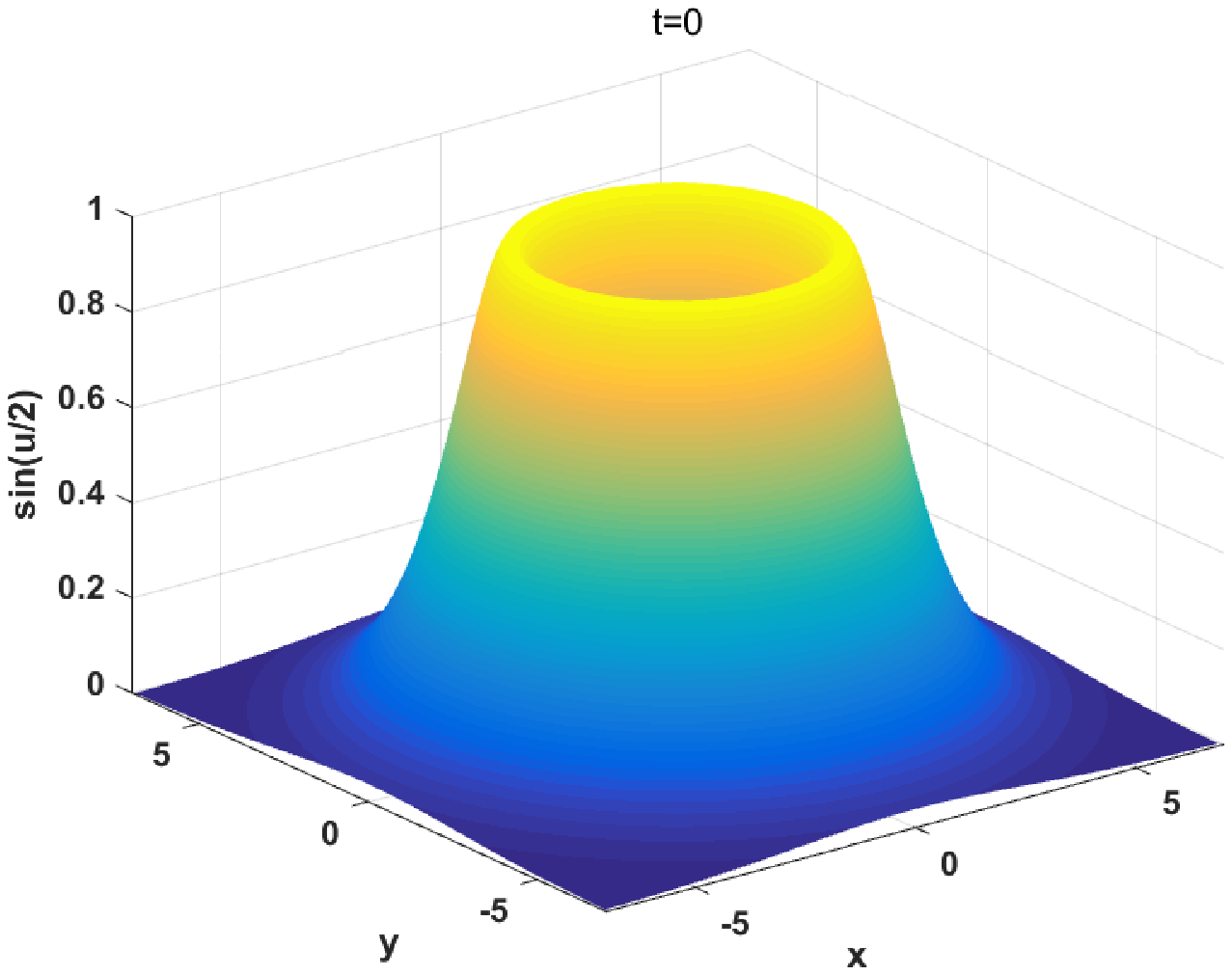}\hspace{-3.5mm}
\includegraphics[width=0.32\linewidth]{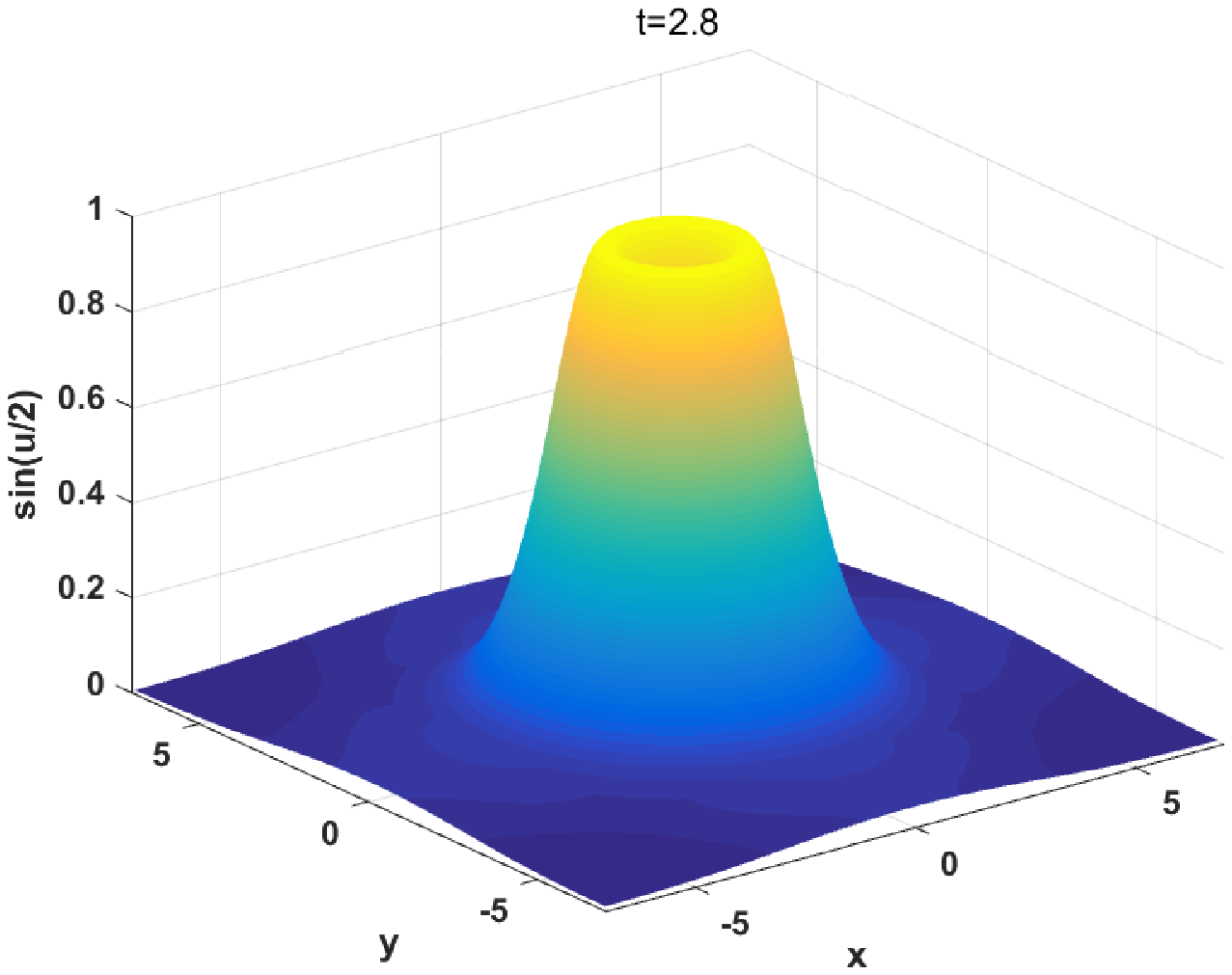}\hspace{-3.5mm}
\includegraphics[width=0.32\linewidth]{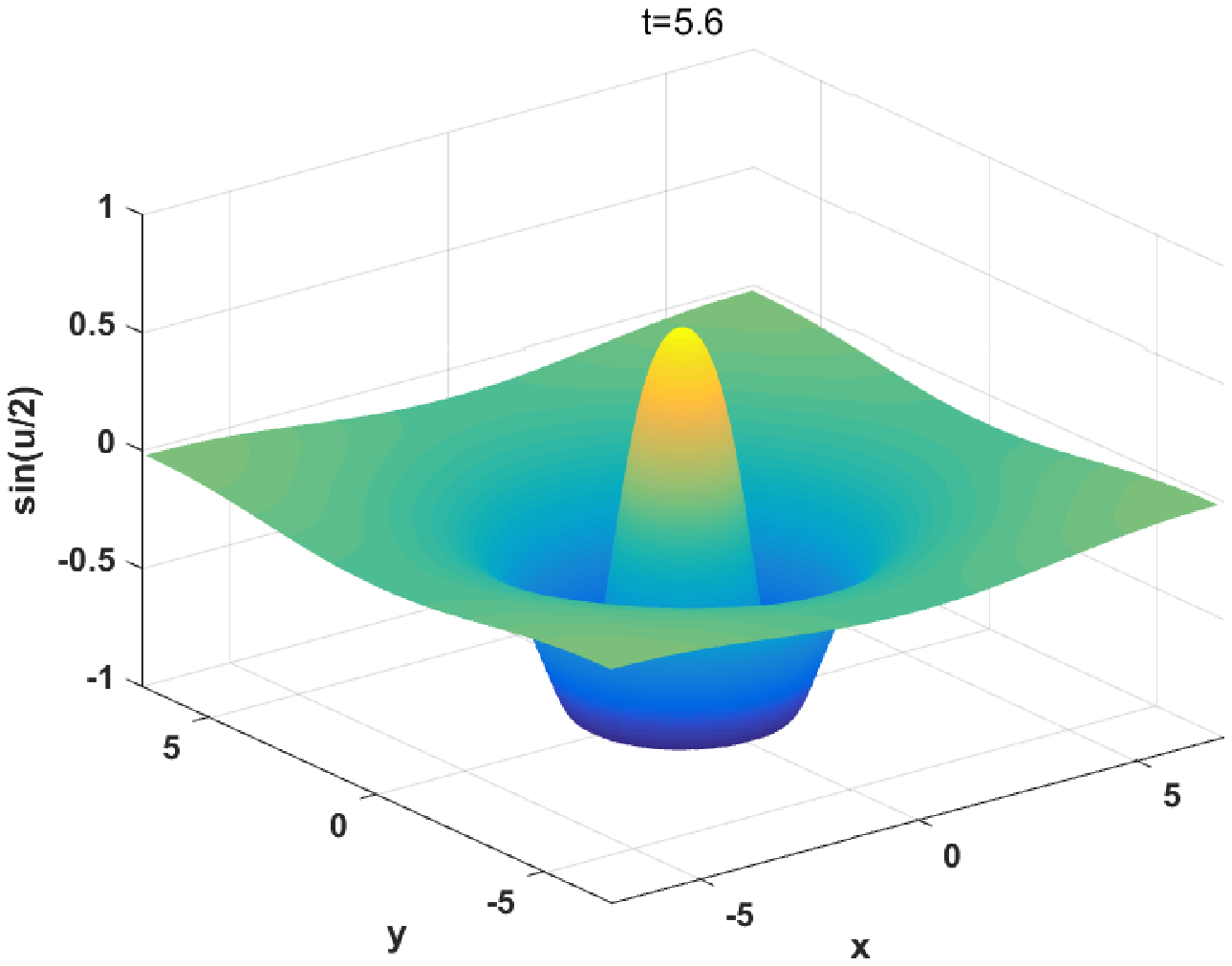}
\includegraphics[width=0.32\linewidth]{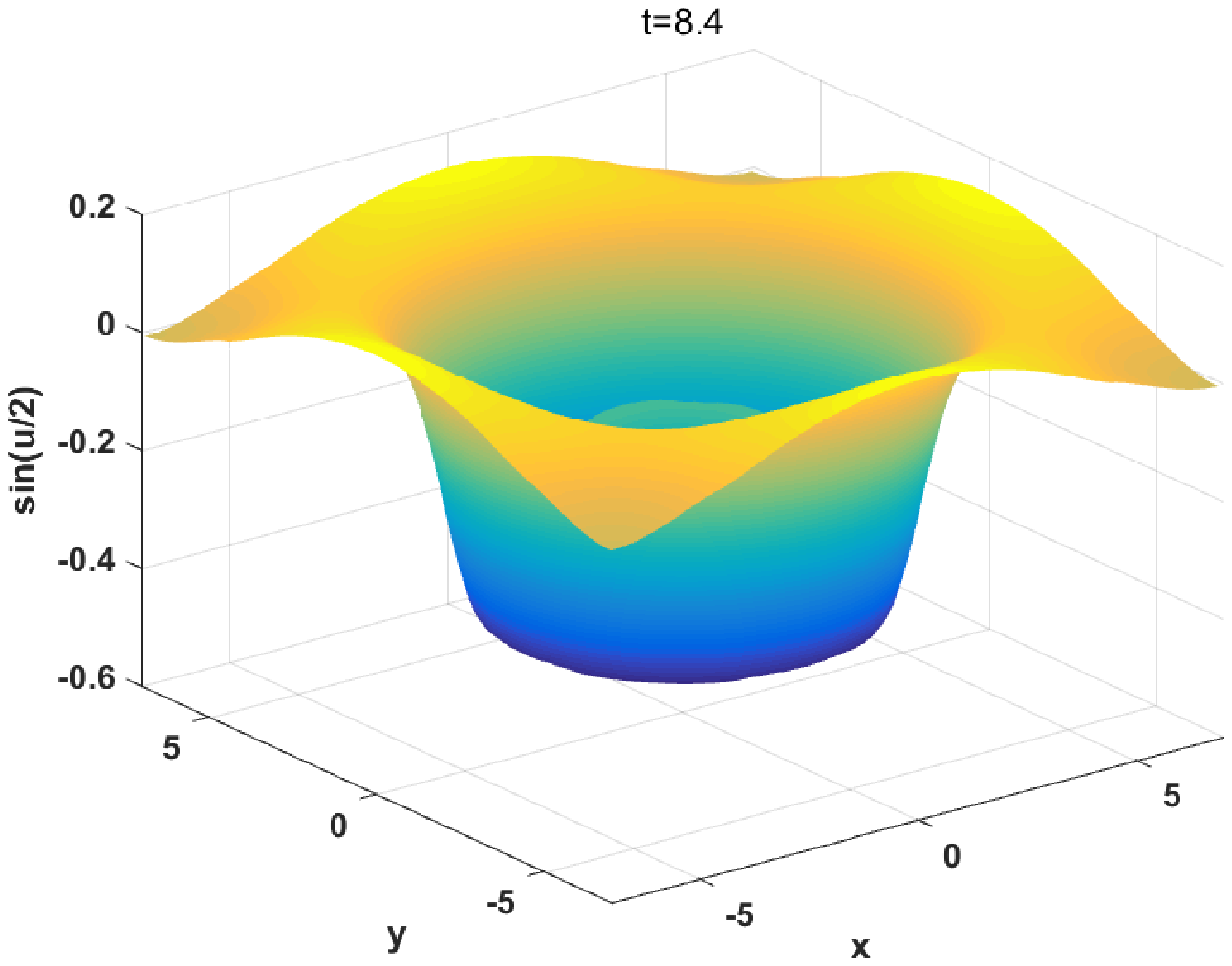}\hspace{-3.5mm}
\includegraphics[width=0.32\linewidth]{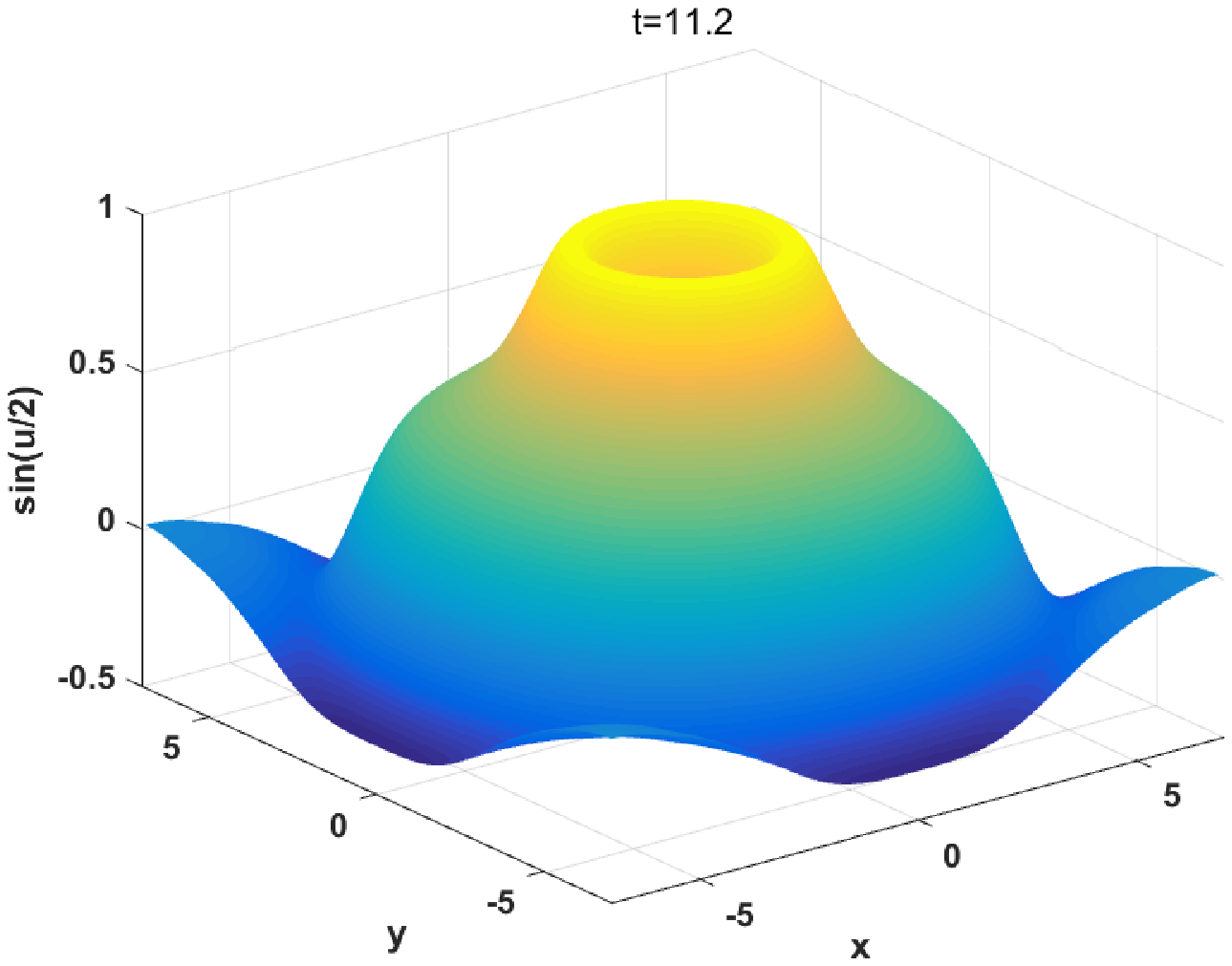}\hspace{-3.5mm}
\includegraphics[width=0.32\linewidth]{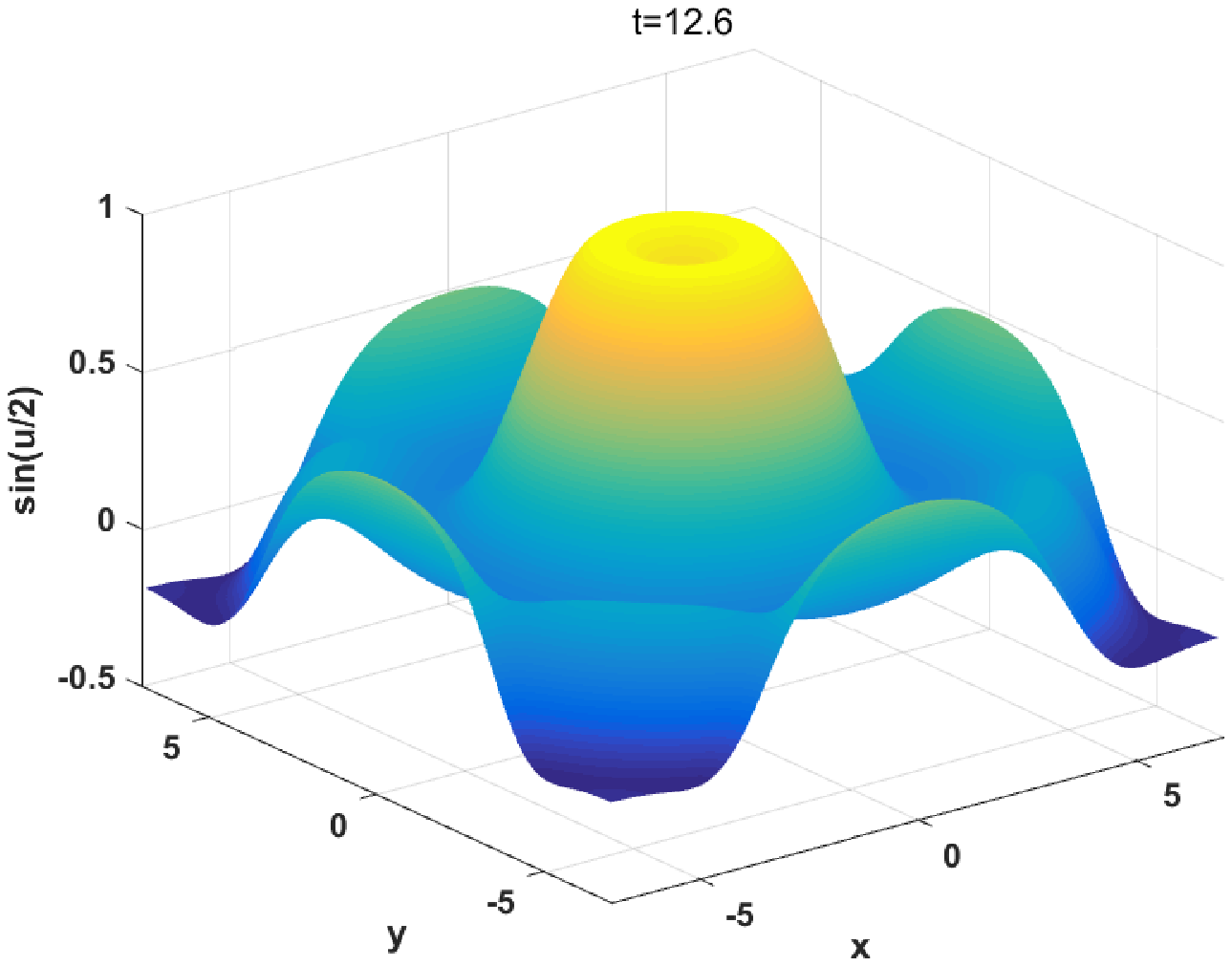}
\caption{Circular ring soliton: the initial condition and numerical solutions of \textbf{ESAV-CN} in terms of $\mbox{sin}(u/2)$ at different times with $N_x=N_y=128$ and $\tau=0.01$.}\label{Fig-6}
\end{figure}

\begin{figure}[H]
\centering
\includegraphics[width=0.32\linewidth]{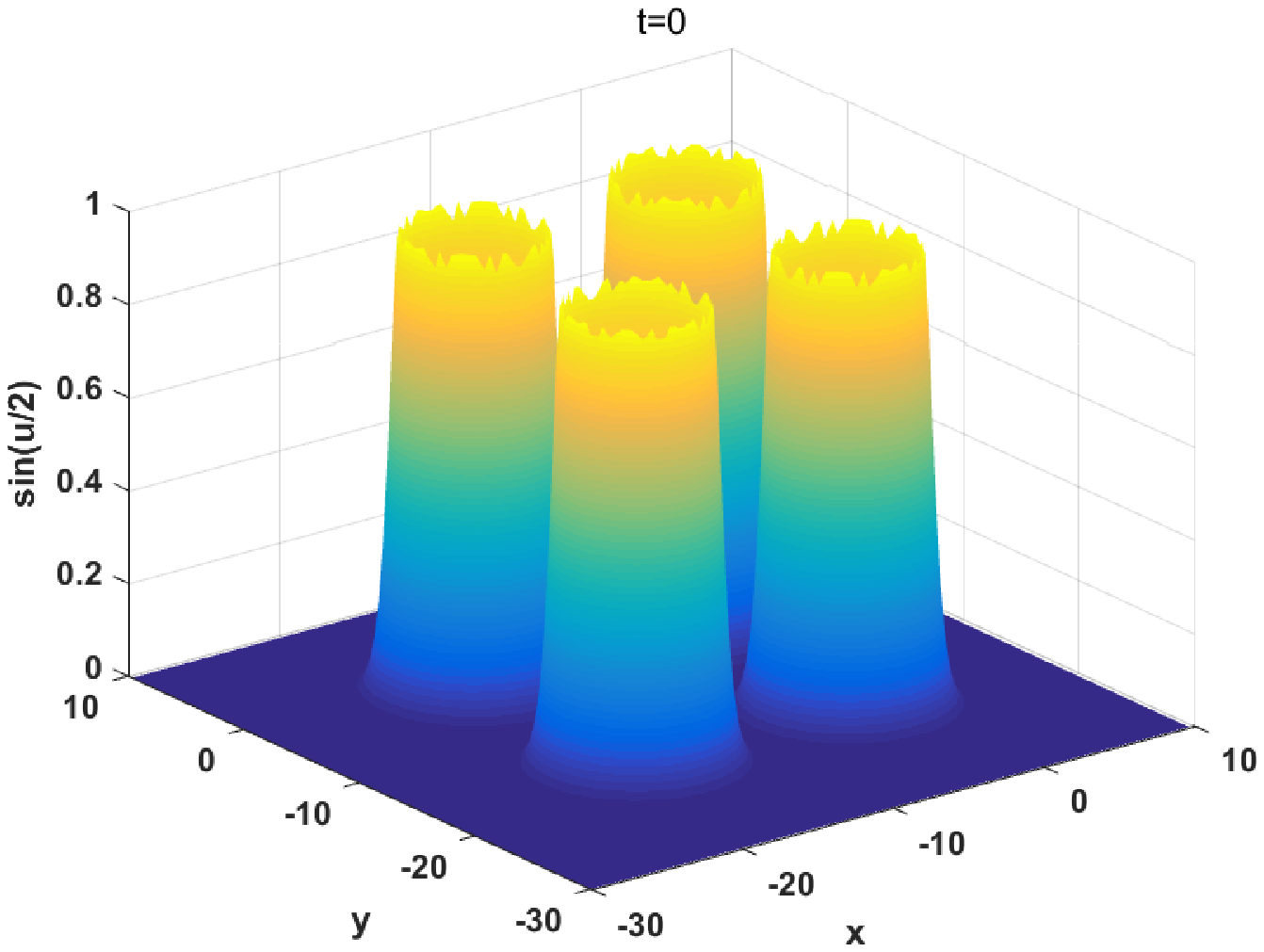}\hspace{-3.5mm}
\includegraphics[width=0.32\linewidth]{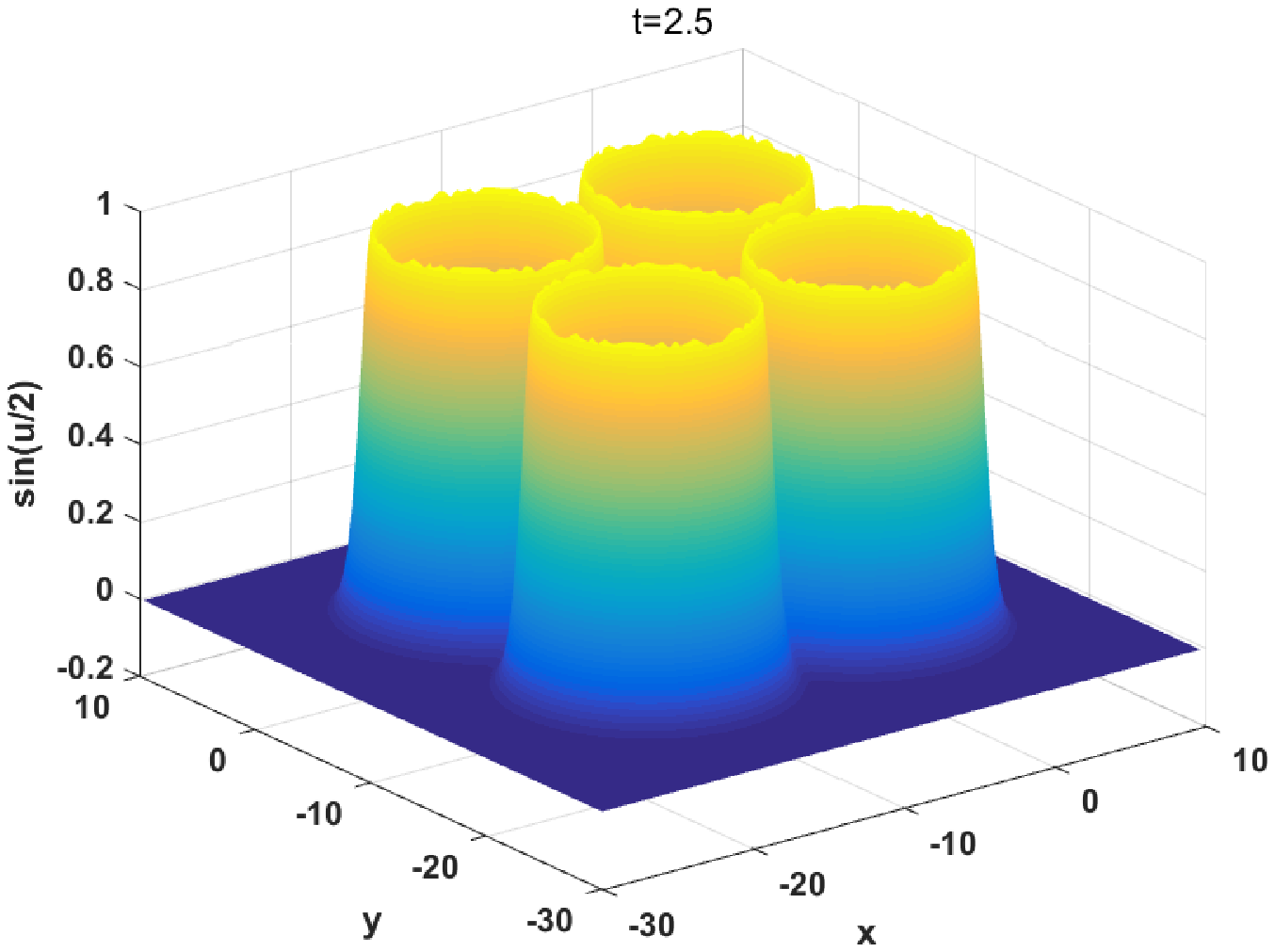}\\
\includegraphics[width=0.32\linewidth]{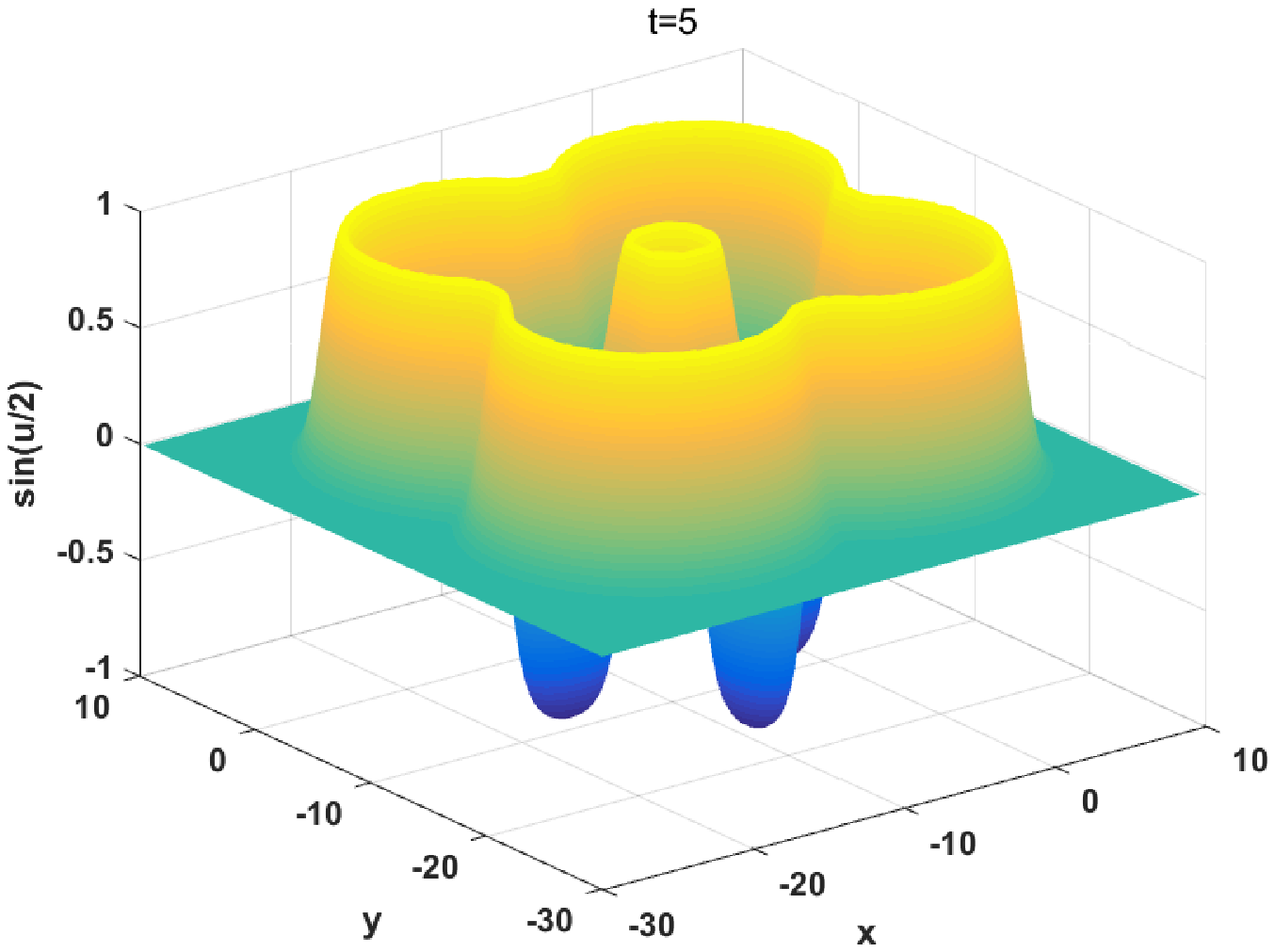}\hspace{-3.5mm}
\includegraphics[width=0.32\linewidth]{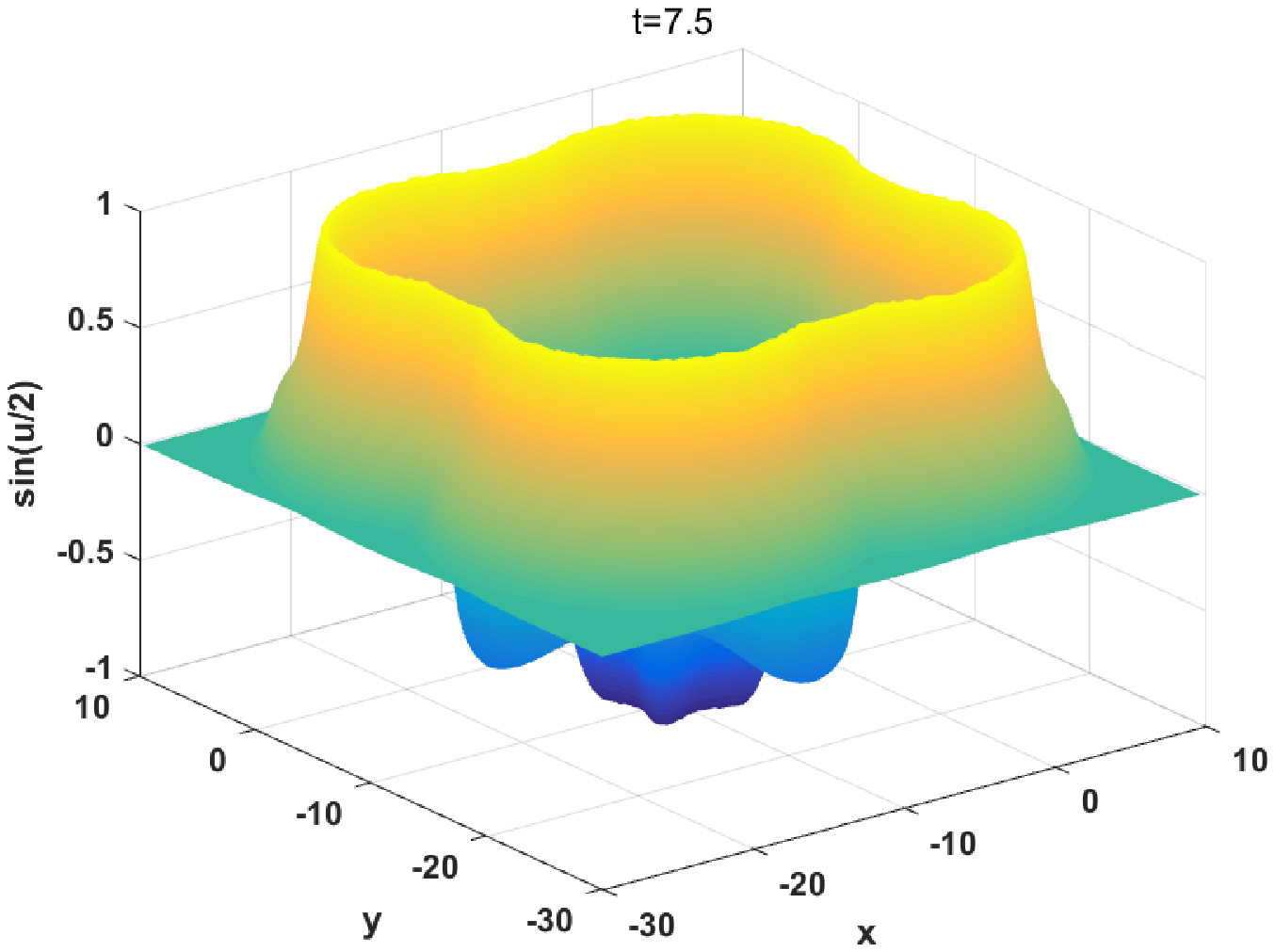}\hspace{-3.5mm}
\includegraphics[width=0.32\linewidth]{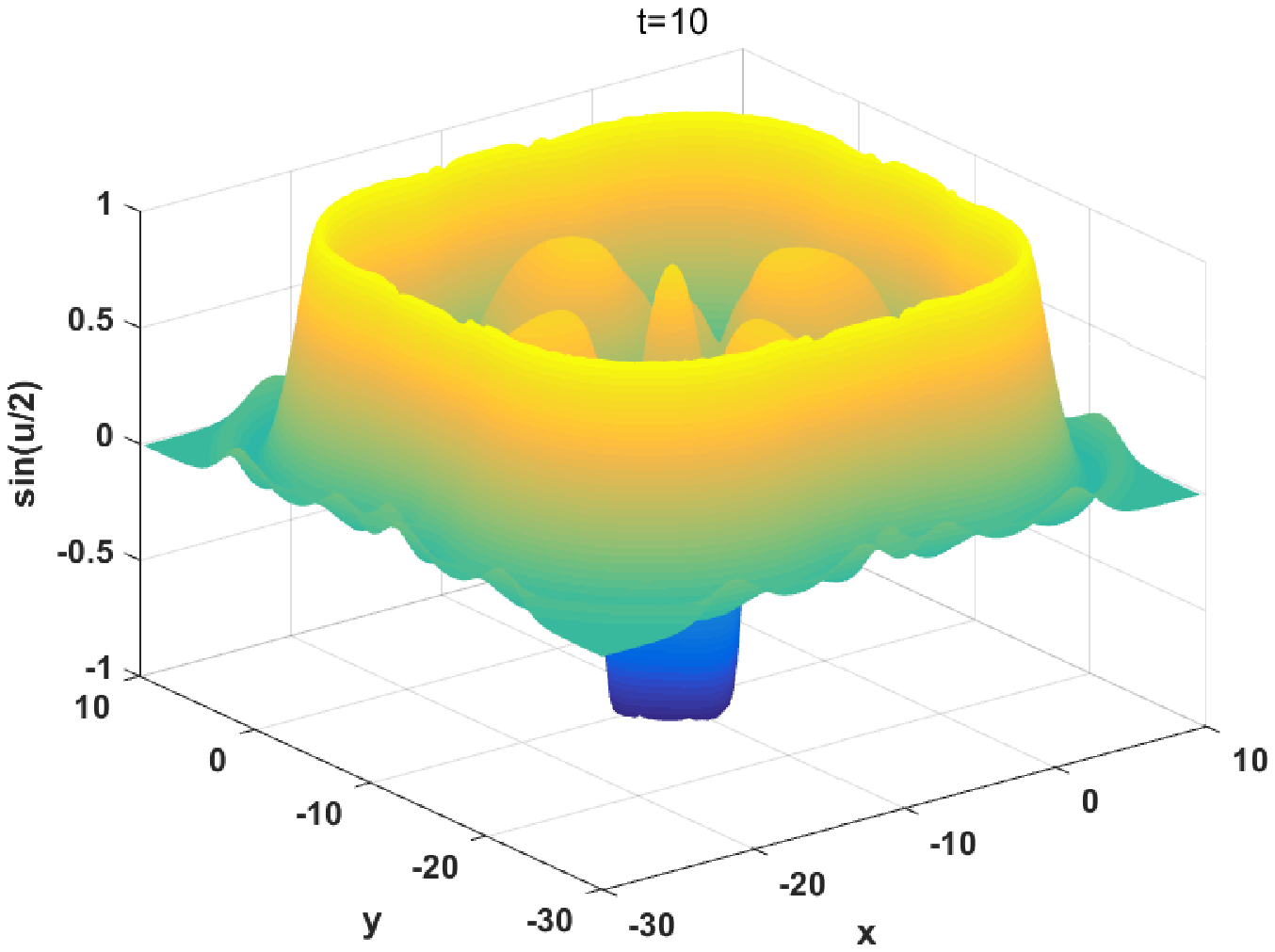}
\caption{Collision of four circular solitons: the initial condition and numerical solutions of \textbf{ESAV-CN} at different times with $N_x=N_y=128$ and $\tau=0.01$.}\label{Fig-7}
\end{figure}

\begin{figure}[H]
\centering
\includegraphics[width=0.42\linewidth]{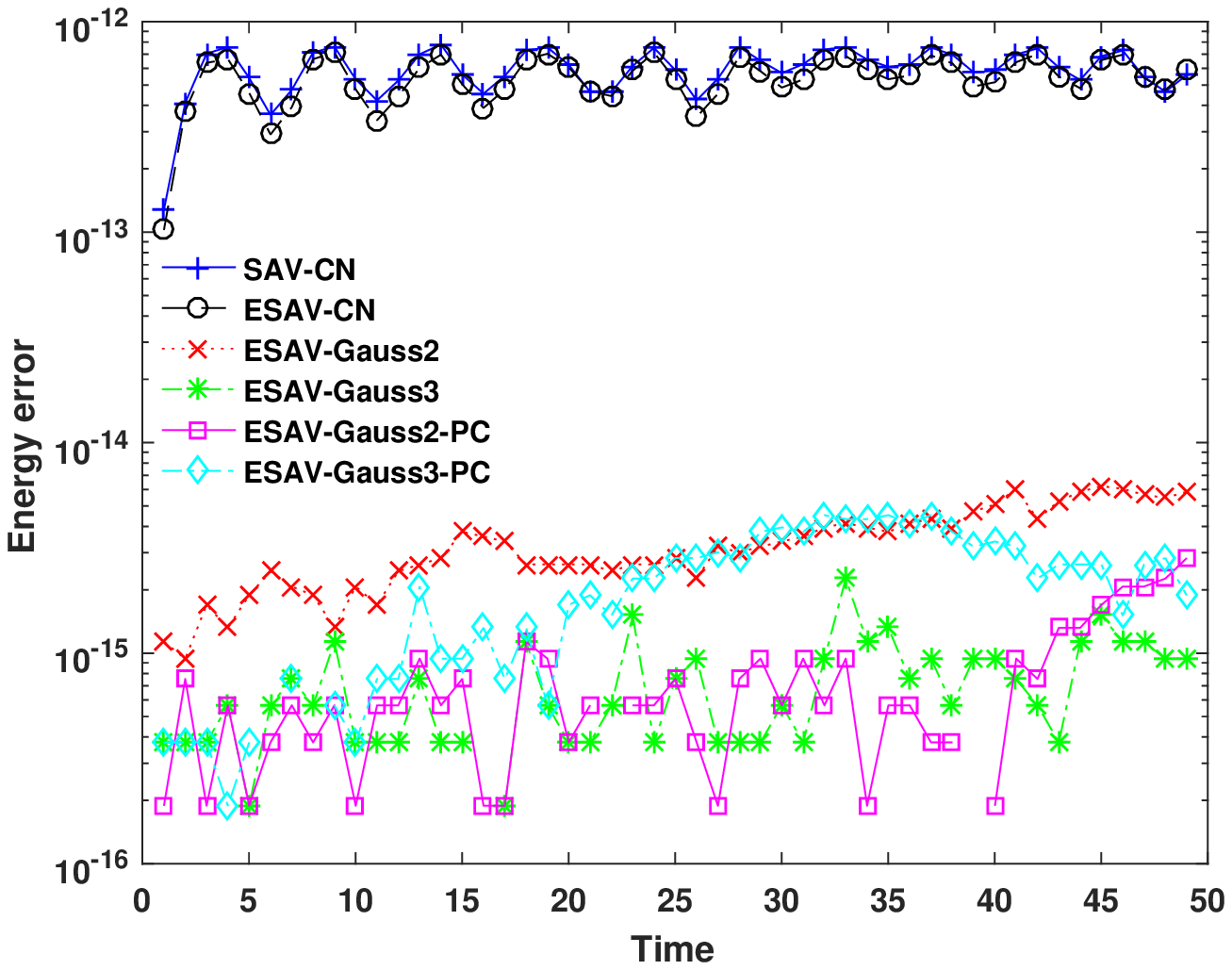}\hspace{-3.5mm}
\includegraphics[width=0.42\linewidth]{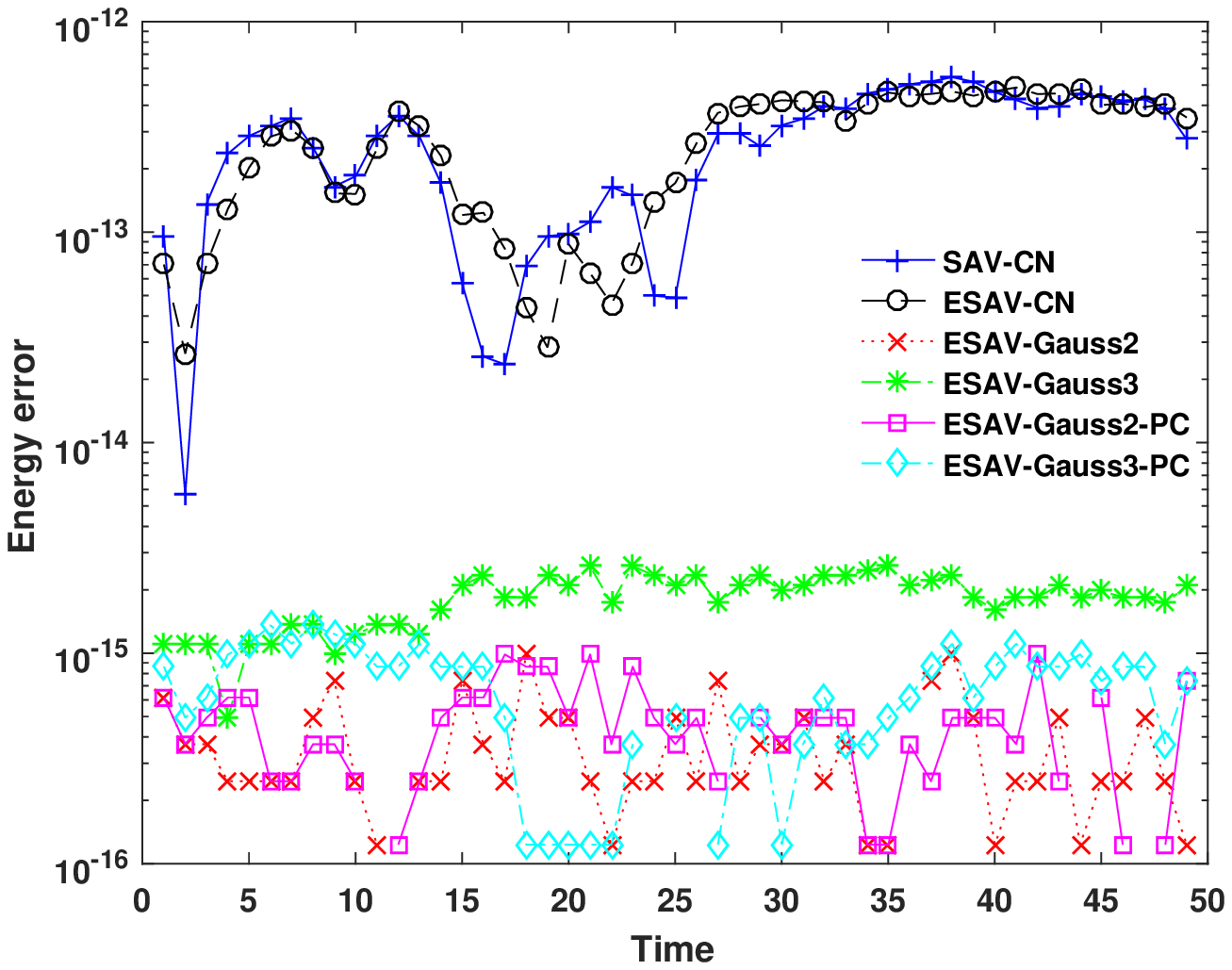}
\caption{Long-time energy errors of the proposed schemes for the circular ring soliton~(left)~and collision of four circular solitons~(right)~with $N_x=N_y=128$ and $\tau=0.01$ until $T=50$.}\label{Fig-8}
\end{figure}

\noindent \textbf{Example 5.4.}~(The KdV equation) This example is obtained from the KdV equation
\begin{equation}\label{eq-5-8}
u_t+\alpha u_{xxx}+\beta uu_x=0,\quad (x,y)\in\Omega\in\mathcal{R}^1,\ t\in (0,T]
\end{equation} 
with $\beta=1$ and periodic boundary conditions. The initial condition is calculated from the following exact solution of this problem by setting $t=0$. The two different cases are as follows:
\begin{itemize}
\item one-soliton wave solution \cite{luigi-19-HBVM-KDV-JCAM}~(see Fig. \ref{Fig-9}): 
\begin{equation}\label{eq-5-9}
u(x,t)= 3\gamma\bigg[\mbox{sech}\left(\sqrt{\frac{\gamma}{4\alpha}}(x-\gamma t)_\Omega\right)\bigg]^2, \quad \Omega=[x_R,x_L],~\alpha=0.0013020833,~\gamma=\frac{1}{3},
\end{equation}
where the notation
\begin{equation}\label{eq-5-10}
(\theta)_\Omega=
\left\{\aligned
&x_L-\mbox{Rem}\left(x_L-\theta,x_L-x_R\right),\quad \mbox{if}\ \theta<x_R,\\
&\theta,\quad\quad\quad\quad\quad\quad\quad\quad\quad\quad\quad\quad\quad\mbox{if}\ \theta\in\Omega,\\
&x_R+\mbox{Rem}\left(\theta-x_R,x_L-x_R\right), \quad\mbox{if}\ \theta>x_L.
\endaligned\right.
\end{equation}
Here, Rem denotes the remainder of the integer division of two parameters. The spatial region $\Omega=[-3,5]$ is numerically solved by the proposed methods.
\item two-soliton waves solution \cite{luigi-19-HBVM-KDV-JCAM}~(see Fig. \ref{Fig-10}): 
\begin{equation}\label{eq-5-11}
u(x,t)=12\frac{k_1^2e^{\xi_1}+k_2^2e^{\xi_2}+2(k_2-k_1)^2e^{\xi_1+\xi_2}+\rho^2(k_2^2e^{\xi_1}+k_1^2e^{\xi_2})e^{\xi_1+\xi_2}}{\left(1+e^{\xi_1}+e^{\xi_2}+\rho^2e^{\xi_1+\xi_2}\right)^2},\quad \alpha=1,
\end{equation}
where we choose the parameter 
\begin{align}\label{eq-5-12}
\aligned
&k_1=0.4,\ k_2=0.6,\ \rho=(k_1-k_2)/(k_1+k_2)=-0.2,\\
&\xi_1=k_1x-k_1^3t+4,\ \xi_2=k_2x-k_2^3t+15
\endaligned
\end{align}
and take $\Omega=[-40,40]$ as a periodic region.
\end{itemize}

The one-soliton wave \eqref{eq-5-9} is periodic in time \cite{luigi-19-HBVM-KDV-JCAM}, and the period is the time interval $T=24$ we calculated. As is plotted in Fig. \ref{Fig-9}, the waveform of the one-soliton is well conserved, and the numerical solution is very consistent with the the exact solution \eqref{eq-5-9}. Similar numerical behaviors can also be obtained from the simulation of two-soliton waves \eqref{eq-5-11}. In particular, in Fig. \ref{Fig-10}, the two-soliton waves that include a taller one and a lower one gradually approach one another slowly with moving towards right. When $t=80$, they collide and continue moving away from each other. In Fig. \ref{Fig-11}, the numerical Hamiltonian turns out to be preserved by all the schemes. Therefore, this also makes clear that the collision of the two-soliton waves is approximated to machine precision.
\begin{figure}[H]
\centering
\includegraphics[width=0.42\linewidth]{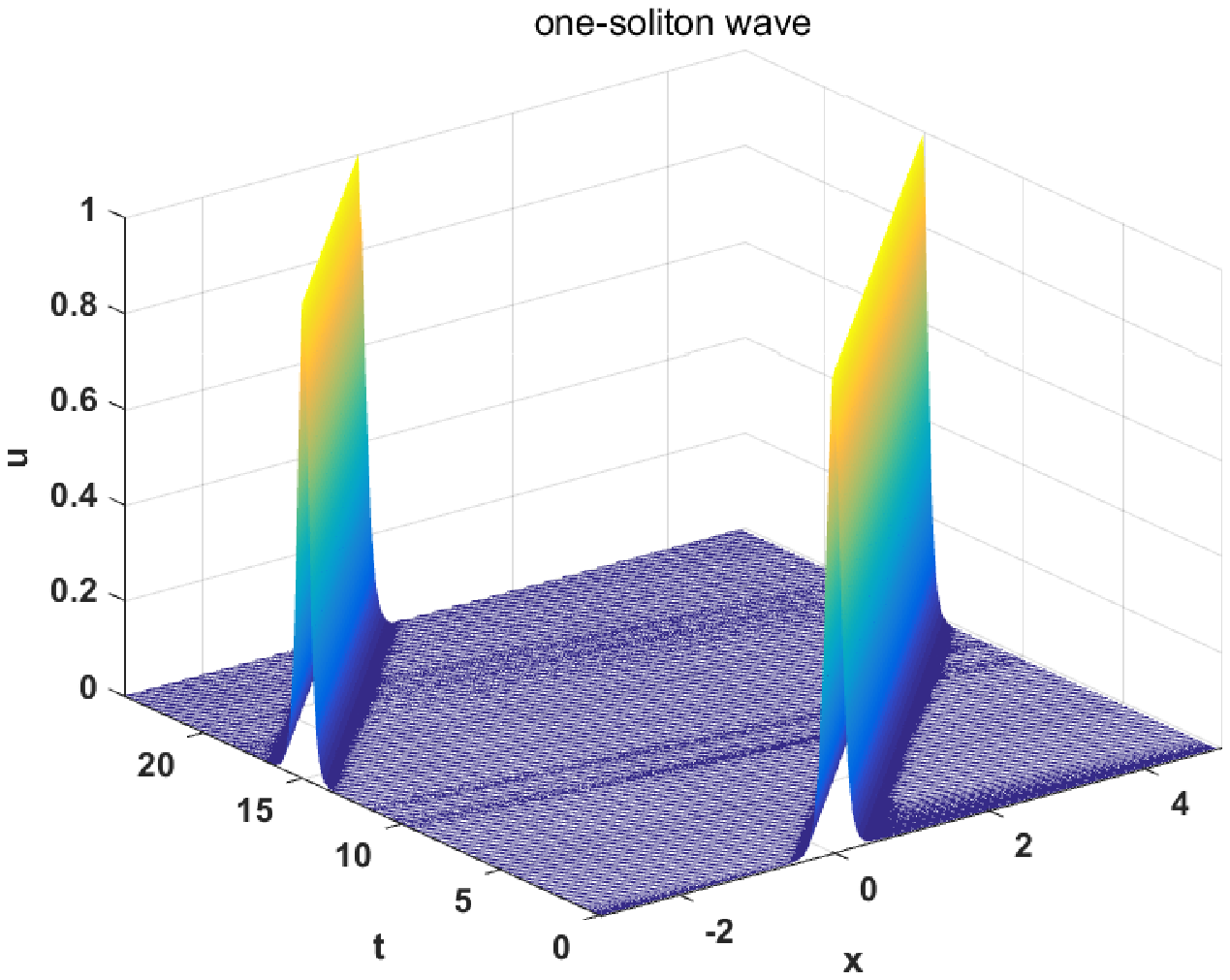}\hspace{-3.5mm}
\includegraphics[width=0.42\linewidth]{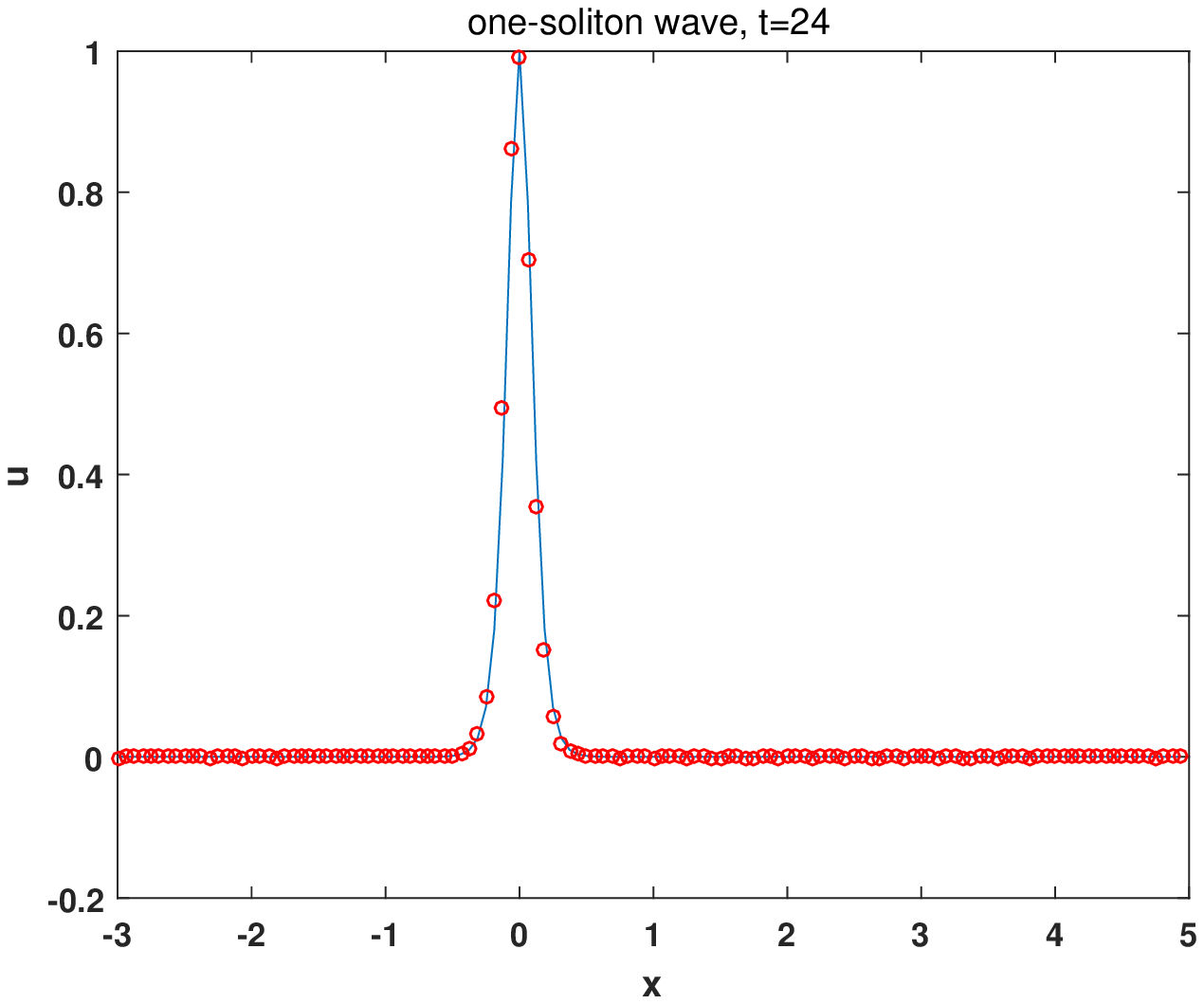}
\caption{The one-soliton wave of \textbf{ESAV-CN} with $N_x=N_y=128$ and $\tau=0.01$. The blue solid line is the exact solution, and the red circle is the numerical solutions.}\label{Fig-9}
\end{figure}

\begin{figure}[H]
\centering
\includegraphics[width=0.32\linewidth]{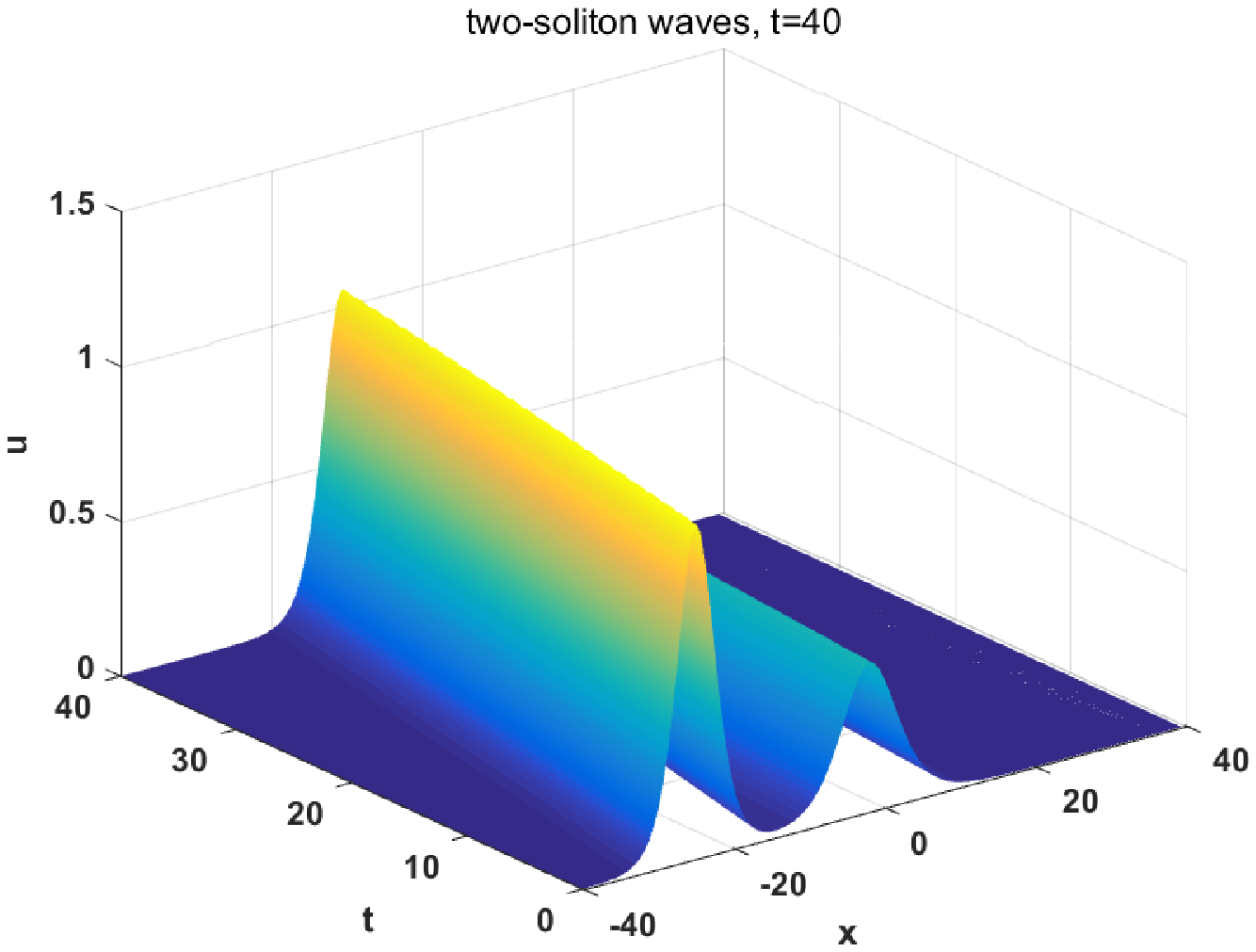}\hspace{-3.5mm}
\includegraphics[width=0.32\linewidth]{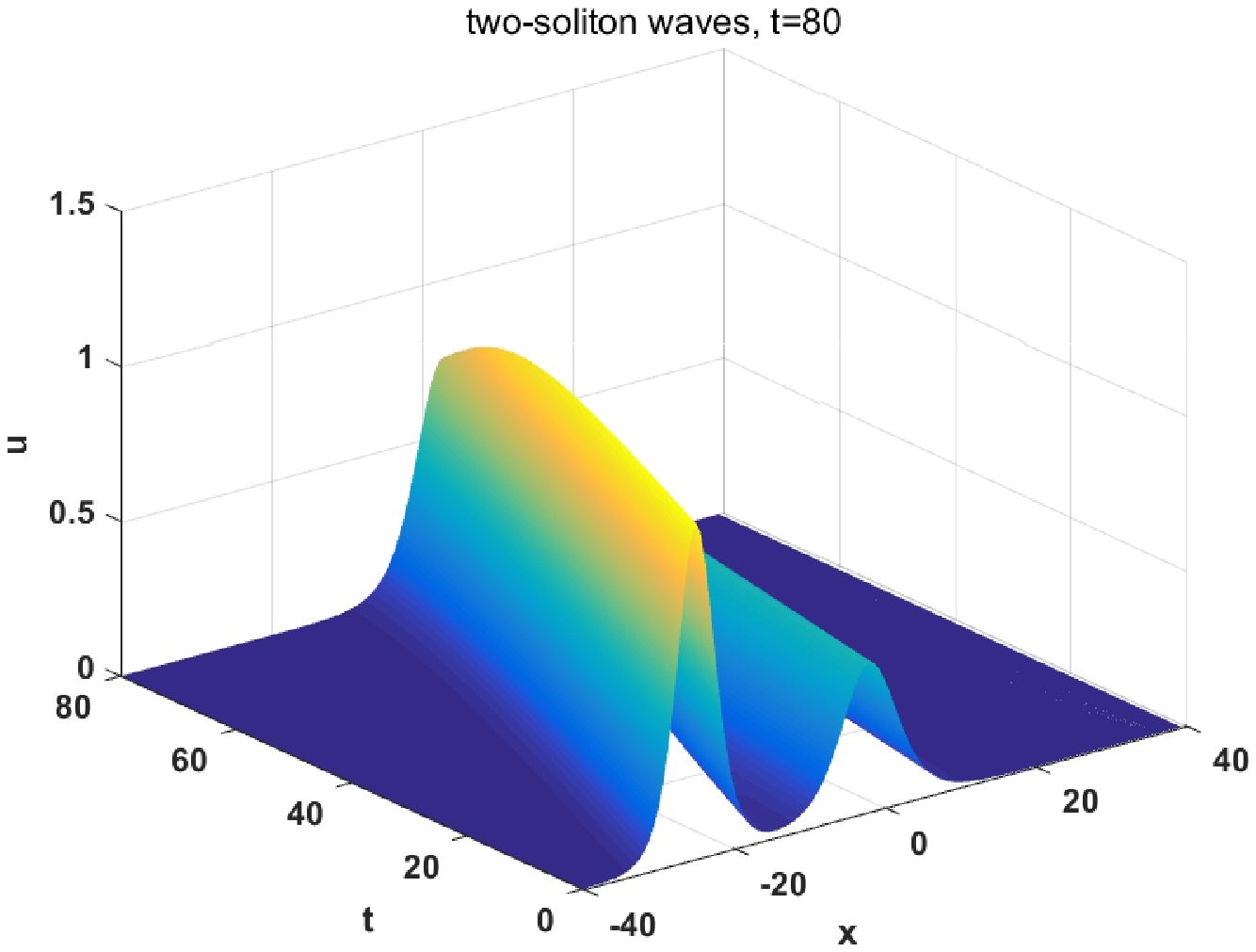}\hspace{-3.5mm}
\includegraphics[width=0.32\linewidth]{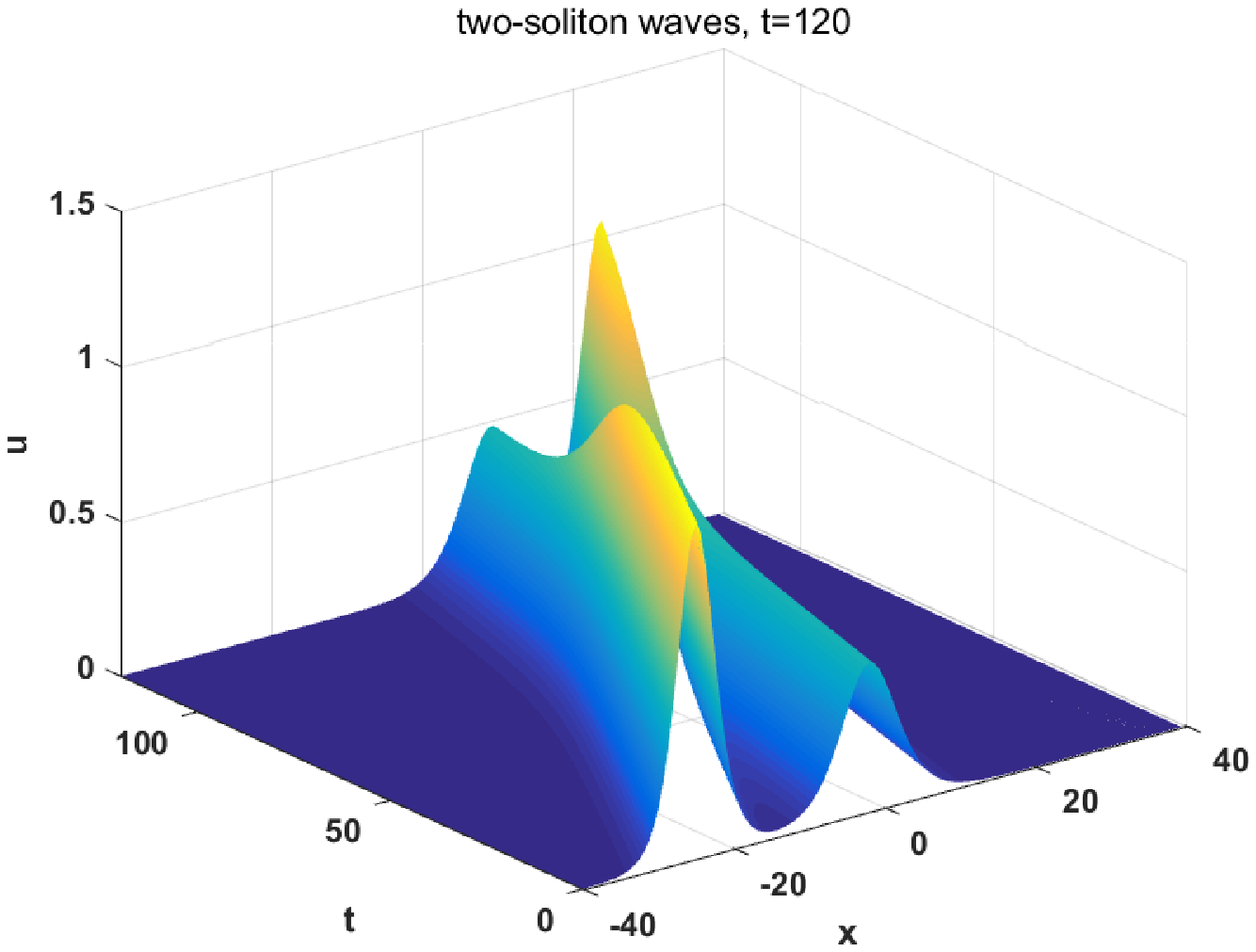}\\
\includegraphics[width=0.32\linewidth]{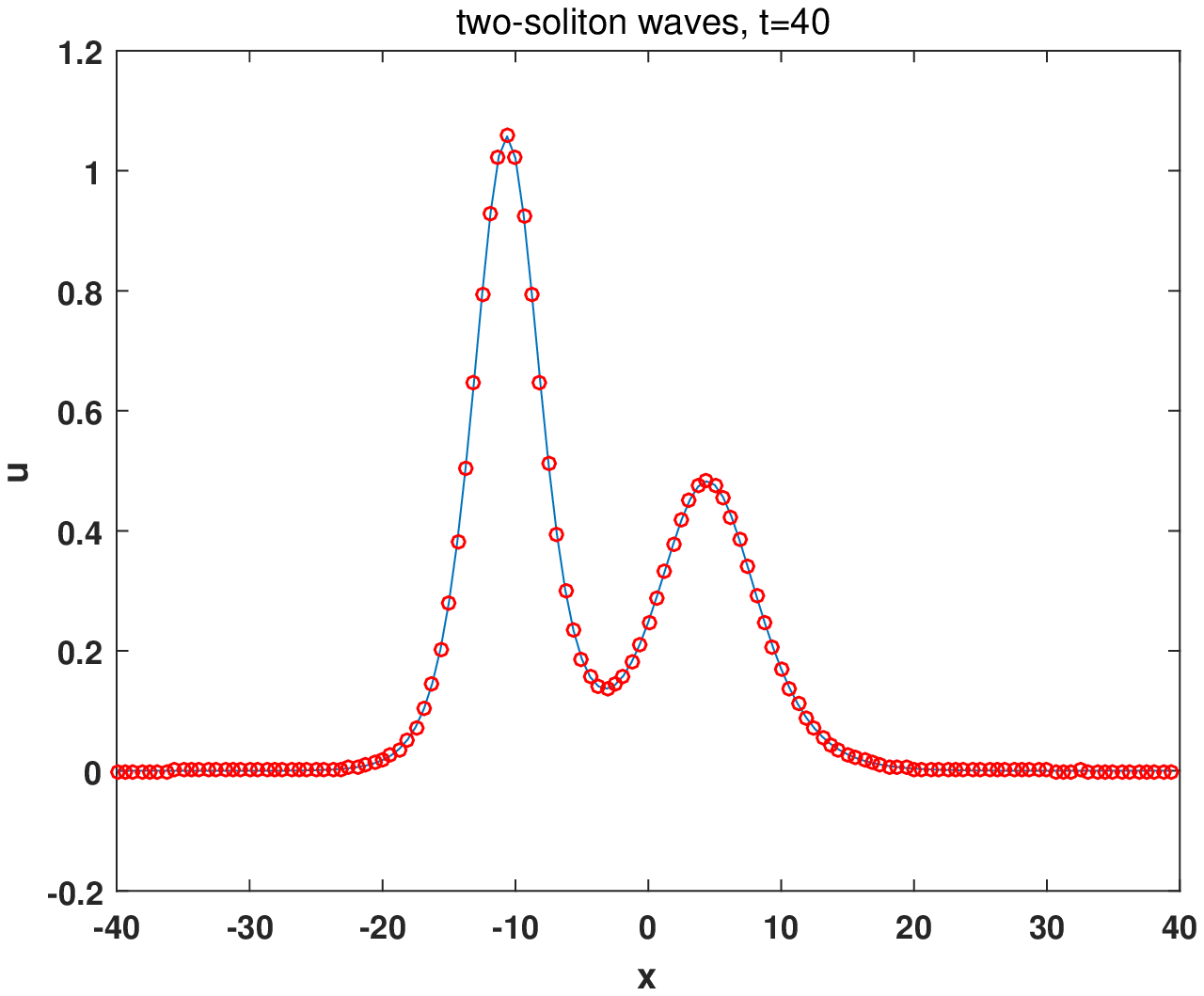}\hspace{-3.5mm}
\includegraphics[width=0.32\linewidth]{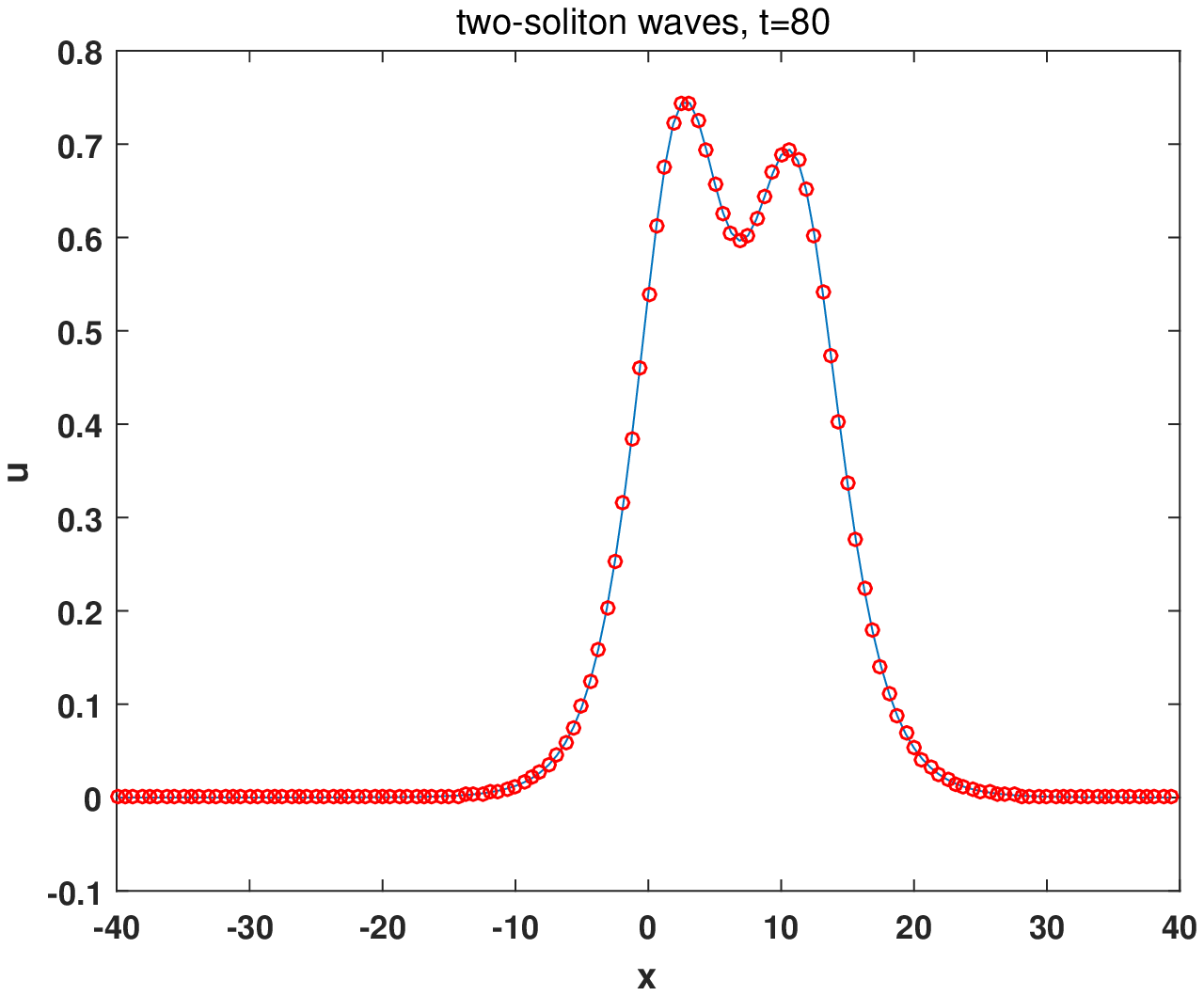}\hspace{-3.5mm}
\includegraphics[width=0.32\linewidth]{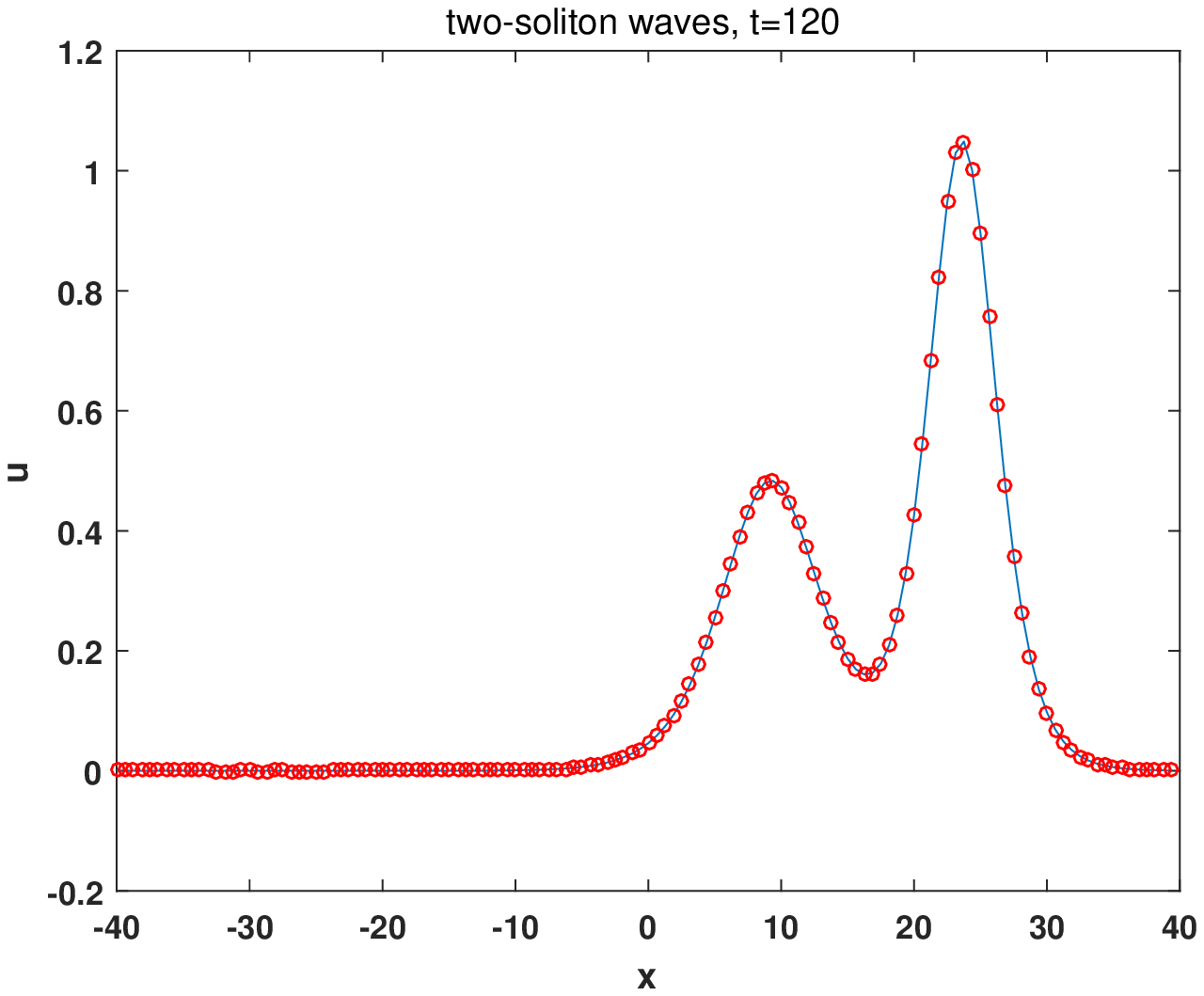}	
\caption{The two-soliton waves of \textbf{ESAV-CN} with $N_x=N_y=128$ and $\tau=0.01$. The blue solid line is the exact solution, and the red circle is the numerical solutions.}\label{Fig-10}
\end{figure}

\begin{figure}[H]
\centering
\includegraphics[width=0.42\linewidth]{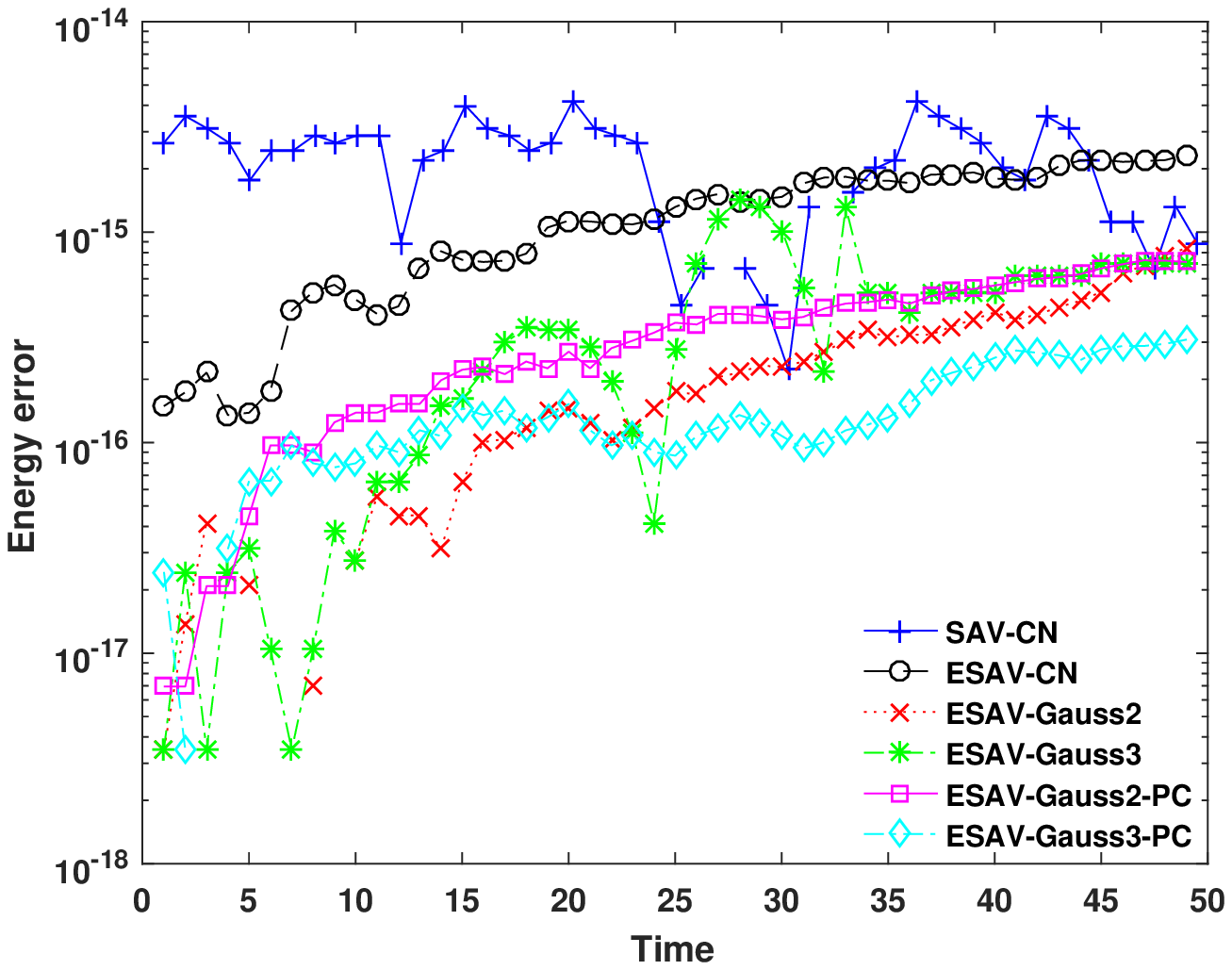}\hspace{-3.5mm}
\includegraphics[width=0.42\linewidth]{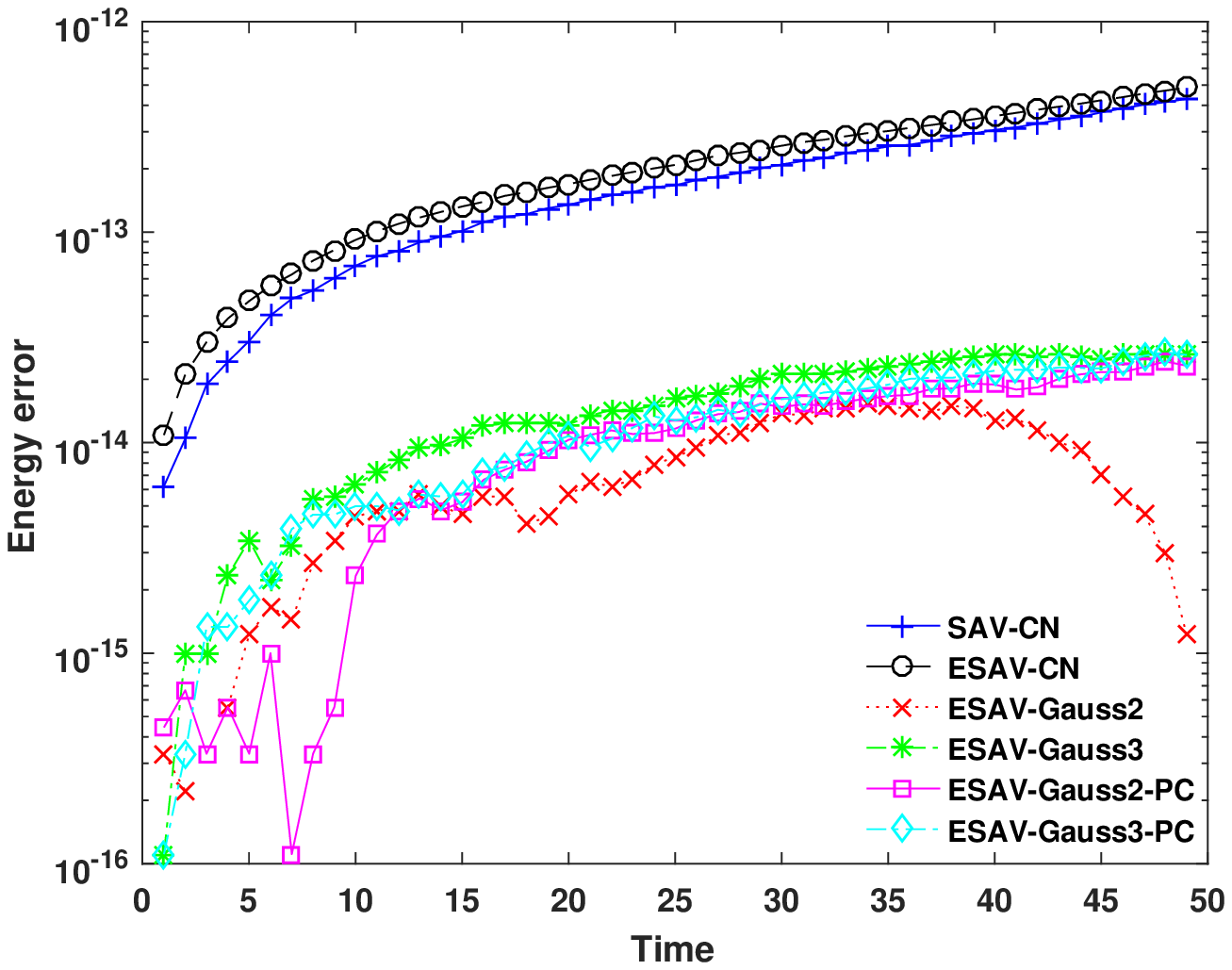}
\caption{Long-time energy errors of the proposed schemes for the one-soliton~(left)~and the two-soliton waves~(right)~with $N_x=N_y=128$ and $\tau=0.01$ until $T=50$.}\label{Fig-11}
\end{figure}

\section{Concluding remarks}\label{Sec-6}
In this paper, we systematically construct linearly implicit energy-preserving schemes with arbitrary order for Hamiltonian PDEs. The basic idea is to combine the newly proposed ESAV approach and the symplectic RK method, as well as an extrapolation strategy. The solution variables and the introduced auxiliary variable are totally decoupled in the implementation, and no extra inner products have to be added. When the pseudo-spectral method is used for spatial discretization, the resulting ESAV schemes are completely explicit so that the practical computation is more efficient than the classical SAV schemes. Numerical experiments are provided for three specific Hamiltonian PDEs to demonstrate the superior behaviors. 

\section*{Acknowledgements}
\noindent This work is supported by the National Key Research and Development Project of China (Grant No. 2018YFC1504205), the National Natural Science Foundation of China (Grant No. 11771213, 11971242), the Major Projects of Natural Sciences of University in Jiangsu Province of China (Grant No. 18KJA110003) and the Priority Academic Program Development of Jiangsu Higher Education Institutions.

\bibliographystyle{abbrv}
\bibliography{references}

\begin{thebibliography}{10}

\bibitem{akrivis-19-RK-SAV-SIAMJSC}
G.~Akrivis, B.~Li, and D.~Li.
\newblock Energy-decaying extrapolated {RK-SAV} methods for the {Allen-Cahn}
  and {Cahn-Hilliard} equations.
\newblock {\em SIAM J. Sci. Comput.}, 41:A3703--A3727, 2019.

\bibitem{betsch-00-time-FEM-JCP}
P.~Betsch and P.~Steinmann.
\newblock Inherently energy conserving time finite elements for classical
  mechanics.
\newblock {\em J. Comput. Phys.}, 160:88--116, 2000.

\bibitem{luigi-12-Poisson-JCAM}
L.~Brugnano, M.~Calvo, J.~Montijano, and L.~R\'andez.
\newblock Energy-preserving methods for {Poisson} systems.
\newblock {\em J. Comput. Appl. Math.}, 236:3890--3904, 2012.

\bibitem{luigi-19-HBVM-KDV-JCAM}
L.~Brugnano, G.~Gurioli, and Y.~Sun.
\newblock Energy-conserving {Hamiltonian} boundary value methods for the
  numerical solution of the {Korteweg-de Vries} equation.
\newblock {\em J. Comput. Appl. Math.}, 351:117--135, 2019.

\bibitem{luigi-10-HBVM-JNAIAM}
L.~Brugnano, F.~Iavernaro, and D.~Trigiante.
\newblock Hamiltonian boundary value methods (energy preserving discrete line
  integral methods).
\newblock {\em J. Numer. Anal. Ind. Appl. Math.}, 5:17--37, 2010.

\bibitem{cai-20-linear-MS-JCP}
J.~Cai and J.~Shen.
\newblock Two classes of linearly implicit local energy-preserving approach for
  general multi-symplectic {Hamiltonian PDEs}.
\newblock {\em J. Comput. Phys.}, 401:108975, 2020.

\bibitem{cai-19-SG-NB-JCP}
W.~Cai, C.~Jiang, Y.~Wang, and Y.~Song.
\newblock Structure-preserving algorithms for the two-dimensional {sine-Gordon}
  equation with {Neumann} boundary conditions.
\newblock {\em J. Comput. Phys.}, 395:166--185, 2019.

\bibitem{cai-18-PAVF-JCP}
W.~Cai, H.~Li, and Y.~Wang.
\newblock Partitioned averaged vector field methods.
\newblock {\em J. Comput. Phys.}, 370:25--42, 2018.

\bibitem{cohen-11-lin-poisson-BIT}
D.~Cohen and E.~Hairer.
\newblock Linear energy-preserving integrators for {Poisson} systems.
\newblock {\em BIT}, 51:91--101, 2011.

\bibitem{cooper-87-RK-IMA}
G.~Cooper.
\newblock Stability of {Runge-Kutta} methods for trajectory problems.
\newblock {\em IMA J. Numer. Anal.}, 7:1--13, 1987.

\bibitem{dahlby-11-general-IP-SIAMJSC}
M.~Dahlby and B.~Owren.
\newblock A general framework for deriving integral preserving numerical
  methods for {PDEs}.
\newblock {\em SIAM J. Sci. Comput.}, 33:2318--2340, 2011.

\bibitem{furihata-99-DVD-JCP}
D.~Furihata.
\newblock Finite difference schemes for $\partial u/\partial
  t=(\partial/\partial x)^\alpha\delta {G}/\delta u$ that inherit energy
  conservation or dissipation property.
\newblock {\em J. Comput. Phys.}, 156:181--205, 1999.

\bibitem{furihata-11-DVD-CHCRC}
D.~Furihata and T.~Matsuo.
\newblock {\em Discrete Variational Derivative Method: A Structure-Preserving
  Numerical Method for Partial Differential Equations}.
\newblock Chapman \& Hall/CRC, Boca Raton, FL, USA, 2011.

\bibitem{gong-17-FP-NLS-JCP}
Y.~Gong, Q.~Wang, Y.~Wang, and J.~Cai.
\newblock A conservative {Fourier} pseudo-spectral method for the nonlinear
  {Schr\"{o}dinger} equation.
\newblock {\em J. Comput. Phys.}, 328:354--370, 2017.

\bibitem{gong-20-high-stable-JCP}
Y.~Gong, J.~Zhao, and Q.~Wang.
\newblock Arbitrarily high-order linear energy stable schemes for gradient flow
  models.
\newblock {\em J. Comput. Phys.}, 419:109610, 2020.

\bibitem{gong-18-IEQ-binary-fluid-SIAMJSC}
Y.~Gong, J.~Zhao, X.~Yang, and Q.~Wang.
\newblock Fully discrete second-order linear schemes for hydrodynamic phase
  field models of binary viscous fluid flows with variable densities.
\newblock {\em SIAM J. Sci. Comput.}, 40:B138--B167, 2018.

\bibitem{gonzalez-96-DH-JNS}
O.~Gonzalez.
\newblock Time integration and discrete {Hamiltonian} systems.
\newblock {\em J. Nonlinear Sci.}, 6:449--467, 1996.

\bibitem{hairer-10-coll-JNAIAM}
E.~Hairer.
\newblock Energy-preserving variant of collocation methods.
\newblock {\em J. Numer. Anal. Ind. Appl. Math.}, 5:73--84, 2010.

\bibitem{hairer-06-GNI-ODE}
E.~Hairer, C.~Lubich, and G.~Wanner.
\newblock {\em Geometric Numerical Integration: Structure-Preserving Algorithms
  for Ordinary Differential Equations}.
\newblock Springer-Verlag, Berlin, 2nd edition, 2006.

\bibitem{harten-83-DG-SIAM-REV}
A.~Harten, P.~Lax, and B.~van Leer.
\newblock On upstream differencing and {Godunov}-type schemes for hyperbolic
  conservation law.
\newblock {\em SIAM Rev.}, 25:35--61, 1983.

\bibitem{itoh-88-VDQ-JCP}
T.~Itoh and K.~Abe.
\newblock Hamiltonian-conserving discrete canonical equations based on
  variational difference quotients.
\newblock {\em J. Comput. Phys.}, 77:85--102, 1988.

\bibitem{jiang-19-SG-IEQ-JSC}
C.~Jiang, W.~Cai, and Y.~Wang.
\newblock A linearly implicit and local energy-preserving scheme for the
  {sine-Gordon} equation based on the invariant energy quadratization approach.
\newblock {\em J. Sci. Comput.}, 80:1629--1655, 2019.

\bibitem{jiang-20-CH-SAV-JSC}
C.~Jiang, Y.~Gong, W.~Cai, and Y.~Wang.
\newblock A linearly implicit structure-preserving scheme for the
  {Camassa-Holm} equation based on multiple scalar auxiliary variables
  approach.
\newblock {\em J. Sci. Comput.}, 83:20, 2020.

\bibitem{leimkuhler-04-SHD}
B.~Leimkuhler and S.~Reich.
\newblock {\em Simulating Hamiltonian Dynamics}.
\newblock Cambridge University Press, Cambridge, 2004.

\bibitem{Li-16-AVF-JCM}
H.~Li, Y.~Wang, and M.~Qin.
\newblock A sixth order averaged vector field method.
\newblock {\em J. Comput. Math.}, 34:479--498, 2016.

\bibitem{liu-20-ESAV-SIAMJSC}
Z.~Liu and X.~Li.
\newblock The exponential scalar auxiliary variable {(E-SAV)} approach for
  phase field models and its explicit computing.
\newblock {\em SIAM J. Sci. Comput.}, 42:B630--B655, 2020.

\bibitem{matsuo-01-DC-JCP}
T.~Matsuo and D.~Furihata.
\newblock Dissipative or conservative finite-difference schemes for
  complex-valued nonlinear partial differential equations.
\newblock {\em J. Comput. Phys.}, 171:425--447, 2001.

\bibitem{mcLachlan-99-DG-PTRSLA}
R.~McLachlan, G.~Quispel, and N.~Robidoux.
\newblock Geometric integration using discrete gradients.
\newblock {\em Phil. Trans. R. Soc. Lond. A}, 357:1021--1045, 1999.

\bibitem{qiao-19-PR-SAV-CICP}
Z.~Qiao, S.~Sun, T.~Zhang, and Y.~Zhang.
\newblock A new multi-component diffuse interface model with {Peng-Robinson}
  equation of state and its scalar auxiliary variable (sav) approach.
\newblock {\em Commun. Comput. Phys.}, 26:1597--1616, 2019.

\bibitem{quispel-08-AVF-JPAMT}
G.~Quispel and D.~McLaren.
\newblock A new class of energy-preserving numerical integration methods.
\newblock {\em J. Phys. A: Math. Theor.}, 41:045206, 2008.

\bibitem{shen-11-spectral}
J.~Shen, T.~Tang, and L.-L. Wang.
\newblock {\em Spectral Methods: Algorithms, Analysis and Applications}.
\newblock Springer-Verlag, Berlin, Heidelberg, 2011.

\bibitem{shen-18-SAV-JCP}
J.~Shen, J.~Xu, and J.~Yang.
\newblock The scalar auxiliary variable {(SAV)} approach for gradient flows.
\newblock {\em J. Comput. Phys.}, 353:407--416, 2018.

\bibitem{shen-19-SIAM-REV}
J.~Shen, J.~Xu, and J.~Yang.
\newblock A new class of efficient and robust energy stable schemes for
  gradient flows.
\newblock {\em SIAM Rev.}, 61:474--506, 2019.

\bibitem{tang-12-time-FEM-AMC}
W.~Tang and Y.~Sun.
\newblock Time finite element methods: a unified framework for numerical
  discretizations of {ODEs}.
\newblock {\em Appl. Math. Comput.}, 219:2158--2179, 2012.

\bibitem{wu-13-oscil-JCP}
X.~Wu, B.~Wang, and W.~Shi.
\newblock Efficient energy-preserving integrators for oscillatory {Hamiltonian}
  systems.
\newblock {\em J. Comput. Phys.}, 235:587--605, 2013.

\bibitem{yang-16-IEQ-JCP}
X.~Yang.
\newblock Linear, first and second-order, unconditionally energy stable
  numerical schemes for the phase field model of homopolymer blends.
\newblock {\em J. Comput. Phys.}, 327:294--316, 2016.

\bibitem{yang-17-IEQ-JCP}
X.~Yang, J.~Zhao, and Q.~Wang.
\newblock Numerical approximations for the molecular beam epitaxial growth
  model based on the invariant energy quadratization method.
\newblock {\em J. Comput. Phys.}, 333:104--127, 2017.

\end{thebibliography}
\end{document}